\documentclass[graybox]{svmult}

\usepackage{type1cm}        
\usepackage{makeidx}         
\usepackage{graphicx}        
\usepackage{multicol}        
\usepackage[bottom]{footmisc}

\usepackage{caption}
\usepackage{subcaption}

\usepackage{newtxtext}       %
\usepackage[varvw]{newtxmath}       

\usepackage{amsfonts}
\usepackage{graphicx}
\usepackage{enumerate}
\usepackage[ansinew]{inputenc}
\usepackage{hyperref}
\usepackage{multirow}
\usepackage{color}
\usepackage{colortbl}   

\usepackage{wrapfig}
\usepackage{floatflt,epsfig} 
\usepackage{fancyhdr}
\usepackage{array}
\usepackage[right]{eurosym}

\def\Hil{\mathcal{H}}

\def\RR{\mathbb{R}}

\def\CC{\mathbb{C}}

\def\be{\begin{equation}}
\def\ee{\end{equation}}
\def\ben{\begin{eqnarray}}
\def\een{\end{eqnarray}}

\newcommand{\range}[1]{\mathsf{ran}\left( #1 \right)} 

\definecolor{darkviolet}{rgb}{0.58,0,0.83} 

\definecolor{ballblue}{rgb}{0.13, 0.67, 0.8}

\newcommand{\prode}[1]{\left\langle#1\right\rangle}

\newcommand{\set}[1]{\left\{#1\right\}}

\newcommand{\paren}[1]{\left(#1\right)}

\def\I{\mathcal{I}_{\Hil}}

\makeindex             

\begin{document}

\title*{A Survey of Fusion Frames in Hilbert Spaces}
\author{L. K\"ohldorfer, P. Balazs, P. Casazza, S. Heineken, C. Hollomey, P. Morillas, M. Shamsabadi}
\institute{L. K\"ohldorfer \at Acoustics Research Institute, Austrian Academy of Sciences, \email{lukas.koehldorfer@oeaw.ac.at}
\and P. Balazs \at Acoustics Research Institute, Austrian Academy of Sciences, \email{peter.balazs@oeaw.ac.at}
\and P. Casazza \at University Of Missouri, Dept. of Mathematics \email{casazzapeter40@gmail.com}
\and S. Heineken \at Instituto de Investigaciones Matem\'aticas Luis A. Santal\'o  (UBA-CONICET), \email{sheinek@dm.uba.ar}
\and C. Hollomey \at Acoustics Research Institute, Austrian Academy of Sciences, \email{clara.hollomey@oeaw.ac.at}
\and P. Morillas \at Instituto de Matem\'{a}tica Aplicada San Luis (UNSL-CONICET),  \email{morillas.unsl@gmail.com}
\and M. Shamsabadi \at Acoustics Research Institute, Austrian Academy of Sciences, \email{mitra.shamsabadi@oeaw.ac.at}}
%
%
\maketitle

\abstract*{
Fusion frames are a very active area of research today because of their myriad of applications in pure mathematics, applied mathematics, engineering, medicine, signal and image processing and much more. Fusion frames - a generalization - provide much greater flexibility for designing sets of vectors for applications. 
They are therefore prominent in all the mentioned areas, including e.g. mitigating the effects of noise in a signal and giving robustness to erasures. In this chapter, we present the fundamentals of fusion frame theory with an emphasis on their delicate relation to frame theory. Fusion frame theory is such a broad subject today, it is difficult to get started on it. The goal here is to provide an easy entry into this topic for researchers and students.
Proofs for fusion frames will be self-contained and the differences between frames and fusion frames are analyzed. Therefore, we focus on the particularities of fusion frame duality. We also provide a reproducible research implementation.}

\abstract{
Fusion frames are a very active area of research today because of their myriad of applications in pure mathematics, applied mathematics, engineering, medicine, signal and image processing and much more. They provide a great flexibility for designing sets of vectors for applications and are therefore prominent in all these areas, including e.g. mitigating the effects of noise in a signal or giving robustness to erasures. In this chapter, we present the fundamentals of fusion frame theory with an emphasis on their delicate relation to frame theory. 
The goal here is to provide researchers and students with an easy entry into this topic. Proofs for fusion frames will be self-contained and differences between frames and fusion frames are analyzed. In particular, we focus on the subtleties of fusion frame duality.
We also provide a reproducible research implementation.}

\section{Introduction}\label{Introduction}

Hilbert space frames $(\varphi_i)_{i\in I}$ are (possibly redundant) sequences of vectors in a Hilbert space satisfying a Parseval inequality
(see Definition \ref{framenotions}).
This is a generalization of orthogonal expansions. This subject was introduced in the setting of nonharmonic Fourier series in 1952 in
\cite{duffschaef1}. The topic then became dormant for 30 years until it was brought back to life in \cite{daubgromay86} in the setting
of data processing. After this, the subject advanced rapidly and soon became one
of the most active areas of research in applied mathematics. Frames were originally used in signal and image processing, 
and later in sampling
theory, data compression, time-frequency analysis, coding theory,
and Fourier series. Today, there are ever increasing applications of frames to problems in pure and applied mathematics,
computer science, engineering, medicine, and physics with new applications arising regularly (see \cite{CK2,ole1,W} and the "Further Study" section). 
Redundancy plays a fundamental role in applications of
frame theory. First, it gives much greater flexibility in the design of families of vectors for a particular problem not possible for linearly independent sets of
vectors. In particular, the advantages of redundancy was shown in phase reconstruction, where we need to determine a vector from the absolute value of its frame coefficients
\cite{BCE,BBCE}. Second, we can
mitigate the effects of noise in a signal  and get robustness to erasures by spreading information over a wider range of vectors. 
For an introduction to frame theory, we recommend \cite{CK2,ole1,W}. 

One can think of a frame as a family of weighted one-dimensional subspaces of a Hilbert space.
As a consequence, frame theory has its limitations. In particular, they do not work well in distributed processing where one has to project a signal onto
multidimensional subspaces. This requires fusion frames for the applications which include 
wireless sensor networks \cite{IB}, visual and hearing systems \cite{dist3}, geophones in geophysics \cite{CG}, distributed processing \cite{cakuli08}, packet encoding 
\cite{B,BKP,CKLR}, parallel
processing \cite{CK3}, and much more. Fusion frames also arise in various theoretical applications including the Kadison-Singer Problem
\cite{CT2}, and optimal subspace packings \cite{CHST,CHST2}.
All of these problems require working with a countable family of subspaces of a Hilbert space with positive weights and satisfying a type of
Parseval inequality (see Definition \ref{fusionframedef}). Fusion frames were introduced in \cite{caskut04} where they were called "frames of subspaces". In a followup 
paper \cite{cakuli08}, the name was changed to "fusion frames" so they would not be confused with "frames in subspaces". Because of their many applications,
fusion frame theory has developed very rapidly over the years (see also \cite{CK2}). 
Here, we give a new survey on this topic.

\subsection{Motivation}
The motivation for this chapter is to present the fundamentals of fusion frame theory with an emphasis on results from frame theory which hold, or do not hold,
for fusion frames. The idea is to help future researchers deal with the delicate differences between frames and fusion frames accurately. As such, proofs for
fusion frames will be as self-contained as possible.  
Let us mention that there are very nice articles \cite{caskut04,cakuli08}, we can recommend unrestricted, still we think that there is room for this survey.

Finite frame and fusion frame theory became more and more important recently. It became apparent that there is a necessity for the implementation of fusion frame approaches. 
There are, of course, algorithms, but we provide a reproducible research approach $-$ integrating the codes in the open source Toolbox LTFAT~\cite{soendxxl10, ltfatnote030}.

Fusion frames have a natural connection to domain decomposition methods, and can be used for the solution of operator equations \cite{oswald09}. 
In \cite{xxlcharshaare18} a direct fusion frame formulation is used for that. Here it can be observed that to use the more general definition of the coefficient spaces $-$ not restricting to the Hilbert direct sum of the given subspaces $-$ has advantages, see below. Here we adopt this approach, also used in \cite{Asg_15}. 

Many concepts of classical frame theory have been generalized to the setting of fusion frames. Having a proper notion of dual
fusion frames permits us to have different reconstruction strategies. Also, one wants from a proper definition to yield duality results similar to those known for classical frames. The duality concept for fusion frames is more elaborate but therefore also more interesting than the standard frame setting. This is why we decided to put some focus on this problem, and summarized results from \cite{hemo14,HeinekenMorillas18,heimoanza14}.

For this we have also chosen to make properties of operators between Hilbert spaces explicit, and present the proofs of some results that were often only considered implicitly. 

So, in summary, we provide a new survey on fusion frame theory, aiming to provide an easier entry point into this theory for mathematicians. 

\subsection{Structure}

This chapter is structured as follows: Section \ref{Introduction} introduces and motivates this chapter. In particular, we present a list of differences between frame and fusion frame theory (which we present throughout the chapter) in the next subsection. Section \ref{Preliminaries} introduces notation, general preliminary results and some basic facts on frame theory, which we apply on many occasions throughout these notes. Section \ref{Bessel fusion sequences and fusion frames} presents the notion of fusion frames, Bessel fusion sequences and their canonically associated frame-related operators. The reader familiar with frame theory will recognise many parallels between fusion frames and frames in this section. We also add a subsection on operators between Hilbert direct sums, which serves as a preparation for later and leads to a better understanding of many operators occurring in this chapter. In Section \ref{Other notions in fusion frame theory}, we present other notions in fusion frame theory including generalizations of orthonormal bases, Riesz bases, exact frames and minimal sequences to the fusion frame setting. In Section \ref{Fusion frames and operators}, we study fusion frames under the action of bounded operators, where we present a number of differences between frame and fusion frame theory. In Section \ref{Duality in Fusion Frame Theory}, we present the theory of dual fusion frames. Duality theory probably provides the greatest contrast between frames and fusion frames. We first show the reader how the concept of dual fusion frame arrived. We study properties, present results and provide examples of this concept. We discuss particular cases and other approaches. In Section \ref{Finite Fusion Frames and Implementation} we focus on fusion frames in finite dimensional Hilbert spaces. In Section \ref{Implementations} we present thoughts about concrete implementations in the LTFAT toolbox \cite{soendxxl10}, and a link to a reproducible research page. 
Finally, in Section \ref{Further Study} we list a collection of topics on fusion frames for further study, which are not presented in this chapter.

\subsection{"There Be Dragons"  - traps to be avoided when going from frame theory to fusion frame theory}

In the following, we give an overview of the differences between properties of frames and properties of fusion frames presented in this chapter.

\begin{itemize} 
\item \emph{Existence of Parseval fusion frames:} It is known that for every $m \geq n \in \mathbb{N}$ there is a $m$-element Parseval frame for any $n$-dimensional Hilbert space \cite{CL} (and the vectors may even
be equal norm). In the fusion frame setting, this is not true (see Remark \ref{rank1fusion}). In particular, applying $S_V^{-1/2}$ to the fusion frame does not yield a Parseval fusion frame in general (see Section \ref{Fusion frames and operators}). Currently, there are no necessary and sufficient conditions known for the existence of Parseval fusion frames for arbitrary dimensions of the subspaces. 
\item \emph{Coefficient space:} In contrast to frame theory, there is no natural coefficient space. It can either be chosen to be the Hilbert direct sum of the defining subspaces, which is the standard approach in fusion frame theory. But then the coefficient space is different for each fusion system, and a naive combination of analysis and synthesis is not possible anymore. Or it can be chosen to be the Hilbert direct sum of copies of the whole Hilbert space $-$ the approach taken in this survey, but then  
the analysis operator associated to a fusion frame $V$ maps $\Hil$ bijectively onto the subspace $\mathcal{K}_V^2$ and not the full coefficient space $\mathcal{K}_{\Hil}^2$ (see Theorem \ref{fusionrieszbasischar}).
%

\item {\em Exact systems are not necessarily Riesz}, see Proposition \ref{exactfusion} or \cite{caskut04}.
\item {\em Applying a bounded surjective operator on a fusion frame does not give a fusion frame,} in general. Even for orthonormal fusion bases this is not true, see Section \ref{Fusion frames and operators}.
\item \emph{Duality:} \begin{itemize}
    \item Duality for fusion frames cannot be defined via $D_W C_V = \mathcal{I}_{\Hil}$ as in the frame case, if one wants fusion frame duality to be compatible with frame duality, see Subsection \ref{S other approaches}.
    \item For the dual fusion frame definition, an extra operator $Q$ has to be inserted between $C_W$ and $D_V$, see Definition \ref{D dual fusion frame} and the extensive discussion in Section \ref{Duality in Fusion Frame Theory}.
    \item Duals of fusion Riesz bases are in general not unique, see Subsection  \ref{Examples of dual fusion frames}.
    \item  The fusion frame operator of the canonical dual can differ from the inverse of the fusion frame operator, see Subsection  \ref{Examples of dual fusion frames}. 
    \item The canonical dual of the canonical dual of a fusion frame $V$ is not $V$ in general, see Subsection  \ref{Examples of dual fusion frames}. 
   
\end{itemize}
For a full treatment of fusion frame duality see Section \ref{Duality in Fusion Frame Theory}.
%
\item  The \emph{frame operator of an orthonormal fusion basis $(V_i)_{i\in I}$ is not necessarily the identity}. This is only the case, if the associated weights are uniformly $1$. 
In fact, the definition of an orthonormal fusion basis is independent of weights, whereas we have $S_{V^1}=\mathcal{I}_{\Hil}$ for $V^1=(V_i,1)_{i\in I}$, see Theorem \ref{chaONB}. 
\end{itemize}

\section{Preliminaries}\label{Preliminaries}

Throughout this notes,  $\mathcal{H}$\index{$\mathcal{H}$} is always a separable Hilbert space and all considered index sets are countable. For a closed subspace $V$ of $\mathcal{H}$, $\pi _V$\index{$\pi _V$} denotes the orthogonal projection onto $V$. The set of positive integers $\lbrace 1, 2, 3, \dots \rbrace$ is denoted by $\mathbb{N}$ and $\delta_{ij}$ denotes the Kronecker-delta. The cardinality of a set $J$ is denoted by $\vert J \vert$. The domain, kernel and range of an operator $T$ is denoted by $\mathrm{dom}(T)$\index{$\mathrm{dom}(T)$}, $\mathcal{N}(T)$\index{$\range{T}$} and  $\mathcal{R}(T)$\index{$\mathcal{R}(T)$}  respectively.  $\mathcal{I}_X$\index{$\mathcal{I}_X$} denotes the identity operator on a given space $X$. The set of bounded operators between two normed spaces $X$ and $Y$ is denoted by $\mathcal{B}(X,Y)$\index{$\mathcal{B}(X,Y)$} and we set $\mathcal{B}(X) := \mathcal{B}(X,X)$. For an operator $U \in \mathcal{B}( \mathcal{H}_1 , \mathcal{H}_2 )$ ($\mathcal{H}_1$, $\mathcal{H}_2$ Hilbert spaces) with closed range $\overline{\mathcal{R}(U)} = \mathcal{R}(U)$, its associated \emph{pseudo-inverse}\index{pseudo-inverse} \cite{ole1n} is denoted by $U^{\dagger}$. The pseudo-inverse of $U$ is the unique  operator $U^{\dagger} \in \mathcal{B}(\mathcal{H}_2, \mathcal{H}_1)$, satisfying the three relations 
\begin{equation}\label{pseudoinverse2}
\mathcal{N}(U^{\dagger}) = \mathcal{R}(U)^{\perp}, \, \, \, \, \mathcal{R}(U^{\dagger}) = \mathcal{N}(U)^{\perp}, \, \, \, \, UU^{\dagger}x = x \quad (x\in \mathcal{R}(U)).
\end{equation}
Moreover, $UU^{\dagger} = \pi_{\mathcal{R(U)}}$ and $U^{\dagger}U = \pi_{\mathcal{R(U^{\dagger})}}$. We also note that if $U$ has closed range, then so does $U^*$ and it holds $(U^*)^{\dagger} = (U^{\dagger})^*$. On $\mathcal{R}(U)$ we explicitly have $U^{\dagger} = U^* (UU^*)^{-1}$. In case $U$ is bounded and invertible, it holds $U^{\dagger} = U^{-1}$. We refer to \cite{ole1n} for more details on the pseudo-inverse of a closed range operator between Hilbert spaces. 

\

The following important preliminary results will be applied without further reference throughout this manuscript:

\begin{theorem}[Density Principle \cite{gr01}]
Let $X$ and $Y$ be Banach spaces and let $V$ be a dense subspace of $X$. If $T:V \longrightarrow Y$ is bounded, then there exists a unique bounded extension $\tilde{T}:X \longrightarrow Y$ of $T$ satisfying $\Vert \tilde{T} \Vert = \Vert T \Vert$.
\end{theorem}

\begin{theorem}[Neumann's Theorem \cite{ole1n}]
Let $X$ be a Banach space. If $T \in \mathcal{B}(X)$ satisfies $\Vert \mathcal{I}_X - T \Vert < 1$, then $T$ is invertible.
\end{theorem}

\begin{theorem}[Bounded Inverse Theorem \cite{ole1n}]
Any bounded and bijective operator between two Banach spaces has a bounded inverse.
\end{theorem}

\begin{theorem}[Uniform Boundedness Principle \cite{ole1n}]
Let $X$ be a Banach space and $Y$ be a normed space. Suppose that $(T_n ) _{n \in \mathbb{N}}$ is a family of bounded operators $T_n :X \longrightarrow Y$ and assume that $( T_n )_{n \in \mathbb{N}}$ converges pointwise (as $n\longrightarrow \infty$) to some operator $T:X\longrightarrow Y$, $Tx := \lim _{n\longrightarrow \infty} T_n x$. Then $T$ defines a bounded linear operator and we have 
$$\Vert T \Vert \leq \liminf_{n\in \mathbb{N}} \Vert T_n \Vert \leq \sup _{n \in \mathbb{N}} \Vert T_n \Vert < \infty .$$
\end{theorem}

\subsection{Frame theory}\label{Frame Theory}

We begin with a review of frame theory and present a collection of those definitions and results which are the most important ones for this manuscript. In the later sections, we will apply these results mostly without further reference. We refer to \cite{ole1n} for the missing proofs and more details on frame theory. 

\begin{definition}\label{framenotions}
Let $\varphi = (\varphi_i)_{i\in I}$ be a countable family of vectors in $\Hil$. Then $\varphi$ is called 
\begin{itemize}
\item a \emph{frame} for $\mathcal{H}$, if there exist positive constants $A_{\varphi} \leq B_{\varphi}<\infty$, called \emph{frame bounds}, such that 
\begin{flalign}\label{frameinequ}
A_{\varphi} \Vert f \Vert^2 \leq \sum_{i\in I} \vert \langle f, \varphi_i \rangle \vert^2 \leq B_{\varphi} \Vert f \Vert^2 \qquad (\forall f\in \mathcal{H}).
\end{flalign} 
\item a \emph{frame sequence}, if it is a frame for  $\overline{\text{span}}(\varphi_i)_{i\in I}$. 
\item an \emph{$A_{\varphi}$-tight} frame (sequence), or simply \emph{tight} frame (sequence), if it is a frame (sequence), whose frame bounds $A_{\varphi}$ and $B_{\varphi}$ in (\ref{frameinequ}) can be chosen to be equal.
\item a \emph{Parseval} frame (sequence), if it is a $1$-tight frame (sequence).
\item an \emph{exact} frame, if it ceases to be a frame whenever one of the frame vectors $\varphi_i$ is removed from $\varphi$.
\item  a \emph{Bessel sequence} for $\Hil$, if the right but not necessarily the left inequality in (\ref{frameinequ}) is satisfied for some prescribed $B_{\varphi} >0$. In this case we call $B_{\varphi}$ a \emph{Bessel bound}.
\item a \emph{complete} sequence, if $\overline{\text{span}}(\varphi_i)_{i\in I} = \Hil$.
\item a \emph{Riesz basis}, if it is a complete sequence for which there exist constants $0< \alpha_{\varphi} \leq \beta_{\varphi} < \infty $, called \emph{Riesz constants}, such that 
\begin{equation}\label{rieszinequ}
\alpha_{\varphi} \sum _{i\in I} |c_i | ^2 \leq  \bigg\Vert \sum _{i\in I} c_i \varphi_i \bigg\Vert ^2 \leq \beta_{\varphi} \sum _{i\in I} | c_i | ^2
\end{equation}
for all finite scalar sequences $(c_i)_{i\in I} \in \ell^{00}(I)$.
\item a \emph{minimal} sequence, if for each $i\in I$,
$\varphi_i \notin \overline{\text{span}}(\varphi_k)_{k\in I, k\neq i}$.
\end{itemize}
\end{definition}

To any sequence $\varphi = (\varphi_i)_{i\in I}$ of vectors in $\Hil$, we may associate the following canonical frame-related operators
\cite{xxlstoeant11}:
\begin{itemize}
    \item The \emph{synthesis operator} 
    $D_{\varphi} : \text{dom}(D_{\varphi}) \subseteq \ell ^2 (I) \longrightarrow  \mathcal{H}$, where $\text{dom}(D_{\varphi}) = \big\lbrace (c_i)_{i\in I} \in \ell^2 (I): \sum_{i\in I} c_i \varphi_i \in \mathcal{H}\big\rbrace$ and $D_{\varphi} ( c_i )_{i\in I} = \sum_{i\in I} c_i \varphi_i$.
    \item The \emph{analysis operator} $C_{\varphi} : \text{dom}(C_{\varphi}) \subseteq \mathcal{H} \longrightarrow \ell ^2 (I)$, where $\text{dom}(C_{\varphi}) = \lbrace f \in \mathcal{H}:  (\langle f, \varphi_i \rangle )_{i\in I} \in \ell^2 (I) \rbrace$ and $C_{\varphi} f = (\langle f, \varphi_i \rangle )_{i \in I}$.
    \item The \emph{frame operator} $S_{\varphi}: \text{dom}(S_{\varphi}) \subseteq \mathcal{H} \longrightarrow \mathcal{H}$, where $\text{dom}(S_{\varphi}) = \big\lbrace f \in \mathcal{H}: \sum_{i\in I} \langle f, \varphi_i \rangle \varphi_i \in \mathcal{H} \big\rbrace$ and $S_{\varphi} f = \sum_{i\in I} \langle f, \varphi_i \rangle \varphi_i$.  
\end{itemize}
It can be shown that $\varphi$ is a Bessel sequence with Bessel bound $B_{\varphi}$ if and only if $\text{dom}(D_{\varphi}) = \ell^2(I)$  and $\|D_{\varphi}\|\leq \sqrt{B_{\varphi}}$. In this case,  
$D^*_{\varphi} = C_{\varphi} \in \mathcal{B}( \mathcal{H}, \ell^2 (I))$. If $\varphi$ is a frame, then $S_{\varphi}$ is a bounded, self-adjoint, positive and invertible operator, yielding the possibility of \emph{frame reconstruction} via $\mathcal{I}_{\mathcal{H}} = S_{\varphi} S_{\varphi}^{-1} = S_{\varphi}^{-1} S_{\varphi}$, i.e.
\begin{equation}\label{framerec}
f = \sum_{i\in I} \langle f, S_{\varphi} ^{-1} \varphi_i \rangle \varphi_i = \sum_{i\in I} \langle f, \varphi_i \rangle S_{\varphi} ^{-1} \varphi_i \qquad (\forall f\in \mathcal{H}).
\end{equation}
The family $(\widetilde{\varphi_i})_{i\in I} := (S_{\varphi} ^{-1} \varphi_i)_{i\in I}$ in (\ref{framerec}) is again a frame (called the \emph{canonical dual frame}). More generally, a frame $(\psi_i)_{i\in I}$, which satisfies 
\begin{equation}\label{dualframedef}
f = \sum_{i\in I} \langle f, \varphi_i \rangle \psi_i = \sum_{i\in I} \langle f, \psi_i \rangle \varphi_i \qquad (\forall f\in \mathcal{H}),
\end{equation}
is called a \emph{dual frame} of $\varphi$. This means, that frames yield (possibly non-unique and redundant) series expansions of elements in a separable Hilbert space similar to orthonormal bases, whereas in stark contrast to the latter, the frame vectors are not necessarily orthogonal and may be linearly dependent. Moreover, frames and Bessel sequences can often be characterized by properties of the frame-related operators. For instance, a frame is $A_{\varphi}$-tight if and only if $S_{\varphi} = A_{\varphi} \cdot \mathcal{I}_{\Hil}$. In this case, frame reconstruction reduces to $f = A_{\varphi}^{-1} \sum_{i\in I} \langle f, \varphi_i \rangle \varphi_i$. Further results in this direction are given below:

\begin{theorem}\label{Besselchar}\cite{ch08}
Let $\varphi = ( \varphi_i ) _{i\in I}$ be a countable family of vectors in $\mathcal{H}$. Then the following are equivalent.
\begin{enumerate}
\item[(i)] $\varphi$ is a Bessel sequence for $\mathcal{H}$.
\item[(ii)] The synthesis operator $D_{\varphi}$ is well-defined on $\ell^2(I)$ and bounded.
\item[(iii)] The analysis operator $C_{\varphi}$ is well-defined on $\Hil$ and bounded.
\end{enumerate}
\end{theorem}

\begin{theorem}\label{framechar}\cite{ch08}
Let $\varphi = ( \varphi_i ) _{i\in I}$ be a countable family of vectors in $\mathcal{H}$. Then the following are equivalent.
\begin{enumerate}
\item[(i)] $\varphi$ is a frame (resp. Riesz basis) for $\mathcal{H}$.
\item[(ii)] The synthesis operator $D_{\varphi}$ is well-defined on $\ell^2(I)$, bounded and surjective (resp.  well-defined on $\ell^2(I)$, bounded and bijective).
\item[(iii)] The analysis operator $C_{\varphi}$ is well-defined on $\Hil$, bounded, injective and has closed range (resp. well-defined on $\Hil$, bounded and bijective).
\end{enumerate}
\end{theorem}

In particular, the latter result implies that every Riesz basis is a frame. Below we give other equivalent conditions for a frame being a Riesz basis. Recall, that two families $(\varphi_i)_{i\in I}$ and $(\psi_i)_{i\in I}$ of vectors with same index set are called \emph{biorthogonal}, if $\langle \varphi_i , \psi_j \rangle = \delta_{ij}$ for all $i,j\in I$.

\begin{theorem}\label{rieszbasischar}\cite{ole1n}
Let $\varphi = ( \varphi_i ) _{i\in I}$ be a frame for $\Hil$. Then the following are equivalent.
\begin{enumerate}
\item[(i)] $\varphi$ is a Riesz basis for $\mathcal{H}$.
\item[(ii)] $\varphi$ is minimal.
\item[(iii)] $\varphi$ is exact.
\item[(iv)] $\varphi$ has a biorthogonal sequence.
\item[(iii)] $(\varphi_i)_{i\in I}$ and $(S_{\varphi}^{-1}\varphi_i)_{i\in I}$ are biorthogonal.
\item[(iii)] The canonical dual frame $(S_{\varphi}^{-1}\varphi_i)_{i\in I}$ is the unique dual frame of $\varphi$.
\end{enumerate}
\end{theorem}

Bessel sequences, frames and Riesz bases can be characterized in terms of operator-perturbations of orthonormal bases, see below.

\begin{theorem}\label{sequenceschar}\cite{ole1n}
Let $\Hil$ be a separable Hilbert space. Then the following hold: 
\begin{enumerate}
\item[(i)] The Bessel sequences for $\Hil$ are precisely the families of the form $(Ue_i)_{i\in I}$, where $U\in \mathcal{B}(\Hil)$ and $(e_i)_{i\in I}$ is an orthonormal basis for $\Hil$.
\item[(ii)] The frames for $\Hil$ are precisely the families of the form $(Ue_i)_{i\in I}$, where $U\in \mathcal{B}(\Hil)$ is surjective and $(e_i)_{i\in I}$ is an orthonormal basis for $\Hil$.
\item[(iii)] The Riesz bases for $\Hil$ are precisely the families of the form $(Ue_i)_{i\in I}$, where $U\in \mathcal{B}(\Hil)$ is bijective and $(e_i)_{i\in I}$ is an orthonormal basis for $\Hil$.
\end{enumerate}
\end{theorem}

This allows us to formulate a result about which operators preserve certain properties: 

\begin{corollary}\label{sequenceschar2}\cite{ole1n}
Let $\Hil$ be a separable Hilbert space. Then the following hold: 
\begin{enumerate}
\item[(i)] If $(\varphi_i)_{i\in I}$ is a Bessel sequence for $\Hil$ and $U:\Hil \longrightarrow \Hil$ is bounded, then $(U\varphi_i)_{i\in I}$ is a Bessel sequence for $\Hil$ as well.
\item[(ii)] If $(\varphi_i)_{i\in I}$ is a frame for $\Hil$ and $U:\Hil \longrightarrow \Hil$ is bounded and surjective, then $(U\varphi_i)_{i\in I}$ is a frame for $\Hil$ as well.
\item[(iii)] If $(\varphi_i)_{i\in I}$ is a Riesz basis for $\Hil$ and $U:\Hil \longrightarrow \Hil$ is bounded and bijective, then $(U\varphi_i)_{i\in I}$ is a Riesz basis for $\Hil$ as well.
\end{enumerate}
\end{corollary}

We conclude this section with the following characterization of all dual frames of a given frame $\varphi$, which stems from a characterization of all bounded left-inverses of the analysis operator $C_{\varphi}$ (analogous to Lemma \ref{leftright}) and an application of Theorem \ref{sequenceschar}.

\begin{theorem}\label{dualframechar}\cite{ole1n}
Let $\varphi = (\varphi_i)_{i\in I}$ be a frame for $\Hil$. Then the dual frames of $\varphi$ are precisely the families 
\begin{equation}\label{dual2}
(\psi_i)_{i\in I} = \Big( S_{\varphi}^{-1} \varphi_i + h_i - \sum_{k\in I} \langle S_{\varphi}^{-1} \varphi_i, \varphi_k \rangle h_k \Big)_{i\in I},
\end{equation}
where $(h_i)_{i\in I}$ is a Bessel sequence for $\Hil$.
\end{theorem}

By the above result and Theorem \ref{rieszbasischar}, a frame, which is not a Riesz basis, has infinitely many dual frames. In case $\varphi$ is a Riesz basis, we see that the remainder terms $h_i - \sum_{k\in I} \langle S_{\varphi}^{-1} \varphi_i, \varphi_k \rangle h_k$ in (\ref{dual2}) vanish, since $(\varphi_i)_{i\in I}$ and $(S_{\varphi}^{-1}\varphi_i)_{i\in I}$ are biorthogonal in this case.

\section{Fusion frames}
\label{Bessel fusion sequences and fusion frames}

In this section, we present the basic definitions and results regarding fusion frames, Bessel fusion sequences and their relation to the fusion frame-related operators.
Because this is a survey on this topic, we will give full proofs. 
The reader will notice many parallels between frames and fusion frames in this section. However, as we have already mentioned and will see later, fusion frame theory raises a lot of unexpected and interesting questions, which are non-existent for Hilbert space frames.

At first, we motivate the concept of fusion frames by generalizing the frame definition. We rewrite the terms $\vert \langle f, \varphi_i \rangle \vert^2$ from the frame inequalities (\ref{frameinequ}) as follows:  
\begin{flalign}\label{mot1}
\vert \langle f,  \varphi_i \rangle \vert ^2 &= \langle f,  \varphi_i \rangle \langle \varphi_i, f \rangle \notag \\
&= \big\langle \langle f, \varphi_i \rangle \varphi_i, f \big\rangle \notag \\
&= \Vert \varphi_i \Vert ^2 \left\langle \left\langle f, \frac{\varphi_i}{\Vert \varphi_i \Vert} \right\rangle \frac{\varphi_i}{\Vert \varphi_i \Vert}, f \right\rangle. 
\end{flalign}
If we set $V_i:= \overline{\text{span}}\{\varphi_i\} = \text{span}\{\varphi_i\}$, then the singleton $\left\{\Vert \varphi_i \Vert^{-1} \cdot \varphi_i \right\}$ is an orthonormal basis for $V_i$. In particular, the orthogonal projection $\pi_{V_i}$ of $\Hil$ onto the one-dimensional subspace $V_i$ is given by 
$\pi_{V_i} f = \left\langle f, \frac{\varphi_i}{\Vert \varphi_i \Vert} \right\rangle \frac{\varphi_i}{\Vert \varphi_i \Vert}$. Inserting that into (\ref{mot1}) and setting $v_i := \Vert \varphi_i \Vert$ yields
\begin{flalign}
\vert \langle f, \varphi_i \rangle \vert ^2 
&= v_i^2 \langle \pi _{V_i} f, f\rangle \notag \\
&= v_i^2 \langle \pi _{V_i}^2 f, f\rangle \notag \\
&= v_i^2 \langle \pi _{V_i} f, \pi_{V_i} f \rangle \notag = v_i^2 \Vert \pi _{V_i} f \Vert^2 \notag .
\end{flalign}
This allows us to rewrite the frame inequalities to \cite[Proposition 2.14]{cakuli08}
$$A_{\varphi} \Vert f \Vert^2 \leq \sum_{i\in I} v_i^2 \Vert \pi _{V_i} f \Vert^2 \leq B_{\varphi} \Vert f \Vert^2 \qquad (\forall f\in \Hil).$$
Therefore, we see that frames can also be viewed as weighted sequences of $1$-dimensional closed subspaces satisfying the inequalities (\ref{frameinequ}). Admitting arbitrary weights  - in the sense of weighted frames \cite{xxljpa1} - and in particular, arbitrary closed subspaces leads to the definition of a fusion frame and related notions \cite{caskut04}: 

\begin{definition}\label{fusionframedef}
Let $V = (V_i , v_i )_{i\in I}$ be a countable family of closed subspaces $V_i$ of $\mathcal{H}$ and weights $v_i >0$. Then $V$ is called 
\begin{itemize}
\item a \emph{fusion frame} for $\mathcal{H}$, if there exist positive constants $0<A_V \leq B_V < \infty$, called \emph{fusion frame bounds}, such that
\begin{flalign}\label{fusframeinequ}
A_V \Vert f \Vert^2 \leq \sum_{i\in I} v_i^2 \Vert \pi_{V_i} f \Vert^2 \leq B_V \Vert f \Vert^2 \qquad (\forall f\in \mathcal{H}).
\end{flalign}
 \item a \emph{fusion frame sequence}, if it is a fusion frame for $\Hil := \overline{\text{span}}(V_i)_{i\in I}$. 
\item an \emph{$A_V$-tight fusion frame (sequence)}, or simply \emph{tight fusion frame (sequence)}, if it is a fusion frame (sequence), whose fusion frame bounds $A_V$ and $B_V$ in (\ref{fusframeinequ}) can be chosen to be equal. 
\item a \emph{Parseval fusion frame (sequence)}, if it is a $1$-tight fusion frame (sequence).
\item an \emph{exact fusion frame}, if it is a fusion frame, that ceases to be a fusion frame whenever one component $(V_k,v_k)$ is removed from $V$.
\item  a \emph{Bessel fusion sequence} for $\Hil$, if the right but not necessarily the left inequality in (\ref{fusframeinequ}) is satisfied for some $B_V >0$. In this case we call $B_V$ a \emph{Bessel fusion bound}.
\end{itemize}
\end{definition} 

\begin{remark}\label{rank1fusion}

\

\noindent (a) For any Bessel fusion sequence $V = (V_i, v_i)_{i\in I}$ (and hence any fusion frame), we will always implicitly assume (without loss of generality), that $V_i \neq \lbrace 0 \rbrace$ for all $i\in I$, because a zero-space $V_i = \lbrace 0 \rbrace$ does not contribute to the sum in (\ref{fusframeinequ}) anyway. This is done to avoid case distinctions in some of the theoretical results. 

\noindent (b) As a first glimpse at the differences between frames and fusion frames, consider the following: It is known that for every $m,n \in \mathbb{N}$ there is a $m$-element Parseval frame for any $n$-dimensional Hilbert space \cite{CL} (and the vectors may even
be equal norm). The situation for fusion frames is much more
complicated. For example, there is no Parseval fusion frame $V = (V_i,v_i)_{i=1}^2$ for $\mathbb{R}^3$ consisting of two 2-dimensional subspaces for any choice of weights. 
This is true since, in this case, there must exist a unit vector $f$ contained in the $1$-dimensional space $V_1\cap V_2$ 
and so $v_1^2\|\pi_{V_1}f\|^2 + v_2^2\|\pi_{V_2}f\|^2=(v_1^2+v_2^2)\|f\|^2$, while
a unit vector $f\in V_1\setminus V_2$ will satisfy $v_1^2\|\pi_{V_1}f\|^2 + v_2^2\|\pi_{V_2}f\|^2<(v_1^2+v_2^2) \|f\|^2$, and thus we cannot have $S_V = \mathcal{I}_{\Hil}$ (see Corollary \ref{tightfusionchar}). Currently, there are no necessary and sufficient
conditions known for the existence of Parseval fusion frames for arbitrary dimensions of the subspaces.

\noindent (c) A converse to our motivation for fusion frames via frames is the following: Assume that $V = (V_i, v_i)_{i\in I}$ is a fusion frame such that $\text{dim}(V_i) = 1$ for all $i\in I$. Then, for every $i$, there exists a suitable vector $\varphi_i$, such that $V_i = \text{span}\lbrace \varphi_i \rbrace =  \overline{\text{span}}\lbrace \varphi_i \rbrace$ and $\pi_{V_i} f = \langle f, \Vert \varphi_i \Vert^{-1} \cdot \varphi_i  \rangle \Vert \varphi_i \Vert^{-1} \cdot \varphi_i$. It is now easy to see, that $V$ being a fusion frame is equivalent to $(v_i\Vert \varphi_i \Vert^{-1} \cdot \varphi_i)$ being a frame.

\noindent (d) If $V = (V_i, v_i)_{i\in I}$ is a fusion frame for $\Hil$, then the family $(V_i)_{i\in I}$ of subspaces $V_i$ is \emph{complete}, i.e. it holds $\overline{\text{span}}(V_i)_{i\in I} = \Hil$. To see this, assume for the contrary, that there exists a non-zero $g \in \big( \overline{\text{span}}(V_i)_{i\in I} \big)^{\perp}$. Then this $g$ would violate the lower fusion frame inequality in (\ref{fusframeinequ}).
\end{remark} 

A different motivation for fusion frames comes from the following problem:

\begin{question}{Goal}
Let  $(V_i)_{i\in I}$ be a family of closed subspaces  of $\mathcal{H}$ and assume that for every $i\in I$, $(\varphi_{ij})_{j\in J_i}$ is a frame for $V_i$. Give necessary and sufficient conditions on the subspaces $V_i$, such that the collection $(\varphi_{ij})_{i\in I, j\in J_i}$ of all local frames (including all resulting repetitions of the vectors $\varphi_{ij}$) forms a global frame for the whole Hilbert space $\Hil$.
\end{question}  

The following result shows that the notion of a fusion frame is perfectly fitting for this task. 

\begin{theorem}\label{fusframesysTHM}\cite{cakuli08}
Let $(V_i)_{i \in I}$ be a family of closed subspaces of $\mathcal{H}$ and $(v_i)_{i\in I}$ be a family of weights. For every $i\in I$, let $(\varphi _{ij})_{j\in J_i}$ be a frame for $V_i$ with frame bounds $A_i$ and $B_i$, and suppose that $0<A = \inf_{i\in I} A_i \leq \sup_{i \in I} B_i = B < \infty$. Then the following are equivalent:
\begin{enumerate}
\item[(i)] $(V_i , v_i)_{i \in I}$ is a fusion frame for $\mathcal{H}$. 
\item[(ii)] $(v_i \varphi_{ij} )_{i\in I, j\in J_i}$ is a frame for $\mathcal{H}$.
\end{enumerate}
In particular, if $(V_i , v_i)_{i \in I}$ is a fusion frame with bounds $A_V\leq B_V$, then $(v_i \varphi_{ij} )_{i\in I, j\in J_i}$ is a frame for $\mathcal{H}$ with frame bounds $AA_V \leq BB_V$. Conversely, if $(v_i \varphi_{ij} )_{i\in I, j\in J_i}$ is a frame for $\mathcal{H}$ with bounds $A_{v\varphi}\leq B_{v\varphi}$, then $(V_i , v_i)_{i \in I}$ is a fusion frame with fusion frame bounds bounds $\frac{A_{v\varphi}}{B}\leq \frac{B_{v\varphi}}{A}$.
\end{theorem}

\begin{proof}
Since $(\varphi _{ij})_{j\in J_i}$ is a frame for $V_i$, we formally see that
\begin{align}
A\sum_{i\in I} v_i ^2 \Vert \pi_{V_i} f \Vert ^2 &\leq \sum_{i\in I} v_i ^2 A_i \Vert \pi_{V_i} f \Vert ^2 \notag \\
&\leq \sum_{i\in I} v_i ^2  \sum_{j\in J_i} \vert \langle \pi_{V_i} f, \varphi_{ij} \rangle \vert ^2 = \sum_{i\in I} \sum_{j\in J_i} \vert \langle f, v_i \varphi_{ij} \rangle \vert ^2  \notag \\
&\leq \sum_{i\in I} v_i ^2 B_i \Vert \pi_{V_i} f \Vert ^2 \leq B\sum_{i\in I} v_i ^2 \Vert \pi_{V_i} f \Vert ^2 . \notag
\end{align}
Thus, if $(V_i , v_i)_{i\in I} $ is a fusion frame for $\mathcal{H}$ with bounds $A_V \leq B_V$, then we obtain
$$AA_V \Vert f \Vert ^2 \leq \sum_{i\in I} \sum_{j\in J_i} \vert \langle f, v_i f_{ij} \rangle \vert ^2 \leq BB_V \Vert f \Vert ^2 .$$
Conversely, if  $(v_i \varphi_{ij} )_{i\in I, j\in J_i}$ is a frame for $\mathcal{H}$ with frame bounds $A_{v\varphi} \leq B_{v\varphi}$, then
$$\frac{A_{v\varphi}}{B} \Vert f \Vert ^2 \leq \sum_{i\in I} v_i ^2 \Vert \pi_{V_i} f \Vert ^2 \leq \frac{B_{v\varphi}}{A} \Vert f \Vert ^2 .$$
This completes the proof.
\end{proof}

The special case $A = A_i = K = B_i = B$, for every $i\in I$, in the latter result yields the following corollary.

\begin{corollary}\label{parsfusion}\cite{cakuli08} Let $(V_i)_{i \in I}$ be a family of closed subspaces of $\mathcal{H}$ and $(v_i)_{i\in I}$ be a family of weights. For every $i\in I$, let $(\varphi _{ij})_{j\in J_i}$ be a Parseval frame for $V_i$. Then, for every $K>0$, the following are equivalent:
\begin{enumerate}
\item[(i)] $(V_i , v_i)_{i \in I}$ is a $K$-tight fusion frame for $\mathcal{H}$. 
\item[(ii)] $(v_i \varphi_{ij} )_{i\in I, j\in J_i}$ is a $K$-tight frame for $\mathcal{H}$.
\end{enumerate}
\end{corollary}

This $-$ corresponding to the approach in many  applications that one considers a fusion frame
with local frames for each of its subspaces $-$ leads to the
following definition, see also Sections~\ref{S dual fusion frame systems} and~\ref{Implementations}:

\begin{definition}\label{def:fusionframeseystem}\cite{cakuli08}
Let $V=(V_i,v_i)_{i\in I}$ be a fusion frame (Bessel fusion
sequence), and let $\varphi^{(i)}:=(\varphi_{ij})_{j\in J_i}$ be a
frame for $V_{i}$ for $i \in I$ with frame bounds $A_i$ and $B_i$. 
Assume that $A = \inf A_i > 0$ and $B = \sup B_i < \infty$. Then $(V_i,v_i,\varphi^{(i)})_{i
\in I}$ is called a \emph{fusion frame system} (\emph{Bessel fusion system}).
\end{definition}

In the following, we collect a few observations on the weights $v_i$ associated to a Bessel fusion sequence. For that, the following notion will be useful: We call a family $(v_i)_{i\in I}$ of weights \emph{semi-normalized} (see also \cite{xxlmult1}), if there exist positive constants $m\leq M$ such that $m\leq v_i \leq M$ for all $i\in I$.

\begin{lemma}\label{weights}\cite{heimoanza14}
Let $V=(V_i , v_i)_{i \in I}$ be a Bessel fusion sequence with Bessel fusion bound $B_V$. Then $v_i \leq \sqrt{B_V}$ for all $i\in I$.
\end{lemma}

\begin{proof}
Since we always assume $V_i \neq \lbrace 0 \rbrace$ for a Bessel fusion sequence $(V_i,v_i)_{i\in I}$, we may choose some non-zero $f\in V_i$ and see that 
$$v_i^2\|f\|^2=v_i^2\left\|\pi_{V_i}f\right\|^2\leq \sum_{i\in I}v_i^2\left\|\pi_{V_i}f\right\|^2\leq B_V\|f\|^2$$
for every $i\in I$. Dividing by $\Vert f \Vert^2$ yields the result.
\end{proof}

By Lemma \ref{weights}, the weights $v_i$ associated to a fusion frame $V = (V_i, v_i)_{i\in I}$ are uniformly bounded from above by the square root of its upper frame bound. However, $(v_i)_{i\in I}$ needs not to be semi-normalized in general. An easy counter-example is the Parseval fusion frame $(\Hil, 2^{-n})_{n=1}^{\infty}$ for $\Hil$. 

\begin{lemma}\label{weights2}
Let $(w_i)_{i\in I}$ be a semi-normalized family of weights with respect to the constants $m\leq M$. Then $(V_i,v_i)_{i\in I}$ is a fusion frame for $\Hil$ if and only if $(V_i,v_i w_i)_{i\in I}$ is a fusion frame for $\Hil$. In particular, $(V_i,v_i)_{i\in I}$ is a Bessel fusion sequence for $\Hil$ if and only if $(V_i,v_i w_i)_{i\in I}$ is a Bessel fusion sequence for $\Hil$.
\end{lemma}

\begin{proof}
Similar to the proof of Theorem \ref{fusframesysTHM}, the statement follows from 
$$m^2 \sum_{i\in I}v_i^2 \Vert \pi_{V_i} f \Vert^2 \leq \sum_{i\in I} v_i^2 w_i^2 \Vert \pi_{V_i} f \Vert^2 \leq M^2 \sum_{i\in I}v_i^2 \Vert \pi_{V_i} f \Vert^2 \qquad (\forall f\in \Hil).$$
\end{proof}

Next we introduce the fusion frame-related operators and characterize Bessel fusion sequences and fusion frames with them.

\

The canonical representation space for frames is the sequence space $\ell^2(I)$. In the theory of fusion frames we consider generalizations of $\ell^2(I)$. For a family $(V_i)_{i \in I}$ of Hilbert spaces (e.g. $V_i$ closed subspaces of $\mathcal{H}$), the \emph{Hilbert direct sum} $\big( \sum_{i\in I} \oplus V_i \big)_{\ell ^2}$, also called the \emph{direct sum of the Hilbert spaces} $V_i$ \cite{conw1,cas03}, is defined by 
\begin{equation*}\label{HilbertDirectSum}
    \Big( \sum_{i\in I} \oplus V_i \Big)_{\ell ^2} = \left\lbrace (f_i)_{i\in I} \in  (V_i)_{i \in I}: \sum_{i\in I}\Vert f_i \Vert ^2 <\infty   \right\rbrace .
\end{equation*}
It is easy to verify that for $(f_i)_{i\in I}, (g_i)_{i\in I} \in \big( \sum_{i\in I} \oplus V_i \big)_{\ell ^2}$, the operation 
$$\langle (f_i)_{i\in I}, (g_i)_{i\in I} \rangle_{( \sum_{i\in I} \oplus V_i )_{\ell ^2}} := \sum_{i\in I} \langle f_i, g_i \rangle  $$
defines an inner product on $\big( \sum_{i\in I} \oplus V_i \big)_{\ell ^2}$ (see also \cite{conw1}), yielding the norm 
\begin{equation}\label{norm}
    \Vert (f_i)_{i\in I} \Vert ^2 _{( \sum_{i\in I} \oplus V_i ) _{\ell ^2}} = \sum_{i\in I}\Vert f_i \Vert^2 .
\end{equation}
Moreover, adapting any elementary completeness proof for $\ell^2(I)$ yields that the space $\big( \sum_{i\in I} \oplus V_i \big)_{\ell ^2}$ is complete with respect to the norm (\ref{norm}), i.e. a Hilbert space (see e.g. \cite{kohl21}). Note that $\big( \sum_{i\in I} \oplus \mathbb{C} \big)_{\ell ^2} = \ell^2(I)$. In analogy to the dense inclusion $\ell^{00}(I) \subseteq \ell^2(I)$, the space of finite vector sequences
$$\Big( \sum_{i\in I} \oplus V_i \Big)_{\ell ^{00}} = \Big\lbrace (f_i)_{i\in I} \in  (V_i)_{i \in I}: f_i \neq 0 \, \, \text{for only finitely many} \, \, i \Big\rbrace$$
is a dense subspace of $\big( \sum_{i\in I} \oplus V_i \big)_{\ell ^2}$. For the remainder of this manuscript we abbreviate
$$\mathcal{K}_V^2 := \Big( \sum_{i\in I} \oplus V_i \Big)_{\ell ^2} \quad \text{and} \quad \mathcal{K}_{V}^{00} := \Big( \sum_{i\in I} \oplus V_i \Big)_{\ell ^{00}}$$
and analogously
$$\mathcal{K}_{\Hil}^2 := \Big( \sum_{i\in I} \oplus \Hil \Big)_{\ell ^2} \quad \text{and} \quad \mathcal{K}_{\Hil}^{00} := \Big( \sum_{i\in I} \oplus \Hil \Big)_{\ell ^{00}}.$$
We will continue our discussion on Hilbert direct sums in Subsection \ref{Operators between Hilbert direct sums}. For now, we proceed with the definition and basic properties of the fusion frame-related operators. 

Let $V=(V_i, v_i)_{i\in I}$ be an arbitrary weighted sequence of closed subspaces in $\Hil$. The following canonical operators associated to $V$ play a central role in fusion frame theory: 
\begin{itemize}
    \item The \emph{synthesis operator} $D_V :\text{dom}(D_V) \subseteq \mathcal{K}_{\Hil}^2 \longrightarrow \mathcal{H}$, where $\text{dom}(D_V) = \big\lbrace (f_i)_{i\in I} \in \mathcal{K}_{\Hil}^2: \sum_{i \in I} v_i \pi_{V_i} f_i \, \, \text{converges} \big\rbrace$ and $D_V (f_i )_{i \in I} = \sum_{i \in I} v_i \pi_{V_i}f_i$,
    \item The \emph{analysis operator} 
    $C_V : \text{dom}(C_V) \subseteq \mathcal{H} \longrightarrow \mathcal{K}_{\Hil}^2$, where $\text{dom}(C_V) = \big\lbrace f \in \mathcal{H}:  ( v_i \pi _{V_i} f )_{i \in I} \in \mathcal{K}_{\Hil}^2 \big\rbrace$ and $C_V f = ( v_i \pi _{V_i} f )_{i \in I}$,
    \item The \emph{fusion frame operator}
    $S_V : \text{dom}(S_V) \subseteq \mathcal{H} \longrightarrow \mathcal{H}$, where $\text{dom}(S_V) = \big\lbrace f \in \mathcal{H}: \sum_{i\in I} v_i ^2 \pi_{V_i} f \, \, \text{converges}  \big\rbrace$ and $S_V f = \sum_{i\in I} v_i ^2 \pi_{V_i} f$.
\end{itemize}

\noindent The synthesis operator is always densely defined, since $\mathcal{K}_{\Hil}^{00} \subseteq \text{dom}(D_V)$ is dense in $\mathcal{K}_{\Hil}^2$, and we always have $D_V^* = C_V$ (see below for the bounded case). On the other hand, one can show \cite{petermitraun2021}, that the analysis operator $C_V$, in general, is not necessarily densely defined. If $C_V$ is densely defined (or, equivalently, $D_V$ is closable), then $D_V \subseteq C_V^*$  and in this case, $\overline{D_V}=C_V^*$. Moreover, one can show that always $S_V = D_V C_V$ and that $S_V$ is symmetric on its domain. We refer to \cite{petermitraun2021} for more details on these (possibly unbounded) operators associated to general subspace sequences. For our purpose it is sufficient to treat the bounded case only (compare with Theorem \ref{synthesisthm}). 

\begin{lemma}\label{adjointlemma}\cite{caskut04}
Let $V=(V_i, v_i)_{i\in I}$ be a weighted sequence of closed subspaces in $\Hil$. If $D_V \in \mathcal{B}(\mathcal{K}_{\Hil}^2, \mathcal{H})$ 
then $D_V^* = C_V$; if $C_V \in \mathcal{B}(\mathcal{H}, \mathcal{K}_{\Hil}^2)$, then $C_V^* = D_V$; and either of the two latter cases implies that $S_V = D_V C_V \in \mathcal{B}(\mathcal{H})$ is self-adjoint.
\end{lemma}

\begin{proof}
For $g \in \Hil$ and $(f_i)_{i\in I} \in \mathcal{K}_{\mathcal{H}}^2$ we have
\begin{flalign}
\big\langle D_V^* g, (f_i)_{i\in I} \big\rangle_{\mathcal{K}_{\mathcal{H}}^2} &= \big\langle g, D_V (f_i)_{i\in I} \big\rangle \notag \\ 
&= \sum_{i\in I} \langle g, v_i \pi_{V_i} f_i\rangle \notag \\
&= \sum_{i\in I} \langle v_i \pi_{V_i} g, f_i\rangle \notag \\
&= \big\langle (v_i \pi_{V_i} g)_{i\in I}, (f_i)_{i\in I} \big\rangle_{\mathcal{K}_{\mathcal{H}}^2} \notag \\
&= \big\langle C_V g, (f_i)_{i\in I} \big\rangle_{\mathcal{K}_{\mathcal{H}}^2} = \big\langle g, C_V^* (f_i)_{i\in I} \big\rangle_{\mathcal{K}_{\mathcal{H}}^2} . \notag 
\end{flalign}
This implies the first two statements  and the rest is clear.
\end{proof}

\begin{theorem}\label{synthesisthm}\cite{caskut04}
Let $V=(V_i, v_i)_{i\in I}$ be a weighted sequence of closed subspaces of $\Hil$. Then the following are equivalent:
\begin{enumerate}
    \item[(i)] $V$ is a Bessel fusion sequence with Bessel fusion bound $B_V$.
    \item[(ii)] The synthesis operator $D_V$ associated to $V$ is a bounded operator from $\mathcal{K}_{\Hil}^2$ into $\mathcal{H}$ with $\Vert D_V \Vert \leq \sqrt{B_V}$.
    \item[(iii)] The analysis operator $C_V$ associated to $V$ is a bounded operator from $\mathcal{H}$ into $\mathcal{K}_{\Hil}^2$ with $\Vert C_V \Vert \leq \sqrt{B_V}$.
    \item[(iv)] The fusion frame operator $S_V$ associated to $V$ is bounded on $\mathcal{H}$ with $\Vert S_V \Vert \leq B_V$.
\end{enumerate}
\end{theorem}

\begin{proof}
(i) $\Rightarrow$ (ii) Assume that $V$ is a Bessel fusion sequence with Bessel fusion bound $B_V$. Then for any finite vector sequence $h = (h_i)_{i\in J} \in \mathcal{K}_{\Hil}^{00}$ we have
\begin{flalign}\label{comp1}
\Big\Vert \sum_{i\in J} v_i \pi_{V_i} h_i \Big\Vert ^2 &= \sup_{\Vert g \Vert = 1} \Big\vert \Big\langle \sum_{i\in J} v_i \pi_{V_i} h_i , g \Big\rangle \Big\vert^2
\notag \\
&= \sup_{\Vert g \Vert = 1} \Big\vert \sum_{i\in J} v_i \langle h_i , \pi_{V_i} g \rangle \Big\vert^2 \notag \\
&\leq \sup_{\Vert g \Vert = 1} \Big( \sum_{i\in J} v_i \Vert \pi_{V_i} g \Vert \Vert h_i \Vert  \Big)^2 \notag \\
&\leq \sup_{\Vert g \Vert = 1} \Big( \sum_{i\in J} v_i^2 \Vert \pi_{V_i} g \Vert ^2  \Big) \Big( \sum_{i\in J} \Vert h_i \Vert^2 \Big) \notag  \leq B_V \Vert h \Vert_{\mathcal{K}_{\Hil}^2}^2. \notag
\end{flalign}
This implies that $D_V$ is bounded on $\mathcal{K}_{\Hil}^{00}$ by $\sqrt{B_V}$. Since $\mathcal{K}_{\Hil}^{00}$ is a dense subspace of $\mathcal{K}_{\Hil}^{2}$, $D_V$ extends to a bounded operator $D_V: \mathcal{K}_{\Hil}^{2} \longrightarrow \mathcal{H}$ with $\Vert D_V \Vert \leq \sqrt{B_V}$. 

\noindent (ii) $\Rightarrow$ (iii) $\Rightarrow$ (iv) This follows from Lemma \ref{adjointlemma}.

\noindent (iv) $\Rightarrow$ (i) If $S_V$ is bounded with $\Vert S_V \Vert \leq B_V$, then for all $f\in \Hil$ we have 
$$\sum_{i\in I} v_i^2 \Vert \pi_{V_i} f\Vert^2 = \langle S_V f, f \rangle \leq \Vert S_V f \Vert \Vert f \Vert \leq \Vert S_V \Vert \Vert f \Vert ^2 \leq B_V \Vert f \Vert ^2,$$
i.e. $V$ is a Bessel fusion sequence with Bessel fusion bound $B_V$.
\end{proof}

\begin{remark}\label{unconditional}
The proof of Theorem \ref{synthesisthm} reveals, that for any Bessel fusion sequence $V = (V_i,v_i)_{i\in I}$, the sum $\sum_{i\in I}v_i^2 \pi_{V_i}f_i$ converges unconditionally in $\Hil$ for all $(f_i)_{i\in I}\in \mathcal{K}^2_{\Hil}$.
\end{remark}

By Theorem \ref{synthesisthm}, the fusion frame operator $S_V$ is bounded on $\Hil$, whenever $V = (V_i, v_i)_{i\in I}$ is a Bessel fusion sequence. If, in addition, $V$ is a fusion frame, then $S_V$ is invertible and thus yields the possibility of perfect reconstruction, see Theorem \ref{fusionframethm} below.  

\begin{svgraybox}
In analogy to frame theory, fusion frames enable perfect reconstruction without the assumption, that the subspaces $V_i$ are orthogonal. This can be seen as the most important advantage of fusion frames in comparison with orthogonal subspace decompositions.
\end{svgraybox}

\begin{theorem}\label{fusionframethm} \cite{caskut04}
Let $V = (V_i, v_i)_{i\in I}$ be a fusion frame with fusion frame bounds $A_V \leq B_V$. Then the fusion frame operator $S_V$ is bounded, self-adjoint, positive and invertible with 
\begin{equation}\label{fusionnorms}
    A_V \leq \Vert S_V \Vert \leq B_V \quad \text{and} \quad B_V^{-1} \leq \Vert S_V^{-1} \Vert \leq A_V^{-1} . 
\end{equation}
In particular
\begin{equation}\label{fusionframereconstruction}
f = \sum_{i\in I} v_i ^2 \pi_{V_i} S_V ^{-1} f = \sum_{i\in I} v_i ^2 S_V ^{-1} \pi_{V_i} f \qquad (\forall f\in \Hil)
\end{equation}
where both series converge unconditionally.
\end{theorem}

\begin{proof}
If $V$ is a fusion frame, then $V$ in particular is a Bessel fusion sequence. Hence, Theorem \ref{synthesisthm} and Lemma \ref{adjointlemma} imply, that $S_V$ is bounded and self-adjoint. Moreover, by the definition of $S_V$, we may rewrite the fusion frame inequalities (\ref{fusframeinequ}) to 
\begin{equation}\label{positiveinequ}
\langle A_V f, f \rangle \leq \langle S_V f, f \rangle \leq \langle B_V f, f \rangle \qquad (\forall f\in \mathcal{H}).
\end{equation}
Thus $S_V$ is positive. Furthermore, by manipulating (\ref{positiveinequ}), we obtain $$0 \leq \big\langle (\mathcal{I}_{\Hil} - B_V^{-1}S_V) f, f \big\rangle \leq \Big\langle \frac{B_V - A_V}{B_V} f, f \Big\rangle \qquad (\forall f\in \Hil)$$
and thus, since $\mathcal{I}_{\Hil} - B_V^{-1}S_V$ is self-adjoint,
$$\Vert \mathcal{I}_{\Hil} - B_V^{-1}S_V \Vert = \sup_{\Vert f \Vert =1} \langle (\mathcal{I}_{\Hil} - B_V^{-1}S_V) f, f \rangle \leq  \frac{B_V - A_V}{B_V} < 1.$$
Now, an application of Neumann's Theorem yields that $B_V^{-1}S_V$ is invertible, hence $S_V$ is invertible too. In particular, for all $f\in \Hil$, we have $f = S_V S_V ^{-1} f = S_V ^{-1} S_V f$, which implies (\ref{fusionframereconstruction}) with unconditional convergence (see Remark \ref{unconditional}). Finally, the left-hand side of (\ref{fusionnorms}) follows immediately from (\ref{positiveinequ}). To show the right-hand side of (\ref{fusionnorms}), observe that we have $\Vert S_V^{-1} \Vert \geq \Vert S_V \Vert^{-1} \geq B_V^{-1}$ and that $A_V \Vert S_V ^{-1} f \Vert ^2 \leq \langle S_V S_V ^{-1} f , S_V ^{-1} f \rangle = \langle f , S_V ^{-1} f \rangle \leq \Vert S_V ^{-1} \Vert \Vert f \Vert ^2$ implies $\Vert S_V^{-1} \Vert \leq A_V^{-1}$. 
\end{proof}

So, for any fusion frame $V$, we have the possibility of perfect reconstruction of any $f\in \Hil$ via (\ref{fusionframereconstruction}), where the subspaces $V_i$ need not to be orthogonal and can have non-trivial intersection. In particular, we can combine fusion frame reconstruction with local frame reconstruction in a fusion frame system. More precisely, given a fusion frame system $(V_i, v_i, \varphi^{(i)})_{i\in I}$ with associated global frame $v\varphi$, we obtain several possibilities to perform perfect reconstruction via the frame vectors $v_i \varphi_{ij}$ (see Definition \ref{def:fusionframeseystem}), which demonstrates that fusion frames are tremendously useful for distributed processing tasks \cite{cakuli08}.  

\begin{lemma}\label{globalisfusionre}\cite{cakuli08}
Let $V=(V_i,v_i, \varphi^{(i)})_{i\in I}$ be a fusion frame system in $\Hil$ with global frame $v\varphi$. If $(\varphi_{ij}^d)_{j\in J_i}$ is a dual of $\varphi^{(i)}$ for every $i\in I$, then
\begin{enumerate}
    \item[(i)] $(S_V^{-1} v_i \varphi^d_{ij})_{i\in I, j\in J_i}$ is a dual frame of the global frame $v\varphi$. In particular 
$$f=\sum_{i\in I} \sum_{
j\in J_i}\left< f,v_i\varphi_{ij}\right>S_V^{-1}v_i\varphi_{ij}^d \qquad (\forall f\in \Hil).$$
\item[(ii)] $S_V=D_{v\varphi^d}C_{v\varphi}=D_{v\varphi}C_{v\varphi^d}$, where $v\varphi^d=(v_i\varphi_{ij}^d)_{i\in I,j\in J_i}$.
\end{enumerate}
\end{lemma}

\begin{proof}
(i) For every $f\in \Hil$ it holds $\pi_{V_i}f=\sum_{i\in I}\left<f,\varphi_{ij}\right>\varphi_{ij}^d$ and hence
$$f=\sum_{i\in I}v_i^2S_V^{-1}\pi_{V_i}f=\sum_{i\in I} \sum_{
j\in J_i}\left< f,v_i\varphi_{ij}\right>S_V^{-1}v_i\varphi_{ij}^d.$$

(ii) Similarly, we see that 
$$S_Vf=\sum_{i\in I}v_i^2\pi_{V_i}f=
\sum_{i\in I}\sum_{j\in J_i}\left< f,v_i\varphi_{ij}\right>v_i\varphi_{ij}^d,$$
where we note that $v\varphi^d$ is indeed a frame for $\Hil$ by Theorem \ref{fusframesysTHM}.
\end{proof}

In view of applications, knowing the action of the inverse fusion frame operator $S_V^{-1}$ on elements $f\in \Hil$ is of utmost importance. If one prefers to avoid the inversion of $S_V$, approximation procedures such as the \emph{fusion frame algorithm} (see Lemma \ref{num2}) can be performed. However, for tight fusion frames, the computation of $S_V^{-1}$ becomes fairly simple and thus fusion frame reconstruction can be computed with ease, see (\ref{fusionframereconstruction2}) and (\ref{fusionframereconstruction3}).  

First, we note the following reformulation of the inequalities (\ref{positiveinequ}). 

\begin{corollary}\label{charoffusfr}\cite{cakuli08}
Let $V=(V_i, v_i)_{i\in I}$ be a weighted sequence of closed subspaces of $\Hil$ and let $0<A_V \leq B_V < \infty$. Then the following are equivalent:
\begin{enumerate}
    \item[(i)] $V$ is a fusion frame with fusion frame bounds $A_V \leq B_V$.
    \item[(ii)] $A_V\cdot \mathcal{I}_{\Hil} \leq S_V \leq B_V\cdot \mathcal{I}_{\Hil}$ (in the sense of positive operators).
\end{enumerate}
\end{corollary}

\begin{corollary}\label{tightfusionchar} \cite{cakuli08}
Let $V=(V_i, v_i)_{i\in I}$ be a weighted sequence of closed subspaces of $\Hil$. Then the following hold.
\begin{enumerate}
    \item[(i)] $V$ is an $A_V$-tight fusion frame if and only if $S_V = A_V\cdot \mathcal{I}_{\Hil}$.
    \item[(ii)] $V$ is a Parseval fusion frame if and only if $S_V = \mathcal{I}_{\Hil}$.
\end{enumerate}
In particular, for an $A_V$-tight fusion frame we have 
\begin{equation}\label{fusionframereconstruction2}
f = \frac{1}{A_V} \sum_{i\in I} v_i ^2 \pi_{V_i} f \qquad (\forall f\in \Hil),
\end{equation}
and for a Parseval fusion frame 
\begin{equation}\label{fusionframereconstruction3}
f = \sum_{i\in I} v_i ^2 \pi_{V_i} f \qquad (\forall f\in \Hil).
\end{equation}
\end{corollary}

Similar to Theorem \ref{synthesisthm}, we can characterize fusion frames in terms of their associated canonical operators. The crucial observation for this characterization is 
\begin{equation}\label{analysisnorm}
\Vert C_V f\Vert_{\mathcal{K}_{\Hil}^2} ^2 = \sum_{i\in I} v_i^2 \Vert \pi_{V_i} f\Vert^2.
\end{equation}

\begin{theorem}\label{fusionframechar}\cite{cakuli08}
Let $V=(V_i, v_i)_{i\in I}$ be a weighted sequence of closed subspaces of $\Hil$. Then the following are equivalent:
\begin{enumerate}
    \item[(i)] $V$ is a fusion frame for $\Hil$.
    \item[(ii)] The synthesis operator $D_V$ is well-defined on $\mathcal{K}^2_{\Hil}$, bounded and surjective.
    \item[(iii)] The analysis operator $C_V$ is well-defined on $\Hil$, bounded, injective and has closed range.
\end{enumerate}
\end{theorem}

\begin{proof}
The equivalence (ii) $\Leftrightarrow$ (iii) is generally true for a bounded operator between Hilbert spaces and its adjoint \cite{conw1}. All conditions imply boundedness, as we will see in the following.

(i) $\Rightarrow$ (ii) If $V$ is a fusion frame, then $D_V \in \mathcal{B}(\mathcal{K}_{\Hil}^2, \Hil )$ by Theorem \ref{synthesisthm}. If $D_V$ was not surjective, then there would exist some non-zero $h \in \mathcal{R}(D_V)^{\perp} = \mathcal{N}(D_V^*) = \mathcal{N}(C_V)$. However, inserting this $h$ into equation (\ref{analysisnorm}) and the lower fusion frame inequality (\ref{fusframeinequ}) would lead to a contradiction. 

(iii) $\Rightarrow$ (i) If (ii) and (iii) hold, then 
$C_V$ is bounded below on $\Hil$ \cite{conw1} and also bounded on $\Hil$ (by Lemma \ref{adjointlemma}). This means that the fusion frame inequalities are satisfied for all $f\in \Hil$. 
\end{proof}

As we have seen above, the frame-related operators of fusion frames share many properties with the frame-related operators associated to frames (see Theorem \ref{framechar}). Analogously as in classical frame theory, there are $-$ of course $-$ other characterizations of weighted families of closed subspaces being a fusion frame via these operators. For instance, the following result gives a characterization of fusion frames and their frame bounds via their synthesis operator.

\begin{proposition}\label{L Caracterizacion fusion frames}\cite{HeinekenMorillas18}
Let $V=(V_i, v_i)_{i\in I}$ be a weighted family of closed subspaces in $\Hil$. Then $V$ is a fusion frame for $\mathcal{H}$ with
bounds $A_V \leq B_V$ if and only if the following two conditions are
satisfied:
\begin{enumerate}
\item[(i)] $(V_i)_{i\in I}$ is complete, i.e. $\overline{\text{span}}(V_i)_{i\in I} = \Hil$. 
\item[(ii)] The synthesis operator $D_{V}$ is well defined on $\mathcal{K}_{\Hil}^2$ and
\begin{equation}\label{NDVperpineq}
    A_V \Vert (f_i)_{i\in I}\Vert ^2 \leq \Vert D_{V}(f_i)_{i\in I}\Vert ^2\leq B_V \Vert (f_i)_{i\in I}\Vert ^2 \quad (\forall(f_i)_{i\in I}\in \mathcal{N}(D_{V})^{\perp}).
\end{equation}
\end{enumerate}
\end{proposition}

\begin{proof}
First assume that $V$ is a fusion frame with fusion frame bounds $A_V \leq B_V$. By Remark \ref{rank1fusion} (d), $(V_i)_{i\in I}$ is complete and Theorem \ref{fusionframechar} implies that the upper inequality in (\ref{NDVperpineq}) is satisfied and that $\mathcal{R}(C_V)$ is closed, since $\mathcal{R}(D_V)$ is closed. The latter implies that $\mathcal{N}(D_V)^{\perp} = \overline{\mathcal{R}(C_V)} = \mathcal{R}(C_V)$, i.e. $\mathcal{N}(D_V)^{\perp}$ consists all of families of the form $(v_i \pi_{V_i}f)_{i\in I}$, where $f\in \Hil$. Now, for arbitrary $f\in \Hil$, we see that
$$\Big( \sum_{i\in I} v_i^2 \Vert \pi_{V_i} f\Vert^2 \Big)^2 = \vert \langle S_V f, f \rangle \vert^2 \leq \Vert S_V f \Vert^2 \Vert f \Vert^2 \leq \Vert S_V f \Vert^2 \frac{1}{A_V} \sum_{i\in I} v_i^2 \Vert \pi_{V_i}f \Vert^2.$$
This implies that
$$A_V \sum_{i\in I} v_i^2 \Vert \pi_{V_i}f \Vert^2 \leq \Vert S_V f \Vert^2 = \Vert D_V (v_i \pi_{V_i} f)_{i\in I} \Vert^2,$$
as desired.

Conversely, assume that (i) and (ii) are satisfied. Since $\mathcal{K}_{\Hil}^2 = \mathcal{N}(D_V) \oplus \mathcal{N}(D_V)^{\perp}$, (ii) implies that the synthesis operator $D_V$ is norm-bounded on $\mathcal{K}_{\Hil}^2$ by $B_V$. Thus, by Theorem \ref{synthesisthm}, $V$ is a Bessel fusion sequence with Bessel fusion bound $B_V$ and it only remains to show the lower fusion frame inequality with respect to $A_V$. We first show that $\mathcal{R}(D_V) = \Hil$. By the validity of (i) and $\text{span}(V_i)_{i\in I} \subseteq \mathcal{R}(D_V) \subseteq \Hil$, the claim follows if we show that $\mathcal{R}(D_V)$ is closed. To this end, let $(g_n)_{n=1}^{\infty}$ be a Cauchy sequence in $\mathcal{R}(D_V)$. Note that $(g_n)_{n=1}^{\infty} = (D_V h_n)_{n=1}^{\infty}$, where $(h_n)_{n=1}^{\infty}$ is a sequence in $\mathcal{N}(D_V)^{\perp}$, and the lower inequality in (\ref{NDVperpineq}) implies that $(h_n)_{n=1}^{\infty}$ is a Cauchy sequence in $\mathcal{N}(D_V)^{\perp}$. Since $\mathcal{N}(D_V)^{\perp}$ is a closed subspace of $\mathcal{K}_{\Hil}^2$, $(h_n)_{n=1}^{\infty}$ converges to some $h\in \mathcal{N}(D_V)^{\perp}$. By continuity of $D_V$, this means that $(g_n)_{n=1}^{\infty}$ converges to $D_V h$ in $\mathcal{R}(D_V)$. Thus $\mathcal{R}(D_V)$ is closed and $\mathcal{R}(D_V) = \Hil$ follows. Note that due to Theorem \ref{fusionframechar} we can already deduce that $V$ is a fusion frame. In order to establish that $A_V$ is a lower fusion frame bound, recall from (\ref{pseudoinverse2}) that $D_V^{\dagger} D_V = \pi_{\mathcal{N}(D_V)^{\perp}}$ and $D_V D_V^{\dagger} = \pi_{\mathcal{R}(D_V)} = \mathcal{I}_{\Hil}$. Hence, (ii) implies that for any $(f_i)_{i\in I} \in \mathcal{K}_{\Hil}^2$ we can estimate
$$A\Vert D_V^{\dagger} D_V (f_i)_{i\in I} \Vert^2 \leq \Vert D_V D_V^{\dagger} D_V (f_i)_{i\in I} \Vert^2 = \Vert D_V (f_i)_{i\in I} \Vert^2.$$
Since also $\mathcal{N}(D_V^{\dagger}) = \mathcal{R}(D_V)^{\perp} = \lbrace 0 \rbrace$, the latter implies that $\Vert D_V^{\dagger} \Vert^2 \leq A_V^{-1}$. In particular, this yields $\Vert (D_V^*)^{\dagger} \Vert^2 = \Vert C_V^{\dagger} \Vert^2 \leq A_V^{-1}$. Since $C_V^{\dagger} C_V$ is the orthogonal projection onto $\mathcal{R}(C_V^{\dagger}) = \mathcal{R}((D_V^{\dagger})^*) = \mathcal{N}(D_V^{\dagger})^{\perp} = \mathcal{R}(D_V) = \Hil$, we finally see that for arbitrary $f\in \Hil$ it holds
$$\Vert f \Vert^2 = \Vert C_V^{\dagger} C_V f \Vert^2 \leq \frac{1}{A_V} \Vert C_V f \Vert^2 = \frac{1}{A_V} \sum_{i\in I} v_i^2 \Vert \pi_{V_i} f\Vert^2,$$
i.e. $A_V$ is a lower fusion frame bound for $V$.
\end{proof}
\subsection{Operators between Hilbert direct sums}\label{Operators between Hilbert direct sums}

Before we proceed with the next section, we provide an excursion into the realm of linear operators between Hilbert direct sums and their matrix representations. The results presented here mainly serve as a preparation for several proof details in the subsequent sections. Thus, the reader who is only interested in the main results for fusion frames, may skip this subsection. 

\

Instead of considering spaces of the type $\mathcal{K}_V^2 = \big( \sum_{i\in I} \oplus V_i \big)_{\ell^2}$, where each $V_i$ is a closed subspace of a given Hilbert space $\Hil$, we may consider the slightly more general case, where $V_i$ are arbitrary Hilbert spaces, and set
\begin{flalign}\label{HDS2}
\Big( \sum_{i\in I} \oplus V_i \Big) _{\ell ^2} := \Big\lbrace (f_i )_{i\in I} : \, f_i \in V_i , \sum_{i\in I}\Vert f_i \Vert_{V_i} ^2 <\infty \Big\rbrace .
\end{flalign}
Again, for $f= ( f_i )_{i\in I}$, $g=(g_i )_{i\in I} \in ( \sum_{i\in I} \oplus V_i ) _{\ell ^2}$, the operation 
\begin{equation}\label{innerproduct2}
\langle f, g\rangle _{( \sum_{i\in I} \oplus V_i ) _{\ell ^2}} := \sum_{i\in I} \langle f_i ,g_i \rangle_{V_i}
\end{equation}
defines an inner product with induced norm  
\begin{equation*}\label{norm2}
\Vert f \Vert_{( \sum_{i\in I} \oplus V_i ) _{\ell ^2}} = \big( \sum_{i\in I} \Vert f_i \Vert_{V_i}^2 \big) ^{1/2}
\end{equation*}
and $( \sum_{i\in I} \oplus V_i ) _{\ell ^2}$ is complete with respect to this norm. 

Of course, definition (\ref{HDS2}) extends to the case where the spaces $V_i$s are only pre-Hilbert spaces, resulting in $( \sum_{i\in I} \oplus V_i ) _{\ell ^2}$ equipped with the inner product (\ref{innerproduct2}) 
being only a pre-Hilbert 
space. If and only if all $V_i$s are Hilbert spaces, the space $( \sum_{i\in I} \oplus V_i ) _{\ell ^2}$ is a Hilbert space, in which case we call it a \emph{Hilbert direct sum}
of the $(V_i)_{i\in I}$. In particular, since subspaces of Hilbert spaces are closed if and only if they are complete, we obtain the following corollary.

\begin{lemma}\label{complete2}\cite{koeba23}
Consider the Hilbert direct sum $\big( \sum_{i\in I} \oplus V_i \big) _{\ell ^2}$ and for each $i\in I$, let $U_i$ be a subspace of $V_i$. Then $\big( \sum_{i\in I} \oplus U_i \big) _{\ell ^2}$ is a closed subspace of $\big( \sum_{i\in I} \oplus V_i \big) _{\ell ^2}$ if and only if $U_i$ is a closed subspace of $V_i$ for every $i\in I$. 
\end{lemma}

Reminiscent of the previous result, we can prove the following result on orthogonal complements, see also \cite{koeba23}.

\begin{lemma}\label{complete3}\cite{koeba23}
Consider the Hilbert direct sum $\big( \sum_{i\in I} \oplus V_i \big) _{\ell ^2}$ and the (not necessarily closed) subspaces $U_i$ of $V_i$ ($i\in I$) and $\big( \sum_{i\in I} \oplus U_i \big) _{\ell ^2}$ of $\big( \sum_{i\in I} \oplus V_i \big) _{\ell ^2}$. Then
\begin{flalign}
\bigg( \big( \sum_{i\in I} \oplus U_i \big) _{\ell ^2} \bigg) ^{\perp} = \big( \sum_{i\in I} \oplus U_i ^{\perp} \big) _{\ell ^2}. \notag
\end{flalign}
\end{lemma}

\begin{proof}
To show the "$\supseteq$"-part, choose $g = ( g_i )_{i\in I} \in \big( \sum_{i\in I} \oplus U_i ^{\perp} \big) _{\ell ^2}$ and observe that for any $f = ( f_i ) _{i\in I} \in \big( \sum_{i\in I} \oplus U_i \big) _{\ell ^2}$ we have $\langle f, g \rangle_{( \sum_{i\in I} \oplus V_i ) _{\ell ^2}} = \sum_{i\in I} \langle f_i , g_i \rangle _{V_i} = 0$, implying $g \in \left( ( \sum_{i\in I} \oplus U_i ) _{\ell ^2} \right) ^{\perp}$.

\noindent To show the "$\subseteq$"-part, we have to show that if $h = ( h_i ) _{i\in I} \in \big( ( \sum_{i\in I} \oplus U_i ) _{\ell ^2} \big) ^{\perp}$ then $h_i \in U_i ^{\perp}$  for every $i\in I$. Since $\langle f, h \rangle_{( \sum_{i\in I} \oplus V_i ) _{\ell ^2}} = 0$ holds for all $f \in \big( \sum_{i\in I} \oplus U_i \big) _{\ell ^2}$, this is particularly true for $f = (..., 0, 0, f_i , 0, 0, ...) \in \big( \sum_{i\in I} \oplus U_i \big) _{\ell ^2}$ (where $f_i \in U_i$ is in the $i$-th entry) and hence we see that in this case $0 = \langle f , h \rangle_{( \sum_{i\in I} \oplus V_i ) _{\ell ^2}} = \langle f_i , h_i \rangle _{V_i}$, which implies $h_i \in U_i ^{\perp}$. Since this holds for all $i\in I$, the proof is finished.
\end{proof}

Next, we show how orthonormal bases for Hilbert direct sums can naturally be constructed from local orthonormal bases:

\begin{lemma}\label{ONBHDS}\cite{kohl21}
For every $i\in I$, let $(e_{ij})_{j\in J_i}$ be an orthonormal basis for $V_i$. Set 
\begin{equation}
\tilde{e}_{ij} = (\delta_{ik} e_{kj} )_{k\in I} = (..., 0, 0, e_{ij}, 0, 0, ...) \, \, \text{($e_{ij}$ in the $i$-th component)}. \notag
\end{equation}
Then $( \tilde{e}_{ij} )_{i\in I, j\in J_i}$ is an orthonormal basis for $\big( \sum_{i\in I} \oplus V_i \big) _{\ell ^2}$.
\end{lemma}

\begin{proof}
It suffices to show that $(\tilde{e}_{ij})_{i\in I, j\in J_i}$ is a complete orthonormal system. Take two arbitrary elemts $\tilde{e}_{ij}, \tilde{e}_{i'j'}$: In case $i\neq i'$ we immediately see that $\langle \tilde{e}_{ij}, \tilde{e}_{i'j'} \rangle_{( \sum_{i\in I} \oplus V_i) _{\ell ^2}} = 0$ by the definition of $\langle . , . \rangle_{( \sum_{i\in I} \oplus V_i) _{\ell ^2}}$ and the elements $\tilde{e}_{ij}$. In case $i=i'$ we have  $\langle \tilde{e}_{ij}, \tilde{e}_{ij'} \rangle_{( \sum_{i\in I} \oplus V_i) _{\ell ^2}} = \langle e_{ij}, e_{ij'} \rangle_{V_i} = \delta_{jj'}$. Thus $(\tilde{e}_{ij})_{i\in I, j\in J_i}$ is an orthonormal system. To see that it is complete, observe that for any $( f_i )_{i\in I} \in \big(\sum_{i\in I} \oplus V_i \big)_{\ell ^2}$ we have
$$(f_i)_{i\in I} = \sum_{i\in I} (\delta_{ik}f_k)_{k\in I} = \sum_{i\in I} \Big(\delta_{ik} \sum_{j\in J_k} \langle f_k, e_{kj} \rangle e_{kj}\Big)_{k\in I} = \sum_{i\in I} \sum_{j\in J_i} \langle f_i, e_{ij} \rangle \tilde{e}_{ij}$$
and the proof is complete.
\end{proof}

Recall the notation 
$$\mathcal{K}_V^2 := \Big(\sum_{j \in J} \oplus V_j \Big)_{\ell ^2} \quad \text{and} \quad \mathcal{K}_W^2 := \Big(\sum_{i \in I} \oplus W_i \Big)_{\ell ^2},$$ where $(V_j)_{j\in J}$ and $(W_i)_{i\in I}$ are families of arbitrary Hilbert spaces associated to the (possibly different) index sets $I$ and $J$. Assume that for all $i\in I$, $j \in J$ we are given linear operators $Q_{ij}:\text{dom}(Q_{ij}) \subseteq V_j \longrightarrow W_i$, and let $f = ( f_j )_{j\in J} \in (\text{dom}(Q_{ij}))_{j\in J} \subseteq (V_j)_{j\in J}$. Then 
\begin{equation}\label{def-Of}
Qf := \Big( \sum_{j \in J} Q_{ij} f_j \Big)_{i\in I}
\end{equation}
defines a linear operator $Q: \text{dom}(Q) \subseteq \mathcal{K}_V^2 \longrightarrow \mathcal{K}_W^2$, where 
\begin{equation}\label{domO}
\text{dom}(Q) = \Big\lbrace (f_j )_{j\in J} \in (\text{dom}(Q_{ij}))_{j\in J} : \sum_{i \in I} \Big\Vert \sum_{j \in J} Q_{ij} f_j \Big\Vert_{W_i} ^2 < \infty \Big\rbrace .
\end{equation}
If we view elements $f = (f_j)_{j\in J}$ from $\mathcal{K}_V^2$ as (possibly infinite) column vectors, then (by assuming $I=J=\mathbb{N}$ without loss of generality) we can represent the operator $Q$ by the matrix
\begin{equation}\label{matrix}
Q = \begin{bmatrix} 
    Q_{11} & Q_{12} & Q_{13} & \dots \\
    Q_{21} & Q_{22} & Q_{23} & \dots \\
    Q_{31} & Q_{32} & Q_{33} & \dots \\
    \vdots & \vdots & \vdots & \ddots & \\ 
    \end{bmatrix} ,
\end{equation}
because (\ref{def-Of}) precisely corresponds to the formal matrix multiplication of $Q$ with $f$. 

Our goal is to show that there is a one-to-one correspondence between all bounded operators $Q\in \mathcal{B}(\mathcal{K}_V^2 , \mathcal{K}_W^2)$ and all matrices $[Q_{ij}]_{j\in J, i\in I}$ of bounded operators $Q_{ij} \in \mathcal{B}(V_j, W_i)$ with $\text{dom}\big( [Q_{ij}]_{j\in J, i\in I} \big) = \mathcal{K}_V^2$ (where the domain is defined as in \ref{domO}).

For that purpose, we define the coordinate functions 
$$\mathcal{C}_{V,k} : \mathcal{K}_V^2 \longrightarrow V_k , \qquad \mathcal{C}_{V,k} (f_j )_{j \in J} := f_k \qquad (k\in J).$$
Clearly, each operator $\mathcal{C}_{V,k}$ is norm bounded by $1$ and hence possesses a uniquely determined adjoint operator $\mathcal{C}_{V,k}^* : V_k \longrightarrow \mathcal{K}_V^2$. For $f_k\in V_k$ and $(g_j)_{j\in J} \in \mathcal{K}_V^2$, the computation 
\begin{align}
\big\langle \mathcal{C}_{V,k}^* f_k , (g_j)_{j\in J} \rangle _{\mathcal{K}_V^2} &= \langle f_k , \mathcal{C}_{V,k} (g_j )_{j\in J} \rangle _{V_k} \notag \\
&= \langle f_k, g_k \rangle _{V_k} \notag \\
&= \langle (..., 0, 0, f_k, 0, 0, ...) ,  (g_j )_{j\in J} \rangle _{\mathcal{K}_V^2} \notag
\end{align}
shows that $\mathcal{C}_{V,k}^*$ maps $f_k \in V_k$ onto $(..., 0, 0, f_k, 0, 0, ...) \in \mathcal{K}_V^2$ ($f_k$ in the $k$-th component). 
Analogously, we define 
$$\mathcal{C}_{W,l} : \mathcal{K}_W^2 \longrightarrow W_l , \qquad \mathcal{C}_{W,l} (g_i)_{i \in I} := g_l \qquad (l\in I)$$
with $\Vert \mathcal{C}_{W,l} \Vert = 1$ and $\mathcal{C}_{W,l}^* : W_l \longrightarrow \mathcal{K}_W^2$ given by $\mathcal{C}_{W,l}^* g_l = (0, ..., 0, g_l, 0, 0, ...)$ ($g_l$ in the $l$-th component).

\begin{proposition}\label{LemmaBounded2}\cite{kohl21}
Let $Q \in \mathcal{B}(\mathcal{K}_V^2 , \mathcal{K}_W^2 )$ be bounded. Then there exists a uniquely determined matrix $[Q_{ij}]_{j\in J, i\in I}$ of bounded operators $Q_{ij} \in \mathcal{B}(V_j, W_i)$, such that the action of $Q$ on any $(f_j)_{j \in J} \in \mathcal{K}_V^2$ is given by the formal matrix multiplication of $(Q_{ij})_{j\in J, i\in I}$ with $(f_j)_{j \in J}$, i.e.
$$ Q (f_j)_{j \in J} = \Big( \sum _{j\in J} Q_{ij}f_j \Big)_{i \in I}.$$
\end{proposition}

\begin{proof}
For each $f = (f_j )_{j \in J} \in \mathcal{K}_V^2$ we have $f = \sum_{j\in J} \mathcal{C}_{V,j}^* f_j$ and thus $Qf = \sum_{j\in J} Q \mathcal{C}_{V,j}^* f_j$ by the boundedness of $Q$.  Furthermore, since $Qf \in \mathcal{K}_W^2$, we may write $Qf = ( [Qf]_i )_{i\in I}$, where $[Qf]_i\in W_i$ denotes the $i$-th component of $Qf$. By definition of the coordinate functions $\mathcal{C}_{W,i}$ we obtain
\begin{flalign}
    Qf = ( [Qf]_i )_{i\in I} &= (\mathcal{C}_{W,i} Qf )_{i\in I} \notag \\
    &= \Big( \mathcal{C}_{W,i} \sum_{j\in J} Q \mathcal{C}_{V,j}^* f_j  \Big)_{i\in I} = \Big( \sum_{j\in J} \mathcal{C}_{W,i} Q \mathcal{C}_{V,j}^* f_j  \Big)_{i\in I}. \notag
\end{flalign}
We set $Q_{ij} := \mathcal{C}_{W,i} Q \mathcal{C}_{V,j}^* \in \mathcal{B}(V_j , W_i )$ and observe that $Q$ acts on $f= (f_j )_{j \in J} \in \mathcal{K}_V^2$ in the same way as the matrix $[Q_{ij}]_{j\in J, i\in I}$ does in (\ref{def-Of}), when considering $f$ as column vector. The uniqueness of the matrix $[Q_{ij}]_{j\in J, i\in I}$ follows from its construction. 
\end{proof}

Note that we can also define a matrix representation of operators using fusion frames as a generalization of the ideas in the above proofs, where not the 'canonical' operators $\mathcal{C}_{W,i}$ and $\mathcal{C}_{V,j}$ are used, but the fusion frame-related analysis and synthesis operators are, see \cite{xxlcharshaare18}. 

\

Next we show the converse of the above result. Its proof relies heavily on the Uniform Boundedness Principle. In order to write down the proof properly, we make use of the following projection operators defined below. Note that here and in the subsequent proof, we may assume, without loss of generality, that $I = J = \mathbb{N}$. For $m,n \in \mathbb{N}$ we define
\begin{flalign}
    P_{\langle n \rangle}^V : \mathcal{K}_V^2 \longrightarrow \mathcal{K}_V^2 , \qquad &P_{\langle n \rangle}^V(f_j)_{j \in J} = (f_1 , ... , f_n , 0, 0, ...), \notag \\
    P_{\langle m \rangle}^W : \mathcal{K}_W^2 \longrightarrow \mathcal{K}_W^2 , \qquad &P_{\langle m \rangle}^W (g_i)_{i \in I} = (g_1 , ... , g_m , 0, 0, ...). \notag
\end{flalign}
It is easy to verify that $P_{\langle n \rangle}^V$ and $P_{\langle m \rangle}^W$ are orthogonal projections ($m,n \in \mathbb{N}$).

\begin{proposition}\label{LemmaBounded}\cite{kohl21}
Let $Q_{ij} \in \mathcal{B}(V_j , W_i)$ for all $j\in I, i\in I$, and define $Q$ as in (\ref{def-Of}). If $\text{dom}(Q) = \mathcal{K}_V^2$, then $Q \in \mathcal{B}(\mathcal{K}_V^2 , \mathcal{K}_W^2)$.
\end{proposition}

\begin{proof}
The assumption $\text{dom}(Q) = \mathcal{K}_V^2$ means that $Q$ is a well-defined linear operator from $\mathcal{K}_V^2$ into $\mathcal{K}_W^2$. Consider the sub-matrices $T_m = P_{\langle m \rangle}^W Q: \mathcal{K}_V^2 \longrightarrow \mathcal{K}_W^2$. If we can show that each $T_m$ is bounded and that the operators $T_m$ converge pointwise to $Q$ as $m\longrightarrow \infty$, then, by the Uniform Boundedness Principle, this implies $Q \in \mathcal{B}(\mathcal{K}_V^2 , \mathcal{K}_W^2)$.

To see that the operators $T_m$ converge pointwise to $Q$, recall that our assumption $\text{dom}(Q) = \mathcal{K}_V^2$ implies that for all $f \in \mathcal{K}_V^2$ we have $Qf\in \mathcal{K}_W^2$, which means that $\Vert Qf \Vert_{\mathcal{K}_W^2}^2 = \sum_{i=1}^{\infty} \Vert \sum_{j=1}^{\infty} Q_{ij} f_j \Vert_{W_i}^2 < \infty$. In particular, this implies
\begin{flalign}
\lim_{m\longrightarrow \infty} \Vert (Q - T_m)f \Vert_{\mathcal{K}_W^2}^2 &= \lim_{m\longrightarrow \infty} \Vert (\mathcal{I}_{\mathcal{K}_W^2} - P_{\langle m \rangle}^W) Qf \Vert_{\mathcal{K}_W^2}^2 \notag \\
&=  \lim_{m\longrightarrow \infty}  \sum_{i=m+1}^{\infty} \Big\Vert \sum_{j=1}^{\infty} Q_{ij} f_j \Big\Vert_{W_i}^2 = 0. \notag
\end{flalign}

It remains to show that $T_m \in \mathcal{B}(\mathcal{K}_V^2, \mathcal{K}_W^2)$ for each $m$. We will use the Uniform Boundedness Principle once again to see that this is true. To this end, we fix an arbitrary $m\in \mathbb{N}$ and consider the sub-matrices $T_m^{(n)} = P_{\langle m \rangle}^W Q P_{\langle n \rangle}^V: \mathcal{K}_V^2 \longrightarrow \mathcal{K}_W^2$ of $T_m$. To see that $T_m^{(n)}$ is bounded for each $n\in \mathbb{N}$,  observe that we have
\begin{flalign}
\Vert P_{\langle m \rangle}^W Q P_{\langle n \rangle}^V f \Vert_{\mathcal{K}_W^2}^2 &= \sum_{i=1}^K \Big\Vert \sum_{j=1}^n Q_{ij} f_j \Big\Vert_{W_i}^2 \notag \\
&\leq \sum_{i=1} ^K \Big( \sum_{j=1} ^n \Vert Q_{ij} \Vert_{V_j \rightarrow W_i} \, \Vert f_j \Vert_{V_j} \Big) ^2 \notag \\
&\leq \sum_{i=1} ^K \Big( \sum_{j=1} ^n \Vert Q_{ij} \Vert_{V_j \rightarrow W_i} ^2 \Big) \Big( \sum_{j=1} ^n \Vert f_j \Vert_{V_j} ^2 \Big) , \notag 
\end{flalign}
where we used the Cauchy Schwartz inequality. This implies 
$$\Vert P_{\langle m \rangle}^W Q P_{\langle n \rangle}^V \Vert_{\mathcal{K}_V^2 \rightarrow \mathcal{K}_W^2} ^2 \leq mn \sup_{1\leq i \leq m, 1 \leq j \leq n} \Vert Q_{ij} \Vert_{V_j \rightarrow W_i} ^2 <\infty .$$
What remains to be proven is that $T_m^{(n)} \longrightarrow T_m$ pointwise (as $n\longrightarrow \infty$), i.e. that for any $f \in \mathcal{K}_V^2$ we have \begin{equation}\label{pointwise}
\lim_{n\longrightarrow \infty} \big\Vert (T_m - T_m^{(n)})f \big\Vert_{\mathcal{K}_W^2} = \lim_{n\longrightarrow \infty} \big\Vert P_{\langle m \rangle}^W Q (\mathcal{I}_{\mathcal{K}_V^2} - P_{\langle n \rangle}^V)f \big\Vert_{\mathcal{K}_W^2} = 0 .
\end{equation}
To this end, fix an arbitrary 
$f = (f_j) \in \mathcal{K}_V^2$ and define
$$g^{(i)}_n := \sum_{j=1}^n Q_{ij}f_j.$$
Then clearly $g^{(i)}_n \in W_i$ for every $n\in \mathbb{N}$ and every $i$. By construction we have $\lim_{n\longrightarrow \infty} g^{(i)}_n = \sum_{j=1}^{\infty} Q_{ij} f_j \in W_i$, which follows from $Qf \in \mathcal{K}_W^2$. Therefore, for every $i$, the sequence $\lbrace x_n^{(i)} \rbrace_{n\in \mathbb{N}}$, defined by $x_n^{(i)} = \Vert \sum_{j=n}^{\infty} Q_{ij} f_j \Vert_{W_i}^2$ is a convergent sequence in $\mathbb{R}$ with limit $0$. This implies 
\begin{flalign}
\lim_{n\longrightarrow \infty} \big\Vert P_{\langle m \rangle}^W Q (\mathcal{I}_{\mathcal{K}_V^2} - P_{\langle n \rangle}^V)f \big\Vert_{\mathcal{K}_W^2} &= \lim_{n \rightarrow \infty} \sum_{i=1}^{K} \Big\Vert \sum_{j=n+1}^{\infty} Q_{ij} f_j \Big\Vert_{W_i} ^2 \notag \\ 
&= \sum_{i=1}^{K} \lim_{n \rightarrow \infty} \Big\Vert \sum_{j=n+1}^{\infty} Q_{ij} f_j \Big\Vert_{W_i} ^2 = 0 \notag
\end{flalign}
and (\ref{pointwise}) is proven.
\end{proof}

We may summarize Propositions \ref{LemmaBounded2} and \ref{LemmaBounded} as follows:

\begin{svgraybox}
    Every bounded operator $Q \in \mathcal{B}(\mathcal{K}_V^2 , \mathcal{K}_W^2)$ can be naturally identified with a matrix $[Q_{ij}]_{j\in J, i\in I}$ of bounded operators $Q_i \in \mathcal{B}(V_j, W_i)$, which acts on $f = (f_j )_{j\in J}$ via matrix multiplication (as in (\ref{def-Of})), and which additionally satisfies $\text{dom}(Q) = \mathcal{K}_V^2$ (compare to (\ref{domO})).
\end{svgraybox}

This motivates us to use the notation
$$\mathbb{M}(Q) = [Q_{ij}]_{j\in J, i\in I}$$
and call $\mathbb{M}(Q) = [Q_{ij}]_{j\in J, i\in I}$ the \emph{canonical matrix representation} of $Q$.

\begin{example}{Examples}\label{matrixDC}
Assume that $V = (V_i, v_i)_{i\in I}$ is a Bessel fusion sequence for $\Hil$. Then, by Theorem \ref{synthesisthm}, $D_V \in \mathcal{B}(\mathcal{K}_{\Hil}^2, \Hil)$ and $C_V \in \mathcal{B}(\Hil , \mathcal{K}_{\Hil}^2)$ are bounded operators. Their respective  canonical matrix representations are given by 
\begin{flalign}
\mathbb{M}(D_V) &= \begin{bmatrix} 
     \dots & v_{i-1} \pi_{V_{i-1}} & 
     v_i \pi_{V_i} &  v_{i+1} \pi_{V_{i+1}} & \dots \end{bmatrix} , \notag \\
     \mathbb{M}(C_V) &= \begin{bmatrix} 
     \dots & v_{i-1} \pi_{V_{i-1}} & 
     v_i \pi_{V_i} &  v_{i+1} \pi_{V_{i+1}} & \dots \end{bmatrix}^T ,\notag
\end{flalign}
where the exponent $T$ denotes matrix transposition, as usual.
\end{example}

\noindent \textbf{Note:} For the remainder of this subsection, we assume that $\mathcal{K}_V^2$ and $\mathcal{K}_W^2$ are Hilbert direct sums indexed by the \emph{same} index set $I$.

\

One particularly interesting subclass of operators between Hilbert direct sums are the \emph{block-diagonal} operators (also known as \emph{direct sums of operators} \cite{conw1}) between Hilbert direct sums associated to the same index set $I$. For given operators $Q_i : \text{dom}(Q_i) \subseteq V_i \longrightarrow W_i$, the \emph{block-diagonal} operator $\bigoplus_{i\in I} Q_i$ is defined by 
\begin{equation}\label{blockdiagonaldef}
\bigoplus_{i\in I} Q_i : \text{dom}\Big(\bigoplus_{i\in I} Q_i \Big) \subseteq \mathcal{K}_V^2 \longrightarrow \mathcal{K}_W^2, \qquad \bigoplus_{i\in I} Q_i (f_i)_{i\in I} := (Q_i f_i)_{i\in I}.
\end{equation}
Here in this manuscript we only consider the case, where each $Q_i \in \mathcal{B}(V_i, W_i)$ is bounded, see Section \ref{Duality in Fusion Frame Theory}. 

By their definition, bounded block-diagonal operators are precisely those bounded operators in $\mathcal{B}(\mathcal{K}_V^2,\mathcal{K}_W^2)$, whose canonical matrix representation is given by a diagonal matrix, i.e. $\mathbb{M}(\bigoplus_{i\in I} Q_i) = \text{diag}(Q_i)_{i\in I}$. The next lemma gives an easy characterization for block-diagonal operators $\bigoplus_{i\in I} Q_i$ associated to $Q_i \in \mathcal{B}(V_i, W_i)$ ($i\in I$) being bounded. 

\begin{lemma}\label{blockdiagonalbounded}
Assume that we are given bounded operators $Q_i\in \mathcal{B}(V_i, W_i)$ ($i\in I$). Then $\bigoplus_{i\in I} Q_i$, defined as in (\ref{blockdiagonaldef}), is a bounded operator $\bigoplus_{i\in I} Q_i \in \mathcal{B}(\mathcal{K}_V^2,\mathcal{K}_W^2)$ if and only if $\sup_{i\in I} \Vert Q_i \Vert < \infty$. In that case, $\Vert \bigoplus_{i\in I} Q_i \Vert = \sup_{i\in I} \Vert Q_i \Vert < \infty$.
\end{lemma}

\begin{proof}
This proof merely consists of writing down the definitions and is therefore left to the reader, see also \cite{kohl21} for more details. 
\end{proof}

\begin{remark}
A short computation shows that if $\bigoplus_{i\in I} Q_i \in \mathcal{B}(\mathcal{K}_V^2,\mathcal{K}_W^2)$ is bounded and block-diagonal, then its adjoint $\big( \bigoplus_{i\in I} Q_i \big)^* = \bigoplus_{i\in I} Q_i^* \in \mathcal{B}(\mathcal{K}_W^2,\mathcal{K}_V^2)$ is again block-diagonal \cite{koeba23}.
\end{remark}

In this manuscript, we will also consider a sub-class of bounded block-diagonal operators: We call a bounded block-diagonal operator $\bigoplus_{i\in I} Q_i \in \mathcal{B}(\mathcal{K}_V^2,\mathcal{K}_W^2)$ \emph{component preserving}, if each $Q_i \in \mathcal{B}(V_i, W_i)$ is surjective.   

By considering the self-adjoint operators $M_{V,k}= \mathcal{C}_{V, k}^*
\mathcal{C}_{V, k} \in \mathcal{B}(\mathcal{K}_{V}^{2})$,  given by $M_{V,k}(f_i)_{i\in I}=(\delta_{ki}f_{i})_{i \in I}$ for $k \in I$, and analogously $M_{W,k}= \mathcal{C}_{W, k}^*
\mathcal{C}_{W, k} \in \mathcal{B}(\mathcal{K}_{W}^{2})$, we can state the following characterization of bounded operators $Q \in \mathcal{B}(\mathcal{K}_V^2,\mathcal{K}_W^2)$ being block-diagonal (resp. component preserving). Its proof follows immediately from the uniqueness of the canonical matrix representation of $Q$, see above.

\begin{lemma}\label{blockdiagonallem}
Let $Q \in \mathcal{B}(\mathcal{K}_V^2,\mathcal{K}_W^2)$. Then the
following are equivalent:
\begin{enumerate}
  \item[(i)] $Q$ is block-diagonal.
  \item[(ii)] $Q M_{V,k} = M_{W,k} Q$ for each $k \in I$.
  \item[(iii)] $Q M_{V,k}(\mathcal{K}_V^{2}) \subseteq M_{W,k} (\mathcal{K}_{W}^{2})$ for each $k \in I$.
\end{enumerate}
Furthermore, the following two conditions are equivalent:
\begin{enumerate}
  \item[(iv)] $Q$ is component preserving.
   \item[(v)] $Q M_{V,k}(\mathcal{K}_V^{2}) = M_{W,k} (\mathcal{K}_{W}^{2})$ for each $k \in I$.
\end{enumerate}
\end{lemma}

Note, that for the remainder of this survey, we reconsider the special cases $\mathcal{K}_V^2 = \mathcal{K}_{\Hil}^2$ and $\mathcal{K}_V^2 = \big( \sum_{i\in I} \oplus V_i \big)_{\ell^2}$, where each $V_i$ is a closed subspace of $\Hil$, only. In particular, in Section \ref{Duality in Fusion Frame Theory}, we will not distinguish between $M_{V,k}$ and $M_{W,k}$, and simply write $M_k$ instead.

\section{Other notions in fusion frame theory}\label{Other notions in fusion frame theory}

In this section we discuss further notions in fusion frame theory. 

\begin{definition}\label{DefFRB}\cite{caskut04,Shaarbal}
Let $(V_i)_{i\in I}$ be a family of closed non-zero subspaces of $\Hil$ and let $(v_i)_{i\in I}$ be a family of weights.
\begin{itemize}
    \item $V = (V_i , v_i )_{i\in I}$ is called a \emph{fusion Riesz basis}, if $\overline{\text{span}}(V_i )_{i\in I} = \mathcal{H}$ and if there exist constants $0< \alpha_V \leq \beta_V < \infty $, called \emph{lower} and \emph{upper fusion Riesz bounds} respectively, such that for all finite vector-sequences $(f_i)_{i\in I} \in \mathcal{K}_V^{00}$ it  holds
\begin{equation}\label{fusrieszinequ}
\alpha_V \sum _{i\in I} \Vert f_i \Vert ^2 \leq  \bigg\Vert \sum _{i\in I} v_i f_i \bigg\Vert ^2 \leq \beta_V \sum _{i\in I} \Vert f_i \Vert ^2 .
\end{equation}
\item $(V_i)_{i\in I}$ is called an \emph{orthonormal fusion basis} for $\Hil$, if $\Hil = \bigoplus ^{\perp}_{i\in
I} V_{i}$
\footnote{We use the notation $\bigoplus ^{\perp}_{i\in
I}$ to emphasize that the subspaces are mutually orthogonal, so we consider a very special direct sum.}, i.e. $\mathcal{I}_{\Hil} = \sum_{i\in I} \pi_{V_i}$ and $V_i \perp V_j$ for all $i\neq j \in I$. In other words, an orthonormal fusion basis is a family of pairwise orthogonal subspaces, so that $S_{V^1} = \mathcal{I}_{\Hil}$, where $V^1 = (V_i, 1)_{i\in I}$.  Therefore, 
we always consider uniform weights $v_i = 1$ ($i\in I$) in this case.
\item $(V_i)_{i\in I}$ is called a  \emph{Riesz decomposition} of
$\mathcal{H}$, if for every $f\in \mathcal {H}$ there exists a unique family $(f_i)_{i\in I}$ of vectors $f_i \in V_i$ ($i\in I$), such that $f=\sum_{i\in I}f_{i}$. 
\item $(V_i)_{i\in I}$ is called \emph{minimal}, if for each $i\in I$,
$V_i \cap \overline{\text{span}}(V_j)_{j\in I, j\neq i} =\{0\}.$
\end{itemize}
\end{definition}

At first, we show that every fusion Riesz basis is a fusion frame. This fact follows from the following characterization of fusion Riesz bases via the fusion frame-related operators $D_V$ and $C_V$.

\begin{theorem}\label{fusionrieszbasischar}
Assume that $V = (V_i, v_i)_{i\in I}$ is a weighted sequence of closed subspaces $V_i$ of $\mathcal{H}$. Then the following are equivalent.
\begin{enumerate}
    \item[(i)] $V$ is a fusion Riesz basis for $\mathcal{H}$.
    \item[(ii)] $D_V$ is bounded on $\mathcal{K}_{\Hil}^2$ with $\mathcal{R}(D_V) = \Hil$ and $\mathcal{N}(D_V) = (\mathcal{K}_V^2)^{\perp}$.
    \item[(iii)] $C_V$ is bounded on $\Hil$ with $\mathcal{R}(C_V) = \mathcal{K}_V^2$ and $\mathcal{N}(C_V)= \{ 0 \}$.
\end{enumerate}
In particular, every fusion Riesz basis is a fusion frame with the same bounds.
\end{theorem}

\begin{proof}
First, note that by Lemma \ref{complete3}, the orthogonal complement $(\mathcal{K}_V^2)^{\perp}$ of $\mathcal{K}_V^2$ in $\mathcal{K}_{\Hil}^2$ is given by $\big(\sum_{i\in I} \oplus V_i^{\perp} \big)_{\ell^2}$. By definition of $D_V$, this implies that $D_V = D_V (\pi_{\mathcal{K}_V^2} + \pi_{(\mathcal{K}_V^2)^{\perp}}) = D_V \pi_{\mathcal{K}_V^2}$, and thus $D_V \vert_{\mathcal{K}_V^2} = D_V \pi_{\mathcal{K}_V^2}$ on $\mathcal{K}_V^2$. In particular, statement (ii) is equivalent to $D_V \vert_{\mathcal{K}_V^2} \in \mathcal{B}(\mathcal{K}_V^2, \Hil)$ being bijective and statement (iii) is equivalent to $C_V \in \mathcal{B}(\Hil , \mathcal{K}_V^2)$ being bijective. Moreover, as in the proof of Lemma \ref{adjointlemma}, we can show that $(D_V \vert_{\mathcal{K}_V^2})^* = C_V$ and we have $C_V^* = D_V \vert_{\mathcal{K}_V^2}$ (assuming they are bounded), which stems from the fact, that we always have $\mathcal{R}(C_V) \subseteq \mathcal{K}_V^2$. 

(ii) $\Leftrightarrow$ (iii) This is generally true for a bounded bijective operator between Hilbert spaces and its adjoint \cite{conw1}. As we will see in the following all conditions imply that the operators are bounded.

(i) $\Rightarrow$ (ii) If $V$ is a fusion Riesz basis, then by the right-hand inequality of (\ref{fusrieszinequ}), $D_V$ is bounded on $\mathcal{K}_{\Hil}^{00}$ and thus by density also on $\mathcal{K}_{\Hil}^2$. This implies that $C_V$ is also bounded. If $D_V \vert_{\mathcal{K}_V^2}$ was not surjective, there would exist $0\neq h \in \mathcal{R}(D_V \vert_{\mathcal{K}_V^2})^{\perp} = \mathcal{N}(C_V)$ such that $h \perp V_i$ for all $i\in I$. However, since $(V_i)_{i\in I}$ is complete by assumption, there exists some sequence $(g_n)_{n=1}^{\infty} \in \text{span}(V_i)_{i\in I}$, such that 
$$0 = \lim_{n\rightarrow \infty} \Vert h - g_n \Vert^2 = \lim_{n\rightarrow \infty} (\Vert h \Vert^2 - \langle h, g_n \rangle - \langle g_n, h \rangle + \Vert g_n \Vert^2 ) = 2\Vert h \Vert^2 > 0,$$
a contradiction. Therefore $D_V \vert_{\mathcal{K}_V^2}$ is surjective. In particular, $V$ is a fusion frame by Theorem \ref{fusionframechar} and thus series of the form $D_V (f_i)_{i\in I} = \sum_{i\in I} v_i f_i$ (with $(f_i)_{i\in I} \in \mathcal{K}_V^2$) converge unconditionally. To see that $D_V \vert_{\mathcal{K}_V^2}$ is injective, observe that the left inequality of (\ref{fusrieszinequ}) implies that $D_V \vert_{\mathcal{K}_V^{00}}$ is injective. However, by continuity of $D_V$, density of $\mathcal{K}_V^{00}$ in $\mathcal{K}_V^{2}$, and unconditional convergence of $\sum_{i\in I} v_i f_i$, the left inequality of (\ref{fusrieszinequ}) is true for all sequences $(f_i)_{i\in I} \in \mathcal{K}_V^2$, which implies that $D_V \vert_{\mathcal{K}_V^2}$ is injective.

\noindent $(ii) \Rightarrow (i)$ If $D_V \vert_{\mathcal{K}_V^2} : \mathcal{K}_V^2 \longrightarrow \Hil$ is bounded and bijective, then the inequalities (\ref{fusrieszinequ}) are true for all sequences $(f_i)_{i\in I} \in \mathcal{K}_V^2$ and hence for all finite vector-sequences. If $(V_i)_{i\in I}$ was not complete, then there would exist $0 \neq h \in \Hil$ such that $h \perp \overline{\text{span}}(V_i)_{i\in I}$. In particular, $C_V h = 0$, a contradiction. Thus $V$ is a fusion Riesz basis. 

Finally, it follows from Theorem \ref{fusionframechar}, that every fusion Riesz basis is a fusion frame. In particular, if $V$ is a fusion Riesz basis (and thus a fusion frame), then this means that $C_V$ is bounded from above and below by the respective fusion frame bounds and $D_V$ is bounded from above and below by the respective fusion Riesz basis bounds. Since $D_V C_V$ and $C_V D_V$ have the same non-zero spectrum, we can conclude that the set of fusion frame bounds coincides with the set of fusion Riesz basis bounds.
\end{proof}

Recall from Section \ref{Bessel fusion sequences and fusion frames}, that the weights $(v_i)_{i\in I}$ associated to a fusion frame $V$ are uniformly bounded from above, but not necessarily semi-normalized. However, weights associated to fusion Riesz bases are indeed semi-normalized:

\begin{lemma}\label{weightsFRB}\cite{Shaarbal}
Let $V = (V_i , v_i )_{i\in I}$ be a fusion Riesz basis with fusion Riesz basis constants $\alpha_V$ and $\beta_V$. Then
\begin{enumerate}
    \item[(i)] For all $i\in I$, $\sqrt{\alpha_V} \leq v_i \leq \sqrt{\beta_V}$
    \item[(ii)] For every semi-normalized family $(w_i)_{i\in I}$ of weights, $(V_i, v_i w_i)_{i\in I}$ is a fusion Riesz basis as well.
\end{enumerate}
\end{lemma}

\begin{proof}
(i) Recall that by the definition of a fusion Riesz basis, for any finite vector sequence $(f_j)_{j\in J} \in \mathcal{K}_V^{00}$, we have
$$\alpha_V \sum _{j\in J} \Vert f_j \Vert ^2 \leq  \Big\Vert \sum _{j\in J} v_j f_j \Big\Vert ^2 \leq \beta_V \sum _{j\in J} \Vert f_j \Vert ^2 .$$
Now, choose some arbitrary $i\in I$, set $J = \lbrace i \rbrace$, choose some normalized $g \in V_i$ and apply the above. 

(ii) follows analogously to the proof of Lemma \ref{weights2}.
\end{proof}

Recall from frame theory, that every Parseval frame $(\varphi_i)_{i\in I}$ with $\Vert \varphi_i \Vert = 1$ ($\forall i\in I$) is an orthonormal basis (and vice versa) \cite{ole1}. In the fusion frame setting, the analogue result 
is also true (recall from Remark \ref{rank1fusion} (b), that for $1$-dimensional subspaces $V_i = \text{span}\lbrace \varphi_i \rbrace$, $v_i$ corresponds to $\Vert \varphi_i \Vert$).

\begin{theorem}\label{chaONB}\cite{caskut04}
Let $(V_i)_{i\in I}$ be a family of closed subspaces of $\Hil$. Then the following are equivalent.
\begin{enumerate}
    \item[(i)] $(V_i)_{i\in I}$ is an orthonormal fusion basis. 
    \item[(ii)] $(V_i, 1)_{i\in I}$ is a Parseval fusion frame. 
\end{enumerate}
\end{theorem}

\begin{proof}
(i)$\Rightarrow$(ii) If $(V_i)_{i\in I}$ is an orthonormal fusion basis, then for all $f\in \Hil$, $f = \sum_{i\in I} \pi_{V_i} f$ where the subspaces $V_i$ are mutually orthogonal. In particular, this means, that the fusion frame operator $S_{V^1}$ associated to $V^1 = (V_i, 1)_{i\in I}$ satisfies $S_{V^1} = \mathcal{I}_{\Hil}$. By Corollary \ref{tightfusionchar}, this implies that $V^1$ is a Parseval fusion frame.

(ii)$\Rightarrow$(i) For every $i\in I$, let $(e_{ij})_{j\in J_i}$ be an orthonormal basis for $V_i$. If we assume that $(V_i, 1)_{i\in I}$ is a Parseval fusion frame, then, by Corollary \ref{parsfusion}, the family $(e_{ij})_{i\in I, j\in J_i}$ is a Parseval frame. Since every Parseval frame consisting of norm-$1$ vectors is an orthonormal basis \cite{ole1}, $(e_{ij})_{i\in I, j\in J_i}$ is an orthonormal basis for $\Hil$. This implies that $V_i \perp V_j$ for all $i\neq j$ and that $f = \sum_{i\in I} \pi_{V_i} f$ for all $f\in \Hil$. Thus $(V_i)_{i\in I}$ is an orthonormal fusion basis. 
\end{proof} 

The next proposition is a slightly more general version of \cite[Lemma 4.2]{caskut04}, characterizing minimal families of subspaces using local Riesz bases with a new proof.

\begin{proposition}\label{minimalcha}
Let $(V_i)_{i\in I}$ be a family of closed subspaces in $\Hil$ and assume that $(\varphi_{ij})_{j\in J_i}$ is a Riesz basis for $V_i$ for every $i\in I$. Then the following are equivalent.
\begin{enumerate}
\item[(i)] $(V_i)_{i\in I}$
is minimal.
\item[(ii)] $(\varphi_{ij})_{i\in I, j\in J_i}$ is a minimal sequence in $\Hil$.
\end{enumerate}
\end{proposition}

\begin{proof}
Note that for every subset $K\subset I$, $\overline{\text{span}}(V_i)_{i\in K} = \overline{\text{span}}(\varphi_{ij})_{i\in K, j\in  J_i}$: Indeed, since the families $(\varphi_{ij})_{j\in J_i}$ are local Riesz bases, we can deduce that $\text{span}(V_i)_{i\in K} \subseteq \overline{\text{span}}(\varphi_{ij})_{i\in K, j\in J_i}$ and $\text{span}(\varphi_{ij})_{i\in K, j\in J_i} \subseteq \overline{\text{span}}(V_i)_{i\in I}$. Now take the respective closures.

(ii)$\Rightarrow$(i) If $(V_i)_{i\in I}$ was not minimal, then there would exist some non-zero $g\in \Hil$ with $g\in V_k \cap \overline{\text{span}}(V_i)_{i\in I, i\neq k}$ for some $k\in I$. Because $(\varphi_{kj})_{j\in J_k}$ is a Riesz basis for $V_k$, we can write $g = \sum_{j\in J_k} c_{kj}\varphi_{kj}$ for suitable scalars $c_{kj}$. Since $g\neq 0$, there is a $c_{km}\neq 0$ and we have
 \begin{eqnarray*}
 c_{km}\varphi_{km}& \in &\overline{\text{span}}(\varphi_{ij})_{i\neq k, j\in J_i} - \sum_{j\neq m}c_{kj}\varphi_{kj}\\
 & \subseteq & \overline{\text{span}}(\varphi_{ij})_{i\neq k, j\in J_i} + \overline{\text{span}}(\varphi_{kj})_{j\neq m}
 \\
 & \subseteq & \overline{\text{span}}(\varphi_{ij})_{i\in I,j\in J_i, (i,j)\neq (k,m)} .
 \end{eqnarray*}
 This contradicts (ii). 
 
(i)$\Rightarrow$(ii)  If $(\varphi_{ij})_{i\in I, j\in J_i}$ was not minimal, then there would exist some non-zero $\varphi_{kl} \in \overline{\text{span}}(\varphi_{ij})_{i\in I, j\in J_i, (i,j)\neq(k,l)}$. 
Since $(\varphi_{ij})_{j\in J_i}$ is a minimal sequence in $V_i$ for every $i\in I$, we see that $\varphi_{kl} \not\in \overline{\text{span}}(\varphi_{kj})_{j\in J_k, j\neq l}$, which implies 
$\varphi_{kl} \in \overline{\text{span}}(\varphi_{ij})_{i\in I, j\in J_i, i\neq k}= \overline{\text{span}}(V_i)_{i\in I, i\neq k}$, a contradiction. 
\end{proof}

Now, we are almost ready to characterize those fusion frames, that are fusion Riesz bases. We only need to prove the following preparatory lemma.

\begin{lemma}\label{globalisfusion}\cite{koeba23} Let $V=(V_i, v_i)_{i\in I}$ be a weighted family of closed subspaces of $\Hil$ and assume, that for each $i\in I$, $\varphi^{(i)} = (\varphi_{ij})_{j\in J_i}$ is a Parseval frame for $V_i$. Then $V$ is a fusion frame for $\Hil$ if and only if $v\varphi = (v_i \varphi_{ij})_{i\in I, j\in J_i}$ is a frame for $\Hil$. In this case, $S_{v\varphi} = S_V$.
\end{lemma}

\begin{proof}
The first part of the statement follows immediately from Theorem \ref{fusframesysTHM}. For the second part, recall that $\varphi^{(i)}$ is Parseval if and only if $S_{\varphi^{(i)}} = \mathcal{I}_{V_i}$. Hence, for every $f\in \Hil$
\begin{flalign}
    S_{v\varphi} f &= \sum_{i\in I} \sum_{j\in J_i} \langle f, v_i \varphi_{ij} \rangle v_i \varphi_{ij} \notag \\
    &= \sum_{i\in I} v_i^2 \sum_{j\in J_i} \langle f, \varphi_{ij} \rangle S_{\varphi^{(i)}}^{-1} \varphi_{ij} \notag \\ 
    &= \sum_{i\in I} v_i^2 \pi_{V_i} f = S_V f \notag
\end{flalign}
and the proof is complete.
\end{proof}

After this preparation we can prove a slight generalization of a result in \cite{Shaarbal}:

\begin{theorem}\label{minimalchaR}
Let $(V_i,v_i,\varphi^{(i)})_{i\in I}$ be a fusion frame system, and let $\varphi^{(i)}$ be a Riesz basis for $V_i$, for every $i\in I$. Then the following are equivalent.
\begin{enumerate}
    \item[(i)] $(V_i)_{i\in I}$ is  minimal in $\Hil$.
    \item[(ii)] $(V_i)_{i\in I}$ is a Riesz decomposition for $\Hil$.
    \item[(iii)] $(V_i, v_i)_{i\in I}$ is a fusion Riesz basis for $\Hil$.
    \item[(iv)] $(v_i\varphi_{ij})_{i\in I, j\in J_i}$ is a Riesz basis for $\Hil$. 
    \item[(v)] $S_V^{-1}V_i \perp V_k$, for all $i,k\in I$ with $i\neq k$. 
    \item[(vi)] 
    $v_i^2 \pi_{V_k} S_V^{-1}\pi_{V_i}=\delta_{ik}\pi_{V_i}$ for all $i,k\in I$.
\end{enumerate}
In this case, $(S_V^{-1} S_{\varphi^{(i)}}^{-1} v_i \varphi_{ij})_{i\in I, j\in J_i}$ coincides with the canonical dual frame of $(v_i\varphi_{ij})_{j\in J_i.i\in I}$.
\end{theorem}

\begin{proof}
(i)$\Rightarrow$(ii) Assume that $(V_i)_{i\in I}$ is not a Riesz decomposition of $\Hil$. Then there exists some $f\in \Hil$ with $f = \sum_{i\in I}f_i = \sum_{i\in I}g_i$, where $(f_i)_{i\in I} \neq (g_i)_{i\in I} \in (V_i)_{i\in I}$. So, there exists some $k\in I$ with $f_k \neq g_k \in V_k$. This implies, that $0\neq g_k-f_k=\sum_{i\in I, i\neq k}(f_i-g_i)$ and hence 
$g_k-f_k \in V_k \cap \overline{\text{span}}(V_i)_{i\in I, i\neq k}$, which means that $(V_i)_{i\in I}$ is not minimal.

(ii)$\Rightarrow$(iii) By assumption, $V$ is a fusion frame, which, according to Theorem \ref{fusionframechar}, means that $D_V$ is bounded and surjective. Since $\mathcal{R}(D_V \vert _{\mathcal{K}_V^2}) = \mathcal{R}(D_V) = \Hil$, we particularly see that $D_V \vert _{\mathcal{K}_V^2}$ is bounded and surjective. Now, if $(V_i)_{i\in I}$ is a Riesz decomposition, then $D_V \vert _{\mathcal{K}_V^2}$, by definition, is injective as well. By Theorem \ref{fusionrieszbasischar}, this implies that $V$ is a fusion Riesz basis.

(iii)$\Rightarrow$(i) If $V$ is a fusion Riesz basis, then, by Lemma \ref{weightsFRB}, $V^1 := (V_i,1)_{i\in I}$ is a fusion Riesz basis as well. By Theorem \ref{fusionrieszbasischar}, its associated synthesis operator $D_{V^1}$ maps $\mathcal{K}_V^2$ boundedly and bijectively onto $\Hil$. Hence, for every $f\in \Hil$, there exists a unique sequence $(f_i)_{i\in I}\in \mathcal{K}_V^2$, such that $f = D_{V^1} (f_i)_{i\in I} =  \sum_{i\in I} f_i$. 
If was not $(V_i)_{i\in I}$ is not minimal, then there there would exist such an $f$ so that there exists another sequence $(g_i)_{i\in I} \in (V_i)_{i\in I}$ satisfying $\sum_{i\in I} g_i = f$. 
Then, by definition of $D_{V^1}$, $(g_i)_{i\in I} \in \text{dom}(D_{V^1}) = \mathcal{K}_V^2$, which contradicts the bijectivity of $D_{V^1}$.

(iii)$\Leftrightarrow$(iv) 
If (iii) is satisfied, then $(v_i)_{i\in I}$ is semi-normalized by Lemma \ref{weightsFRB}. On the other hand, if (iv) holds, then $(v_i)_{i\in I}$ is semi-normalized as well: Indeed, similarly to Lemma \ref{weightsFRB}, one can show, that the frame vectors associated to a Riesz basis are semi-normalized with uniform bounds given by the respective square-roots of the frame bounds. Therefore (compare with Definition \ref{def:fusionframeseystem}) $\sqrt{A} \leq \sqrt{A_i} \leq \Vert \varphi_{ij} \Vert \leq \sqrt{B_i} \leq \sqrt{B}$ (for all $i\in I, j\in J_i$). At the same time we know that $m \leq \Vert v_i \varphi_{ij} \Vert \leq M$ for suitable constants $m,M >0$ (for all $i\in I, j\in J_i$). Thus $(v_i)_{i\in I}$ must be semi-normalized. Now, observe that (iv) holds if and only if $(v_i \varphi_{ij})_{i\in I, j\in J_i}$ is a minimal sequence. Since $(v_i)_{i\in I}$ is semi-normalized in either case, the latter is true if and only if $(\varphi_{ij})_{i\in I, j\in J_i}$ is minimal. By Proposition \ref{minimalcha}, this is true if and only if $(V_i)_{i\in I}$ is minimal, and by the equivalence (i) $\Leftrightarrow$ (iii) this is equivalent to $V$ being a fusion Riesz basis. 

(iii)$\Rightarrow$(v) For every $i\in I$, let $(e_{ij})_{j\in J_i}$ be an orthonormal basis for $V_i$. By the equivalence (iii)$\Leftrightarrow$(iv) and Lemma \ref{globalisfusion}, $ve := (v_i e_{ij})_{i\in I, j\in J_i}$ is a Riesz basis for $\Hil$ with $S_{ve} = S_V$. By Theorem \ref{rieszbasischar}, $ve$ has a unique biorthogonal sequence given by $(S_{ve}^{-1} v_i e_{ij})_{i\in I, j\in J_i} = (S_V^{-1} v_i e_{ij})_{i\in I, j\in J_i}$, which implies $S_V^{-1} e_{ij} \perp e_{kl}$, whenever $(i,j)\neq (k,l)$. This yields $S_{V}^{-1} V_i \perp V_k$ whenever $i\neq k$, as desired. 

(v)$\Rightarrow$(vi) For each $k\in I$ we have
$$\pi_{V_k} = \pi_{V_k} S_V^{-1} S_V = \sum_{i\in I} v_i^2 \pi_{V_k} S_V^{-1}\pi_{V_i} = v_k^2 \pi_{V_k} S_V^{-1}\pi_{V_k}.$$

(vi)$\Rightarrow$(v)
Let $f\in V_i$ and $g\in V_k$, where $i\neq k$. Then 
$$\left<S_V^{-1}f,g\right>=\left< S_{V}^{-1}\pi_{V_i}f, \pi_{V_k}g \right>= \frac{1}{v_i^2}\left< v_i^2 \pi_{V_k}S_V^{-1}\pi_{V_i}f,g \right>=0.$$

(vi)$\Rightarrow$(i) By Theorem \ref{rieszbasischar}, it suffices to show that the frame $(v_i \varphi_{ij})_{i\in I, j\in J_i}$ is biorthogonal to the family $(S_V^{-1} S_{\varphi^{(i)}}^{-1} v_i \varphi_{ij})_{i\in I, j\in J_i}$. 
By using condition (vi), we see that 
\begin{flalign}
   \langle v_i \varphi_{ij}, S_V^{-1} S_{\varphi^{(k)}}^{-1} v_k \varphi_{kl} \rangle &= \langle v_i \pi_{V_i} \varphi_{ij}, S_V^{-1} \pi_{V_k} S_{\varphi^{(k)}}^{-1} v_k \varphi_{kl} \rangle \notag \\
   &= \frac{1}{v_i}\langle v_i^2 \pi_{V_k} S_V^{-1} \pi_{V_i} \varphi_{ij}, S_{\varphi^{(k)}}^{-1} v_k \varphi_{kl} \rangle \notag \\
   &= \frac{1}{v_i} \langle \delta_{ik} \pi_{V_i} \varphi_{ij}, S_{\varphi^{(k)}}^{-1} v_k \varphi_{kl} \rangle = \frac{1}{v_i} \langle \delta_{ik} \varphi_{ij}, S_{\varphi^{(k)}}^{-1} v_k \varphi_{kl} \rangle. \notag
\end{flalign}
In case $k\neq i$ this gives $\langle v_i \varphi_{ij}, S_V^{-1} S_{\varphi^{(k)}}^{-1} v_k \varphi_{kl} \rangle = 0$ and in case $k=i$ we continue with $\langle v_i \varphi_{ij}, S_V^{-1} S_{\varphi^{(k)}}^{-1} v_k \varphi_{kl} \rangle = \frac{1}{v_k} \langle \varphi_{kj}, S_{\varphi^{(k)}}^{-1} v_k \varphi_{kl} \rangle = \langle \varphi_{kj}, S_{\varphi^{(k)}}^{-1} \varphi_{kl} \rangle = \delta_{jl}$, where we used that each local Riesz basis $\varphi^{(i)}$ is biorthogonal to its canonical dual $(S_{\varphi^{(i)}}^{-1} \varphi_{ij})_{j\in J_i}$. 

The final statement follows from the proof of (vi)$\Rightarrow$(i) combined with Theorem \ref{rieszbasischar}.
\end{proof}

As a consequence we obtain the following. 

\begin{corollary}
If $V = (V_i, v_i)_{i\in I}$ is a fusion Riesz basis such that $S_V$ and $\pi_{V_i}$ commute for all $i\in I$, then $(V_i)_{i\in I}$ is an orthonormal fusion basis.
In particular, every tight fusion Riesz basis is an orthonormal fusion basis. 
\end{corollary}

\begin{proof} 
Using condition (vi) from Theorem \ref{minimalchaR}, we obtain 
$$\sum_{i\in I}\pi_{V_i} = \sum_{i\in I}v_i^2\pi_{V_i}S_V^{-1}\pi_{V_i} = S_V^{-1} S_V = \mathcal{I}_{\Hil},$$
which means that $(V_i, 1)_{i\in I}$ is a Parseval fusion frame. By Theorem \ref{chaONB}, this implies that $(V_i)_{i\in I}$ is an orthonormal fusion basis.
\end{proof}

The following result shows that implication (iii)$\Rightarrow$(i) of Theorem \ref{rieszbasischar} does not hold in the fusion frame setting. Only the converse statement remains true for fusion frames.

\begin{proposition}\label{exactfusion}\cite{caskut04}
Every fusion Riesz basis is an exact fusion frame. The converse is not true in general.
\end{proposition}

\begin{proof}
Let $V=(V_i,v_i)$ be a fusion Riesz basis. Then, by Theorem \ref{fusionrieszbasischar}, $C_V$ maps $\Hil$ bijectively onto $\mathcal{K}_V^2$. Fix some arbitrary index $k\in I$ and choose some non-zero $g\in V_k$. Then $(\dots , 0, 0, g, 0, 0, \dots )$ ($g$ in the $k$-th component) is an element in  $\mathcal{K}_V^2$ and thus contained in $\mathcal{R}(C_V)$. Therefore, there is a unique non-zero $f\in \Hil$ such that $(\dots , 0, 0, g, 0, 0, \dots ) = C_V f = (v_i \pi_{V_i} f)_{i\in I}$. In particular $\pi_{V_i} f = 0$ for all $i\neq k$.
Hence, if we removed $(V_k, v_k)$ from the fusion frame $V$, we would not be able to maintain the fusion frame inequalities for $f$. Since $k$ was chosen arbitrary, this means that $V$ is exact.

To show that the converse statement is false, consider an orthonormal basis $(e_i)_{i=1}^3$ for $\mathbb{R}^3$, and let $E_1 = \text{span}\lbrace e_1,e_2 \rbrace$, $E_2 = \text{span}\lbrace e_2, e_3 \rbrace$ and $w_1=w_2=1$.
It is immediate that this is an exact fusion frame and since $E_1\cap E_2 \neq 0$, it is not a Riesz decomposition and hence (Theorem \ref{minimalchaR}) not a fusion Riesz basis. Alternatively, note that the global frame $(e_1, e_2, e_2, e_3)$ is not a Riesz basis.
\end{proof}


\section{Fusion frames and operators}\label{Fusion frames and operators}

In this section, we consider subspaces under the action of operators and their spanning properties. Here we will encounter more results, which show the contrast between frames and fusion frames, such as Proposition \ref{exactfusion}.

We start with the following preliminary results, which directly popped up at the first discussion of the duality of fusion frames and the related observation that $T \pi_V$ is in general {\em not} a projection.  

\begin{lemma}\label{fusionframeoplem}\cite{cakuli08}
Let $V$ be a closed subspace of $\Hil$ and $T \in \mathcal{B}(\Hil)$. Then
\begin{eqnarray}\label{piTcomb}\pi_{V}T^*=\pi_{V}T^*\pi_{\overline{TV}}.\end{eqnarray}
\end{lemma}

\begin{proof}
First note that $g\in (TV)^{\perp} = (\overline{TV})^{\perp}$ if and only if $T^*g\in V^{\perp}$. Therefore, for all $f\in \Hil$, we see that
$$\pi_VT^*f=\pi_VT^* \pi_{\overline{TV}}f + \pi_V T^* \pi_{(TV)^{\perp}} f = \pi_VT^*\pi_{\overline{TV}}f.$$
\end{proof}

\begin{lemma}\label{techlemma2}
Let $U\in \mathcal{B}(\Hil)$ and $(\varphi_i)_{i\in I} \subseteq \Hil$ be a countable family of vectors in $\Hil$. Then $U\overline{\text{span}}(\varphi_i)_{i\in I} \subseteq \overline{\text{span}}(U \varphi_i)_{i\in I}$. If $U$ is additionally bijective, then $U\overline{\text{span}}(\varphi_i)_{i\in I} = \overline{\text{span}}(U \varphi_i)_{i\in I}$.
\end{lemma}

\begin{proof}
Let $g = Uf \in U\overline{\text{span}}(\varphi_i)_{i\in I}$ for $f\in \overline{\text{span}}(\varphi_i)_{i\in I}$. Then there exists a sequence $(f_n)_{n=1}^\infty \subseteq \text{span}(\varphi_i)_{i\in I}$ such that $\Vert f - f_n \Vert \longrightarrow 0$ as $n\rightarrow \infty$. This implies that $(Uf_n)_{n=1}^\infty \in \text{span}(U\varphi_i)_{i\in I}$ converges to $Uf = g$ in $\Hil$, since $\Vert Uf - Uf_n \Vert \leq \Vert U \Vert \Vert f - f_n \Vert \longrightarrow 0$ as $n\rightarrow \infty$. Thus $g \in \overline{\text{span}}(U \varphi_i)_{i\in I}$. 

For the second part, assume that $g\in \overline{\text{span}}(U \varphi_i)_{i\in I}$. Then there exists a sequence $(g_n)_{n=1}^\infty = (Uf_n)_{n=1}^\infty \subseteq \text{span}(U\varphi_i)_{i\in I}$, such that $\Vert g - Uf_n \Vert \longrightarrow 0$ as $n\rightarrow \infty$. In particular, we have $(f_n)_{n=1}^\infty \subseteq \text{span}(\varphi_i)_{i\in I}$. Moreover, since $U$ is surjective, there exists some $f\in \Hil$ with $Uf=g$. By using the invertibility of $U$, this implies that $f\in \overline{\text{span}}(\varphi_i)_{i\in I}$, because $\Vert f - f_n \Vert = \Vert U^{-1} U f - U^{-1} U f_n \Vert \leq \Vert U^{-1} \Vert \Vert Uf - U f_n \Vert \longrightarrow 0$ as $n\rightarrow \infty$. Thus $g = Uf \in U\overline{\text{span}}(\varphi_i)_{i\in I}$.
\end{proof}

In case $U \in \mathcal{B}(\Hil)$ is not bijective, then the reverse inclusion in Lemma \ref{techlemma2} is not true in general:

\begin{example}{Example}\label{techex}
Let $\Hil = \ell^2(\mathbb{N})$ and $\varphi_i = (\delta_{in} \cdot 1/\sqrt{n})_{n\in \mathbb{N}} \in \ell^2(\mathbb{N})$ for $i\in \mathbb{N}$. The operator $U:\ell^2(\mathbb{N}) \longrightarrow \ell^2(\mathbb{N})$, defined by $U(a_n)_{n\in \mathbb{N}} = (a_n \cdot 1/\sqrt{n})_{n\in \mathbb{N}}$, is bounded on $\ell^2(\mathbb{N})$ with $\Vert U \Vert = 1$ and we have $U\varphi_i = (\delta_{in} \cdot 1/n)_{n\in \mathbb{N}}$. Observe that $(1/n)_{n\in \mathbb{N}} = \sum_{i\in \mathbb{N}} U\varphi_i \in \overline{\text{span}}(U \varphi_i)_{i\in \mathbb{N}}$. However, $(1/n)_{n\in \mathbb{N}}$ is not contained in $U\overline{\text{span}}(\varphi_i)_{i\in \mathbb{N}}$. Indeed, by component-wise definition of $U$, the only possible sequence, which could be mapped onto $(1/n)_{n\in \mathbb{N}}$ by $U$, is $(1/\sqrt{n})_{n\in \mathbb{N}}$. But since $(1/\sqrt{n})_{n\in \mathbb{N}} \notin \ell^2(\mathbb{N})$, we particularly have that $(1/\sqrt{n})_{n\in \mathbb{N}} \notin \overline{\text{span}}(\varphi_i)_{i\in \mathbb{N}}$. Note that the above demonstrates, that $U$ is not surjective.
\end{example}

From Lemma \ref{techlemma2}, we can also easily deduce the following two results.

\begin{corollary}\label{techlemma3}
Let $(\varphi_i)_{i\in I} \subseteq \Hil$ and $U\in \mathcal{B}(\Hil)$.  
\begin{enumerate}
    \item[(i)] If $(U\varphi_i)_{i\in I}$ is minimal and $U$ is injective, then $(\varphi_i)_{i\in I}$ is minimal. In case $U$ is bijective, the converse is true as well. 
    \item[(ii)] If $(\varphi_i)_{i\in I}$ is a complete sequence and $U$ has dense range, then $(U\varphi_i)_{i\in I}$ is a complete sequence. In case $U$ is bijective, the converse is also valid. 
\end{enumerate}
\end{corollary}

\begin{proof}
(i) Assume $(U\varphi_i)_{i\in I}$ is a minimal sequence and that $(\varphi_i)_{i\in I}$ is not minimal. Then there is some $m\in I$ such that $\varphi_m\in \overline{\text{span}}(\varphi_i)_{i\in I, i\neq m}$. Since  $U$ is injective, we have $U\varphi_m \neq 0$ and by Lemma \ref{techlemma2}, $U\varphi_m\in U\overline{\text{span}}(\varphi_i)_{i\in I, i\neq m}\subseteq \overline{\text{span}}(U\varphi_i)_{i\in I, i\neq m}$, a contradiction. The second statement is shown similarly.

(ii) Since $(\varphi_i)_{i\in I}$ is complete, we see by using Lemma \ref{techlemma2} that $U(\Hil)=U\overline{\text{span}}(\varphi_i)_{i\in I}\subseteq \overline{\text{span}}(U\varphi_i)_{i\in I}$. Thus 
$\Hil=\overline{U(\Hil)}\subseteq \overline{\text{span}}(U\varphi_i)_{i\in I}$. Analogously, the second statement is shown by repeating the above argument with $U^{-1}$ instead of $U$.
\end{proof}

\begin{corollary}\label{techlemma}
Let $V$ be a closed subspace of $\Hil$ and let $U\in \mathcal{B}(\Hil)$ be bijective. Then $UV$ is a closed subspace of $\Hil$.
\end{corollary}
After these preparatory results, we can show that the fusion frame property is invariant under the action of a bounded invertible operator. More precisely:  

\begin{theorem}\label{kate}\cite{cakuli08}
Let $V = (V_i,v_i)_{i\in I}$ be  a fusion frame for $\Hil$ with fusion frame bounds $A_V \leq B_V$, and let $U\in \mathcal{B}(\Hil)$ be invertible.  Then $UV = (UV_i,v_i)_{i\in I}$ is a fusion frame for $\Hil$ with fusion frame bounds $A_V\left\|U^{-1}\right\|^{-2}\|U\|^{-2}$ and $B_V \left\|U^{-1}\right\|^2\left\|U\right\|^2$, and its associated fusion frame operator $S_{UV}$ satisfies $$\frac{U S_V U^*}{\Vert U \Vert^2} \leq S_{UV} \leq \Vert U^{-1} \Vert^2 U S_V U^* .$$
\end{theorem}

\begin{proof}
By Corollary \ref{techlemma}, each subspace $UV_i$ is closed. Let $f\in \Hil$ be arbitrary. Then by (\ref{piTcomb}), we have 
\begin{flalign}
\langle S_{UV} f, f\rangle &= 
\sum_{i\in I} v_i^2\left\|\pi_{UV_i} f \right\|^2 \notag \\
&= \sum_{i\in I} v_i ^2 \left\|\pi_{UV_i}(U^{-1})^* \pi_{V_i} U^* f\right\|^2 \notag \\
&\leq  \left\|U^{-1}\right\|^2\sum_{i\in I}v_i^2\left\|\pi_{V_i} U^* f\right\|^2 \notag \\ 
&= \left\|U^{-1}\right\|^2 \Big\langle U S_V U^* f, f  \Big\rangle .\notag
\end{flalign}
By Lemma \ref{fusionframeoplem}, we see that
\begin{flalign}
\Big\langle \frac{U S_V U^* }{\Vert U \Vert^2} f, f\Big\rangle &= \frac{1}{\Vert U \Vert^2} \sum_{i\in I} v_i^2 \Vert \pi_{V_i} U^* f \Vert^2 \notag \\
&= \frac{1}{\Vert U \Vert^2} \sum_{i\in I} v_i^2 \Vert \pi_{V_i} U^* \pi_{UV_i} f \Vert^2 \notag \\
&\leq \sum_{i\in I} v_i^2 \Vert \pi_{UV_i} f \Vert^2 = \langle S_{UV} f,f\rangle . \notag
\end{flalign}
Thus 
$$\frac{U S_V U^*}{\Vert U \Vert^2} \leq S_{UV} \leq \Vert U^{-1} \Vert^2 U S_V U^* .$$
Moreover, observe that $\Vert f \Vert \leq \Vert U^{-1} \Vert \Vert U^*f \Vert$ for all $f\in \Hil$. Piecing all observations together, we obtain
\begin{flalign}
A_V \frac{1}{\Vert U^{-1} \Vert^2\Vert U \Vert^2} \Vert f \Vert^2 
&\leq A_V \frac{1}{\Vert U \Vert^2} \Vert U^* f \Vert^2 \notag \\
&\leq \frac{1}{\Vert U \Vert^2} \Big\langle S_V U^* f, U^* f \Big\rangle \notag \\
&\leq \langle S_{UV} f, f \rangle \notag \\ 
&\leq \left\|U^{-1}\right\|^2 \Big\langle U S_V U^* f, f  \Big\rangle \leq B_V \left\|U^{-1}\right\|^2 \Vert U \Vert ^2 \Vert f \Vert ^2 \notag
\end{flalign}
for all $f\in \Hil$, as desired.
\end{proof}

From the above, we immediately obtain the following results: 

\begin{corollary}\label{3itemeq}\cite{Gavruta7}
Let $V= (V_i,v_i)_{i\in I}$ be a weighted family of closed subspaces in $\Hil$ and let $U\in \mathcal{B}(\Hil)$ be invertible. Then the following are equivalent:
\begin{enumerate}
    \item[(i)] $V$ is a fusion frame for $\Hil$.
    \item[(ii)] $UV = (UV_i,v_i)_{i\in I}$ is a fusion frame for  $\Hil$.
\end{enumerate}
\end{corollary}

\begin{corollary}\label{canonicaldualfusion}\cite{Gavruta7}
Let $(V_i,v_i)_{i\in I}$ be  a fusion frame for $\Hil$ with fusion frame bounds $A_V \leq B_V$. Then $S_V^{-1}V = (S_V^{-1}V_i,v_i)_{i\in I}$ is a fusion frame for $\Hil$ with fusion frame bounds $A_V^3/B_V^2 \leq B_V^3 / A_V^2$.  
\end{corollary}

\begin{proof}
Substituting $S_V^{-1}$ for $U$ in Theorem \ref{kate} yields that $(S_V^{-1}V_i,v_i)_{i\in I}$ is a fusion frame with bounds $A_V\Vert S_V \Vert ^{-2} \Vert S_V^{-1} \Vert^{-2}$ and $B_V \Vert S_V\Vert^2 \Vert S_V^{-1} \Vert^2$. Combining this with the respective operator norm estimates from (\ref{fusionnorms}) gives the claimed fusion frame bounds. 
\end{proof}

For unitary operators $U\in \mathcal{B}(\Hil)$, Theorem \ref{kate} implies the following.

\begin{corollary}\label{surjappfu}
Let $V$ be a  fusion frame for $\Hil$ and $U\in \mathcal{B}(\Hil)$ be unitary. Then
$UV$ is a fusion frame with the same fusion frame bounds and its associated fusion frame operator is given by $S_{UV} = U S_V U^*$.
\end{corollary}

We also note that fusion Riesz bases are invariant under the action of a bounded invertible operator. In other words, Corollary \ref{sequenceschar2} (iii) remains valid in the fusion frame setting. 

\begin{theorem}\label{3itemeqri}
Let $V= (V_i,v_i)_{i\in I}$ be a weighted family of closed subspaces in $\Hil$ and let $U\in \mathcal{B}(\Hil)$ be invertible. The following are equivalent.
\begin{enumerate}
    \item[(i)] $V$ is a fusion Riesz basis for $\Hil$.
    \item[(ii)] $UV:= (UV_i,v_i)_{i\in I}$ is a fusion Riesz basis for  $\Hil$.
\end{enumerate}
\end{theorem}

\begin{proof}
(i)$\Rightarrow$(ii) Theorem \ref{kate} implies that $UV$ is a fusion frame. Hence, according to Theorem \ref{minimalchaR}, it suffices to show that $(UV_i)_{i\in I}$ is minimal. To this end, assume that  $(\varphi_{ij})_{j\in J_i}$ be a Riesz basis for $V_i$ for every $i\in I$. Since $V$ is a fusion Riesz basis, Theorem \ref{minimalchaR} implies that $V$ is a minimal fusion frame. By Proposition \ref{minimalcha}, $(\varphi_{ij})_{i\in I, j\in J_i}$ is a minimal sequence in $\Hil$, and by Corollary \ref{techlemma3}, $(U\varphi_{ij})_{i\in I, j\in J_i}$ is minimal as well. Since $U$ is invertible, it is also easy to see that $(U\varphi_{ij})_{j\in J_i}$ is a Riesz basis for $UV_i$ for every $i\in I$. Thus $UV$ is minimal by Proposition \ref{minimalcha}.

(i)$\Rightarrow$(ii) If $UV$ is a fusion Riesz basis, then so is $V = U^{-1}UV$ by the above.
\end{proof}

In view of Corollary \ref{sequenceschar2} (ii), it seems only natural to investigate fusion frames under the action of a bounded surjective operator next. At first glance one would guess, that if $V = (V_i, v_i)_{i\in I}$ is a fusion frame and $U\in \mathcal{B}(\Hil)$ surjective, then $UV = (UV_i,v_i)_{i\in I}$ is a fusion frame for $\Hil$, since the corresponding result for frames is true as well. However, such a result cannot be achieved with fusion frames in full generality. Indeed, for a surjective operator $U\in \mathcal{B}(\Hil)$, $UV_i$ is not even a closed subspace in general. Even if we want to prove the corresponding result for the closed subspaces $\overline{UV_i}$, we need to make an additional assumption, because otherwise the statement is not valid (see below). 

\begin{proposition}\label{AppUW}
Let $V=(V_i,v_i)_{i\in I}$ be a   fusion frame for $\Hil$ with fusion frame bounds $A_V \leq B_V$, and let $U\in \mathcal{B}(\Hil)$ be surjective with the additional property that $U^* U V_i\subseteq V_i$ for all $i\in I$. Then $(\overline{UV_i},v_i)_{i\in I}$ is a fusion frame with fusion frame bounds $A_V\left\|U^{\dag}\right\|^{-2}\|U\|^{-2} \leq B_V\left\|U^{\dag}\right\|^{2}\|U\|^{2}$.
\end{proposition}

\begin{proof}
Let $(e_{ij})_{j\in J_i}$ be an orthonormal basis for $V_i$ for every $i\in I$. Then, by Theorem \ref{fusframesysTHM}, $(v_i e_{ij})_{i\in I, j\in J_i}$ is a  frame for $\Hil$ with frame bounds $A_V$ and $B_V$. By Theorem \ref{sequenceschar}, $(v_i U e_{ij})_{i\in I, j\in J_i}$ is a frame for $\Hil$ as well and one can compute \cite[Corollary 5.3.2]{ch08} that $A_V\left\|U^{\dag}\right\|^{-2}$ and $B_V\|U\|^2$ are frame bounds for it. Now, let $i\in I$ be arbitrary and choose $Uf_i\in UV_i$. Then 
$$\frac{\left\|Uf_i\right\|^2}{\left\|U^{\dag}\right\|^2} = \frac{\left\|\pi_{\mathcal{R}(U)}^* Uf_i\right\|^2}{\left\|U^{\dag}\right\|^2} = \frac{\left\|(U^{\dag})^* U^* Uf_i\right\|^2}{\left\|U^{\dag}\right\|^2} \leq \left\|U^*Uf_i\right\|^2 .$$
Since $U^* U V_i \subseteq V_i$, we may continue with 
$$\left\|U^*Uf_i\right\|^2 = \sum_{j\in J_i}\left|\left< U^* Uf_i,e_{ij}\right>\right|^2 =  \sum_{j\in J_i}\left|\left< Uf_i,Ue_{ij}\right>\right|^2 .$$ On the other hand, clearly $\left\|U^*Uf_i\right\|^2 \leq \left\|U\right\|^2 \left\|Uf_i\right\|^2$. All together, this means that $Ue^{(i)} := (Ue_{ij})_{j\in J_i}$ satisfies the frame inequalities for all elements from $UV_i$. 
Finally, by using a density argument \cite[Lemma 5.1.7]{ole1}, we can conclude that the families $Ue^{(i)}$ are frames for their respective closed subspaces $\overline{UV_i}$ with common frame bounds $\Vert U^{\dagger} \Vert^{-2}$ and $\Vert U \Vert^2$. Hence, employing Theorem \ref{fusframesysTHM}, yields that $(\overline{ UV_i},v_i)$ is a fusion frame for $\Hil$ with fusion frame bounds $A_V\left\|U^{\dag}\right\|^{-2}\|U\|^{-2}$ and  $B_V\left\|U^{\dag}\right\|^{2}\|U\|^{2}$. 
\end{proof}

The following example demonstrates, that even if we consider an orthonormal fusion basis $(E_i)_{i\in I}$ and apply a bounded surjective operator $U\in \mathcal{B}(\Hil)$ on each $E_i$, then, different to the Hilbert frame setting - Theorem \ref{sequenceschar} and Corollary \ref{sequenceschar2} - we don't necessarily obtain that $(\overline{UE_i}, v_i)_{i\in I}$ is a fusion frame, no matter which family of weights $(v_i)_{i\in I}$ we choose. This is another observation that separates the properties of frames from the properties of fusion frames.

\begin{example}{Example \cite{rust08}}\label{nogosurjective}
Let $(e_n)_{n\in \mathbb{N}}$ be an orthonormal basis for some infinite-dimensional Hilbert space $\Hil$ and define $E_k := \text{span}\lbrace e_{2k-1}, e_{2k} \rbrace$ for every $k\in \mathbb{N}$. Then $(E_k)_{k\in \mathbb{N}}$ is an orthonormal fusion basis. 

Next, consider the linear operator $U:\Hil_0 \longrightarrow \Hil$ defined by
$$Ue_n := 
   \begin{cases}
     2^{-\frac{n+1}{2}} e_1 & \text{if } n \text{ is odd} \\
     e_{\frac{n}{2}+1} & \text{if } n \text{ is even} 
   \end{cases},$$
where, $\Hil_0 := \big\lbrace f \in \Hil: f = \sum_{n\in K}\langle f, e_n \rangle e_n, \vert K\vert < \infty \big\rbrace$. We show that $U$ is bounded on $\Hil_0$. For arbitrary $g = \sum_{n \in K} \langle g, e_n \rangle e_n \in \Hil_0$ we have 
\begin{flalign}
Ug &= \sum_{n\in K} \langle g, e_n \rangle Ue_n \notag \\
&= \sum_{\substack{n \in K \\ n \text{ odd}}} \langle g, e_n \rangle 2^{-\frac{n+1}{2}} e_1 + \sum_{\substack{n\in K \\ n \text{ even}}}  \langle g, e_n \rangle e_{\frac{n}{2}+1} \notag
\end{flalign}
and thus 
\begin{flalign}
\big\Vert Ug \big\Vert^2 &= \left\Vert \sum_{\substack{n\in K \\ n \text{ odd}}} \langle g, e_n \rangle 2^{-\frac{n+1}{2}} e_1 \right\Vert^2 + \sum_{\substack{n\in K \\ n \text{ even}}}  \big\Vert  \langle g, e_n \rangle e_{\frac{n}{2}+1} \big\Vert^2 \notag \\
&\leq \left( \sum_{\substack{n\in K \\ n \text{ odd}}} \vert \langle g, e_n \rangle \vert 2^{-\frac{n+1}{2}} \right)^2 + \sum_{\substack{n\in K \\ n \text{ even}}}  \vert  \langle g, e_n \rangle \vert^2 \notag \\
&\leq \left( \sum_{n\in K} \vert \langle g, e_n \rangle \vert^2 \right) \left( \sum_{n\in K} 2^{-(n+1)} \right) + \sum_{n\in K}  \vert  \langle g, e_n \rangle \vert^2 \notag \\
&\leq 3 \Vert g \Vert^2 . \notag
\end{flalign}
Since $\Hil_0$ is a dense subspace of $\Hil$, there exists a unique extension of $U$ (again denoted by $U$) to a bounded operator on $\Hil$. To see that $U$ is surjective, let $f\in \mathcal{H}$ be arbitrary. Then, by definition of $U$, 
\begin{flalign}
f &= \sum_{n=1}^{\infty} \langle f, e_n \rangle e_n \notag \\
&= \langle f, 2 Ue_1 \rangle 2 Ue_1 + \sum_{n=2}^{\infty} \langle f, U e_{2n-2} \rangle U e_{2n-2} \notag \\
&= U \Big( 4\langle U^*f, e_1 \rangle e_1 + \sum_{n=2}^{\infty} \langle U^* f, e_{2n-2} \rangle e_{2n-2} \Big). \notag
\end{flalign}

Now, set $V_k := UE_k = \text{span}\lbrace e_1, e_{k+1} \rbrace$ for $k\in \mathbb{N}$. Towards a contradiction, suppose that there exists a family $(v_k)_{k\in \mathbb{N}}$ of weights, such that $(V_k, v_k)_{k\in \mathbb{N}}$ is a fusion frame for $\Hil$ with fusion frame bounds $A_v \leq B_v$. Then, by applying the corresponding fusion frame inequalities to $f = e_1$, we obtain $A_v \leq \sum_{k\in \mathbb{N}} v_k^2 \leq B_v$, i.e $(v_k)_{k\in \mathbb{N}} \in \ell^2(\mathbb{N})$. On the other hand, for $f = e_{k+1}$, we obtain 
$$A_v = A_v \Vert e_{k+1} \Vert^2 \leq \sum_{l\in \mathbb{N}} v_l^2 \Vert \pi_{V_l} e_{k+1} \Vert^2 = v_k^2 \qquad (\forall k\in \mathbb{N}),$$ which violates $v_k \rightarrow 0$ (as $k\rightarrow \infty$), a contradiction.
\end{example}

The reason for the failure in the above example, is that $\frac{\gamma (U \pi_{E_k})}{\Vert U \pi_{E_k} \Vert} = 2^{-k} \rightarrow 0$ (as $k\rightarrow \infty$), where $\gamma (A)$ denotes the \emph{minimum modulus} of a bounded operator $A \in \mathcal{B}(\Hil_1, \Hil_2)$, defined by $\gamma (A) = \inf \big\lbrace \Vert Af \Vert , \Vert f \Vert = 1, f\in \mathcal{N}(A)^{\perp} \big\rbrace$ \cite{rust08}.  Operator theoretic notions such as the minimum modulus can be used in order to derive conditions, such that statements in the flavour of Proposition \ref{AppUW} (and analogous ones for fusion frame sequences) can be proven to be true \cite{rust08}. 

\

The converse of the above example is true in the following sense. Below, we show that every fusion frame $V=(V_i,v_i)_{i\in I}$ contains the image of an orthonormal fusion basis $(E_i)_{i\in I}$ under a bounded surjective operator $U\in \mathcal{B}(\Hil)$, meaning that $V_i \supseteq UE_i$ for all $i\in I$. Note, that if $V$ corresponds to a frame $(\varphi_i)_{i\in I}$, i.e. each $V_i = \text{span}\lbrace \varphi_i \rbrace$ is $1$-dimensional, then trivially $V_i = UE_i$ by Theorem \ref{sequenceschar}, since we exclude the case $V_i = \lbrace 0 \rbrace$ by default. In that sense, the next result is indeed a generalization of the necessity-part of Theorem \ref{sequenceschar} (b). Note that in case $V$ is a fusion Riesz basis, we have indeed $V_i = UE_i$ for all $i\in I$, where $U\in \mathcal{B}(\Hil)$ is bijective, see Proposition \ref{riesappT}.

\begin{proposition}\label{chafram}
Let $V=(V_i,v_i)_{i\in I}$ be  a fusion frame for $\Hil$. Then  $V_i \supseteq UE_i$ for all $i\in I$, where $(E_i)_{i\in I}$ is an orthonormal fusion basis and $U\in \mathcal{B}(\Hil)$ is surjective. 
\end{proposition}

\begin{proof}
Suppose that $(e_{ij})_{j\in J_i}$  is an orthonormal basis for $V_i$ for every $i\in I$. Hence, according to Theorem \ref{fusframesysTHM}, $(v_i e_{ij})_{i\in I, j\in J_i}$
is a frame for $\Hil$. By Theorem \ref{sequenceschar}, there exists a surjective operator $U\in \mathcal{B}(\Hil)$ and an orthonormal basis $(\psi_{ij})_{i\in I, j\in J_i}$ for $\Hil$, such that $v_i e_{ij} = U \psi_{ij}$ for all $i\in I, j\in J_i$. Set $E_i:= \overline{\text{span}}(\psi_{ij})_{j\in J_i}$ for every $i\in I$. Then $(E_i)_{i\in I}$ is an orthonormal fusion basis for $\Hil$ and we have $U E_i = U \overline{\text{span}}(\psi_{ij})_{j\in J_i} \subseteq \overline{\text{span}}(U \psi_{ij})_{j\in J_i} = \overline{\text{span}}(v_i e_{ij})_{j\in J_i} = V_i$ for every $i\in I$, by Lemma \ref{techlemma2}.
\end{proof}

\begin{proposition}\label{riesappT}
Let $V=(V_i,v_i)_{i\in I}$ be a fusion Riesz basis for $\Hil$. Then  $(V_i)_{i\in I} = (UE_i)_{i\in I}$, where $(E_i)_{i\in I}$ is an orthonormal fusion basis and $U\in \mathcal{B}(\Hil)$ is bijective. 
\end{proposition}

\begin{proof}
Suppose that $(e_{ij})_{j\in J_i}$  is an orthonormal basis for $V_i$ for every $i\in I$. By Theorem \ref{minimalchaR}, $(v_i e_{ij})_{i\in I, j\in J_i}$
is a Riesz basis for $\Hil$. Hence, by Theorem \ref{sequenceschar} (c), there exists a bijective operator $U\in \mathcal{B}(\Hil)$ and an orthonormal basis $(\psi_{ij})_{i\in I, j\in J_i}$ for $\Hil$, such that $v_i e_{ij} = U \psi_{ij}$ for all $i\in I, j\in J_i$. Set $E_i:= \overline{\text{span}}(\psi_{ij})_{j\in J_i}$ for every $i\in I$. Then $(E_i)_{i\in I}$ is an orthonormal fusion basis for $\Hil$ and we have $U E_i = U \overline{\text{span}}(\psi_{ij})_{j\in J_i} = \overline{\text{span}}(U \psi_{ij})_{j\in J_i} = \overline{\text{span}}(v_i e_{ij})_{j\in J_i} = V_i$ for every $i\in I$, by Lemma \ref{techlemma2}.
\end{proof}

We conclude this section with another major difference between properties of frames and properties of fusion frames. Recall that if $(\varphi_i)_{i\in I}$ is a frame with frame operator $S_{\varphi}$, then the family $(S_{\varphi}^{-1/2} \varphi_i)_{i\in I}$ is always a Parseval frame \cite{ole1}. In the fusion frame setting, the situation is very different. Below, we give an example of a fusion frame $V=(V_i,v_i)_{i\in I}$, so that for any arbitrary bounded and invertible operator $U\in \mathcal{B}(\Hil)$, $UV = (UV_i,w_i)_{i\in I}$ is \emph{never} a Parseval fusion frame, no matter which family $(w_i)_{i\in I}$ of weights we choose:

\begin{example}{Example \cite{rust08}}\label{parsevalex}
Let $\Hil = \mathbb{C}^4$ and $(e_i)_{i=1}^4$ be an orthonormal basis for $\mathbb{C}^4$. Consider the fusion frame $(V_i, 1)_{i=1}^3$ for $\mathbb{C}^4$, where $V_1 = \text{span}\lbrace e_1, e_2 \rbrace$, $V_2 = \text{span}\lbrace e_1, e_3 \rbrace$ and $V_3 = \text{span}\lbrace e_4 \rbrace$. We claim that for any invertible matrix $U\in \mathbb{C}^{4\times 4}$ and any choice $(w_i)_{i=1}^3$ of weights, $(UV_i, w_i)_{i\in I}$ is never a Parseval fusion frame for $\mathbb{C}^4$. Indeed, let $g_1 = \Vert Ue_1 \Vert^{-1}Ue_1$, $g_4 = \Vert Ue_4 \Vert^{-1}Ue_4$ and choose $g_2$ and $g_3$ in such a way
that $(g_1, g_2)$ is an orthonormal basis for $UV_1$ and $(g_1, g_3)$ is an orthonormal basis for $UV_2$. Suppose that $(UV_i, w_i)_{i\in I}$ was a Parseval fusion frame. Then its associated global frame $wg = (w_1 g_1, w_2 g_1, w_1 g_2, w_2 g_3, w_3 g_4)$ would be a Parseval frame by Corollary \ref{parsfusion}. Let $D_{wg} \in \mathbb{C}^{4\times 5}$ be its associated synthesis matrix, i.e. the matrix consisting of the frame elements of $wg$ as column vectors. Further, let $H \in \mathbb{C}^{4\times 4}$ be a unitary matrix whose first row is given by $g_1$ and denote the other rows by $h_2, h_3, h_4$. Then 
$$H \cdot D_{wg} = \begin{bmatrix} 
    w_1 & w_2 & \Vec{v} \\
   \textbf{0} & \textbf{0} & A \\
       \end{bmatrix}$$
for some $\Vec{v} = (0, 0, z) \in \mathbb{C}^3$ and $A\in \mathbb{C}^{3\times 3}$ given by 
$$A = \begin{bmatrix} 
    \langle h_2, w_1 g_2 \rangle & \langle h_2, w_2 g_3 \rangle & \langle h_2, w_3 g_4 \rangle \\
  \langle h_3, w_1 g_2 \rangle & \langle h_3, w_2 g_3 \rangle & \langle h_3, w_3 g_4 \rangle \\
  \langle h_4, w_1 g_2 \rangle & \langle h_4, w_2 g_3 \rangle & \langle h_4, w_3 g_4 \rangle \\
       \end{bmatrix} .$$
Since $H D_{wg} (H 
D_{wg})^* = H D_{wg} D_{wg}^* H^* = H S_{wg} H^* = H H^* = \mathcal{I}_{\mathbb{C}^4}$, we must have $1 = w_1^2 + w_2^2 + \vert z \vert^2$ and $A$ must be unitary. But this is impossible because the first two columns of $A$ have norms $\Vert w_1 g_2 \Vert = w_1$ and $\Vert w_2 g_3 \Vert = w_2$, while $1 = w_1^2 + w_2^2 + \vert z \vert^2$.
\end{example}

In particular, the above example demonstrates that for some fusion frames $(V_i, v_i)_{i\in I}$, the fusion frame $(S_V^{-1/2} V_i, w_i)_{i\in I}$ cannot be Parseval for any choice of weights $(w_i)_{i\in I}$.

We can - though - formulate a particular positive result:

\begin{proposition}
Let $V=(V_i,v_i)_{i\in I}$ be a fusion Riesz basis for $\Hil$. Then $(S_{V}^{-1/2}V_i)_{i\in I}$ is an orthonormal fusion basis for $\Hil$.
\end{proposition}

\begin{proof}
Assume that $(e_{ij})_{j\in J_i}$ is an orthonormal basis for $V_i$ for each $i\in I$. According to Theorem \ref{minimalchaR} and Lemma \ref{globalisfusion}, 
$(v_ie_{ij})_{j\in J_i,i\in I}$ is  a Riesz basis for $\Hil$ with the frame operator $S_V$. Therefore $(v_iS_V^{-1/2}e_{ij})_{j\in J_i,i\in I}$ is an orthonormal basis for $\Hil$. Combined with the fact that  $(v_iS_V^{-1/2}e_{ij})_{j\in J_i}$ is an orthonormal basis of $S_V^{-1/2}V_i$ for each $i\in I$, we see that $(S_V^{-1/2}V_i )_{i\in I}$ is an orthonormal fusion basis for $\Hil$.
\end{proof}


\section{Duality in fusion frame theory}\label{Duality in Fusion Frame Theory}

In this section we present the duality theory for fusion
frames. The difference between frames and fusion frames becomes more
evident in duality theory. Hence we decided to first show the origin
of the concept of dual fusion frames. We then present the duality
theory for fusion frames following \cite{heimoanza14, hemo14}. We
also adapt results of \cite{HeinekenMorillas18}, where the concept
is analyzed in the more general setting of oblique duality. 
We study
properties and provide examples of this concept. Finally, we
will discuss particular cases and other approaches for a dual fusion frame concept.

Recall, that in frame theory each $f \in \mathcal{H}$ is represented
by the collection of scalar coefficients $\langle f,\varphi_i \rangle$,
$i\in I$, that can be thought as a measure of the projection of $f$
onto each frame vector. From these coefficients $f$ can be stably
recovered  using a reconstruction formula via dual frames as in
(\ref{dualframedef}), i.e.
\begin{equation}
\sum_{i\in I} \langle f, \psi_i \rangle \varphi_i = \sum_{i\in I}
\langle f, \varphi_i \rangle \psi_i = f \qquad (\forall f\in
\mathcal{H}).\notag
\end{equation}
In operator notation, this reads
\begin{equation}\label{DCisIframes}
D_{\varphi} C_{\psi} = D_{\psi} C_{\varphi} = \mathcal{I}_{\Hil}.
\end{equation}

As already mentioned in Section \ref{Frame Theory}, a redundant
frame has infinitely many dual frames, which makes frames so
desirable for countless applications. Taking this into account, our
aim is to have a notion of a dual fusion frame as we have it in
classical frame theory, and that furthermore leads to analogous
results. Since the duality condition can be expressed in the two
forms (\ref{dualframedef}) and (\ref{DCisIframes}), it is natural to
try to generalize these expressions to the context of fusion frames
in order to obtain a definition of dual fusion frame.

The concept that satisfies the properties which are desirable for a
dual fusion frame came up and was studied in \cite{heimoanza14,
hemo14}, considering as coefficient space $\mathcal{K}_{V}^{2}$. The
reasoning there was the following:

Let $V=(V_i,v_i)_{i \in I}$ be a fusion frame. Since
$S^{-1}_{V}S_{V}=\mathcal{I}_{\Hil}$, we have the following
reconstruction formula
\begin{equation}\label{E reconst dual fusion frame canonico}
f=\sum_{i\in I}v_{i}^{2}S_{V}^{-1}\pi_{V_{i}}f \qquad  (\forall f
\in \mathcal{H}),
\end{equation} \noindent which is analogous to \eqref{framerec}. The family $(S_{V}^{-1}V_i,v_i)_{i \in I}$ is a
fusion frame (see Corollary \ref{canonicaldualfusion}) which in \cite[Definition 3.19]{caskut04} is called the
dual fusion frame of $V$, and is similar to the canonical dual frame
in the classical frame theory. As it is pointed out in \cite[Chapter
13]{CK2}, (\ref{E reconst dual fusion frame canonico}) - in contrast
to (\ref{dualframedef}) for frames - does not lead automatically to a
dual fusion frame concept.

So, instead of trying to generalize (\ref{dualframedef}), we can try
with (\ref{DCisIframes}). But in this case we find the following
obstacle. Given two fusion frames $V=(V_i,v_i)_{i \in I}$ and $W=(W_i,w_i)_{i \in I}$ for $\mathcal{H}$, with $(V_i)_{i \in I} \neq
(W_i)_{i \in I}$, the corresponding synthesis operators $D_{V}$ and
$D_{W}$ have different domains if defined - as in the original papers - on the respective subspaces $\mathcal{K}_V^2$ and $\mathcal{K}_W^2$ of $\mathcal{K}_{\Hil}^2$. Therefore the composition of $D_{W}$
with $C_{V}$ is not possible. In \cite{heimoanza14} (see also
\cite{hemo14}), to overcome this domain issue, $Q:\mathcal{K}_V^2 \longrightarrow \mathcal{K}_W^2$ is inserted
between $D_W$ and $C_V$:

\centerline{$D_{W}QC_{V}$}

\noindent  giving origin to the commonly named {\em $Q$-dual fusion
frames}. The philosophy then was to require additional conditions on
$Q$, in order to obtain the different desired properties for dual
fusion frames, similar to those which exist for dual frames. This is
why first a general $Q$ is used, asking only for boundedness of $Q$,
getting with this minimal restriction characterizations analogous to
those that exist for dual frames (see Section \ref{sec:Qdual}).

In order to have computationally convenient reconstruction formulae, the \emph{block-diagonal dual fusion frames} were introduced in
\cite{hemo14} (see also \cite{HeinekenMorillas18}). {\em Dual fusion
frame systems}  \cite{hemo14,HeinekenMorillas18}, which are a useful
tool for local processing tasks, turned out to be block-diagonal too. Studying the relation between dual fusion frames and
the bounded left-inverses of the analysis operator, a particular class of
block-diagonal dual fusion frames, called {\em component
preserving dual fusion  frames} aroused \cite{HeinekenMorillas18,heimoanza14,hemo14}. They can be obtained from the left inverse of the analysis operator and hence are easy to construct.
These are of particular importance, since the
canonical dual fusion frame is of this type. Moreover, they are crucial in
the applications, since they are optimal in the presence of erasures
\cite{hemo14, Mo17} (see Section \ref{sec:block}).

\begin{svgraybox}\label{to appr}
If one considers $\mathcal{K}_{\mathcal{H}}^{2}$ as coefficient
space $-$ as we have done in this survey so far $-$ there is no domain problem in the composition of $D_{W}$ with $C_{V}$ and hence there is in principle no need to include an operator $Q$
between them. In this case, the first idea, which probably comes to
mind when searching for a dual fusion frame definition, is generalizing the
duality concept directly from (\ref{DCisIframes}). However, this approach
does not lead to a completely satisfactory notion of duality (see Section \ref{S other approaches}),
showing that inserting an operator $Q$ different from the
identity between $D_{W}$ and $C_{W}$ cannot be avoided.
\end{svgraybox}

\subsection{$Q$-dual fusion frames}\label{sec:Qdual}

As mentioned at the beginning of this section, we are going to
present the duality theory for fusion frames developed in
\cite{heimoanza14, hemo14,HeinekenMorillas18}. But this will be done
considering the space $\mathcal{K}_{\mathcal{H}}^{2}$ as domain of
the synthesis operator, as it was done so far in this survey,
instead of considering $\mathcal{K}_{V}^{2}$ as in those articles.

The definition that overcame the initial domain problem and extends
the notion of canonical dual fusion frame introduced in
\cite{caskut04} (see subsection~\ref{Ej dual canonico}) is the
following:

\begin{definition}\label{D dual fusion frame}\cite{heimoanza14}
Let $V=(V_i,v_i)_{i \in I}$ and $W=(W_i,w_i)_{i \in I}$ be fusion
frames for $\mathcal{H}$. We say that $W$ is a dual fusion frame of
$V$ if there exists $Q \in
\mathcal{B}(\mathcal{K}_{\mathcal{H}}^{2})$ such that
\begin{equation}\label{E D dual fusion frame}
D_{W}QC_{V}=\mathcal{I}_{\Hil}.
\end{equation}
\end{definition}

The operator $Q$ is actually important in the definition. If we need
to do an explicit reference to it we say that $W$ is a $Q$-dual
fusion frame of $V$. Note that if $W$ is a $Q$-dual fusion frame of $V,$ then $V$ is a $Q^{*}$-dual fusion frame of $W$. As we will see in
Lemma~\ref{L equivalencias}, Bessel fusion sequences $V$ and $W$
that satisfy (\ref{E D dual fusion frame}), are automatically fusion
frames.

\begin{remark}\label{motivacion 1}
As we mentioned, one reason to introduce first a general
class of dual fusion frames as in Definition~\ref{D dual fusion
frame}, requiring only boundedness of the operator $Q$, is to ask
for the minimal conditions needed to obtain the different desired
properties for dual fusion frames. In particular, for this general
class we have the following lemma, which generalizes the basic
properties that are valid for dual frames. It gives equivalent
conditions for two Bessel fusion sequences to be dual fusion frames.
\end{remark}

\begin{lemma}\label{L equivalencias}\cite{heimoanza14}
Let $V=(V_i,v_i)_{i \in I}$ and $W=(W_i,w_i)_{i \in I}$ be Bessel
fusion sequences for $\mathcal{H}$, and let $Q \in
\mathcal{B}(\mathcal{K}_{\mathcal{H}}^{2})$. Then the following
statements are equivalent:
\begin{enumerate}
  \item[(i)] $\mathcal{I}_{\Hil}=D_{W}Q C_{V}.$
  \item[(ii)] $\mathcal{I}_{\Hil}=D_{V}Q^* C_{W}.$
  \item[(iii)] $\prode{f,g}=\prode{Q^*C_{W}f,C_{V}g}$
  for all $f, g \in \mathcal{H}$.
 \item[(iv)] $\prode{f,g}=\prode{QC_{V}f,C_{W}g}$
  for all $f, g \in \mathcal{H}$.
 \item[(v)] $C_{V}$ is injective, $D_{W}Q$ is surjective and
  $\paren{C_{V}D_{W}Q}^{2}=C_{V}D_{W}Q$.
 \item[(vi)] $C_{W}$ is injective, $D_{V}Q^*$ is surjective and
  $\paren{C_{W}D_{V}Q^*}^{2}=C_{W}D_{V}Q^*$.
\end{enumerate}
In case any of these equivalent conditions are satisfied, $V$ and W
are fusion frames for $\mathcal{H}$, W  is a $Q$-dual fusion frame
of $V$ and $V$ is a $Q^{*}$-dual fusion frame of $W.$
\end{lemma}

\begin{proof}
First, note that according to Theorem \ref{synthesisthm}, $D_W$ and $C_V$ are bounded. 

(i)$\Leftrightarrow$(ii) This follows from taking adjoints.

(iii)$\Leftrightarrow$(ii) If (iii) is satisfied, then for each $f\in \Hil$ we obtain $$\langle (\mathcal{I}_{\mathcal{H}}  - D_V Q^* C_W) f,g \rangle =0 \qquad (\forall g \in
\mathcal{H}),$$ hence $(\mathcal{I}_{\mathcal{H}}  - D_V Q^* C_W) f = 0$ and (ii) follows. The converse is clear. 

(iv)$\Leftrightarrow$(i) This as shown analogously to (iii)$\Leftrightarrow$(ii). 

(i)$\Rightarrow$(v) By (i), $C_{V}$ is injective, $D_{W}Q$ is
surjective and $$\paren{C_{V}D_{W}Q}^{2}=C_{V} \paren{D_{W}QC_{V}}D_{W}Q
=C_{V}D_{W}Q.$$

(v)$\Rightarrow$(i) If $(C_{V}D_{W}Q)^{2}=C_{V}D_{W}Q$, then
$$\mathcal{K}_{\Hil}^{2}=\mathcal{N}(C_{V}D_{W}Q)\oplus \mathcal{R}(C_{V}D_{W}Q).$$
Since $C_{V}$ is injective we have
$\mathcal{N}(C_{V}D_{W}Q)=\mathcal{N}(D_{W}Q)$ and so
$$\mathcal{K}_{\Hil}^{2}=\mathcal{N}(D_{W}Q)\oplus\mathcal{R}(C_{V}D_{W}Q).$$
This implies $\mathcal{H}= \big\lbrace D_{W}Q(f_{i})_{i \in I}: (f_{i})_{i
\in I} \in \mathcal{R}(C_{V}D_{W}Q) \big\rbrace$. Now, let $f \in \mathcal{H}$ with $f=D_{W}Q(f_{i})_{i \in
I}$ for some $(f_{i})_{i \in I} \in \mathcal{R}(C_{V}D_{W}Q)$. Then 
$$D_{W}QC_{V}f= D_{W}QC_{V}D_{W}Q(f_{i})_{i \in
I}=D_{W}Q(f_{i})_{i \in I}=f.$$
(ii)$\Leftrightarrow$(vi) is proved analogously.

Finally, if (i)-(vi) are satisfied, then $\mathcal{R}(D_{V}) = \Hil = \mathcal{R}(D_W)$, hence both $V$ and $W$ are fusion frames by Theorem \ref{fusionframechar}. 
\end{proof}

\begin{remark}
Two fusion frames $V$ and $W$ for $\mathcal{H}$ are always $RL$-dual
fusion frames for any $L$, being a left-inverse of $C_V$, and $R$ a right-inverse of $D_W$ (compare with Lemma \ref{leftright}). This shows that we can always do the analysis
with one of them and the synthesis with the other. But this happens
in the general framework of Definition~\ref{D dual fusion frame}
where we do not impose any additional condition on $Q$.
\end{remark}

\begin{remark}
Duality of fusion frames was studied viewing them as projective reconstruction systems
\cite{MaRuSto12a, MaRuSto12b, Mo11, Mo13}. But projective
reconstruction systems are not closed under duality, more precisely,
there exist projective reconstruction systems with non projective
canonical dual or without any projective dual \cite{MaRuSto12b}. As
we will see these problems are not present if we use
$Q$-dual fusion frames. For a more detail discussion of dual projective
reconstruction systems in relation with $Q$-dual fusion frames we
refer the reader to \cite{hemo14}.
\end{remark}

\subsection{Block-diagonal and component preserving $Q$-dual fusion frames}\label{sec:block}

We will now present two special types of $Q$-dual fusion frames that make the reconstruction formula that follows from (\ref{E D
dual fusion frame}) simpler (compare to Subsection \ref{Operators between Hilbert direct sums}).

\begin{definition}\label{D bd and cp dual fusion frame}\cite{hemo14}
Let $V=(V_i,v_i)_{i \in I}$ and $W=(W_i,w_i)_{i \in I}$ be fusion
frames for $\mathcal{H}$. If in Definition~\ref{D dual fusion frame}
$Q$ is block-diagonal we say that $W$ is a \emph{block-diagonal dual
fusion frame} of $V$. In case $Q$ is $\mathcal{B}(\mathcal{K}_{V}^{2},\mathcal{K}_{W}^{2})$-component preserving (meaning that $QM_{j}(\mathcal{K}_{V}^{2}) =
M_{j}(\mathcal{K}_{W}^{2})$), we say that $W$ is a \emph{component
preserving dual fusion frame} of $V$.
\end{definition}

Note that in (\ref{E D dual fusion
frame}) it is sufficient to know how $Q$ maps $\mathcal{R}(C_{V}) \subseteq \mathcal{K}_{V}^{2}$
to $\mathcal{N}(D_{W})^{\perp} \subseteq \mathcal{K}_{W}^{2}$.

\begin{remark}\label{motivacion 2}
Another motivation for introducing the notion of duality as in
Definition~\ref{D dual fusion frame} is to obtain flexibility,
therefore asking for restrictions only when needed. The general
framework provided by Definition~\ref{D dual fusion frame} allows to
adjust to the problem at hand. This is another reason to start with
the most general class and then naturally arise the particular
classes with which we work here: block-diagonal and component
preserving dual fusion frames. As we will see in Lemma~\ref{L
Vi=ApiWj entonces W,w dual fusion frame}, $Q$ is component preserving for dual fusion frames obtained from the bounded left inverses of
$C_{V}$. Also, $Q$ is block-diagonal for dual fusion frame systems (see Definition \ref{D dual fusion frame system}).
\end{remark}

By definition of block-diagonal operators (see Section \ref{Operators between Hilbert direct sums}), we  immediately see the following: 

\begin{lemma}\cite{HeinekenMorillas18}
Let $V = (V_i,v_i)_{i\in I}$ be a fusion frame. A fusion frame
$W=(W_i,w_i)_{i\in I}$ is a block-diagonal dual fusion frame with respect to $\bigoplus_{i\in I} Q_i \in \mathcal{B}(\mathcal{K}_{\Hil}^2)$ if and
only if
\begin{equation}\label{dualdef}
\sum_{i\in I} v_i w_i \pi_{W_i} Q_i \pi_{V_i} = \mathcal{I}_{\Hil}.
\end{equation}
\end{lemma}

Note that  $(W_i,w_i)_{i \in I}$ is a block-diagonal dual fusion frame of
$V=(V_i,v_i)_{i \in I}$ with respect to $\bigoplus_{i\in I} Q_i$ if and only if $(W_i,c_{i}w_i)_{i \in I}$ is
a block-diagonal dual fusion frame of $V$ with respect to $\bigoplus_{i \in
I}(\frac{1}{c_{i}}Q_{i})$, as long as $(c_i)_{i\in I}$ is semi-normalized. Both dual fusion frames lead to the
same reconstruction formula. This freedom for the weights is
desirable because we can select those $(w_i)_{i \in I}$ such that
the pair $(W_i,w_i)_{i \in I}$ is the most suitable to treat
simultaneously another problem not related to the reconstruction
formula. Regarding (\ref{dualdef}), we could think of the terms $w_{i}v_{i} \Vert \pi_{W_i} Q_{i} \pi_{V_i} \Vert$ as a measure of the importance of the pair of subspaces $V_{i}$ and $W_{i}$ for the reconstruction.

\

In the following sections we show that there is a close relation of
dual fusion frame systems with dual frames, a fact that supports the idea that
Definition~\ref{D dual fusion frame system} is the proper definition
of dual fusion frame systems. Consequently, this close relation also
reveals that the inclusion of a $``Q"$ in a definition of duality
for fusion frames as in Definition~\ref{D dual fusion frame} is
natural.

\subsection{Dual fusion frame systems}\label{S dual fusion frame systems}

With the aim of allowing local processing in the setting described
before, we define and study in this section the concept of dual
fusion frame systems, see Definition \ref{def:fusionframeseystem}. In order to do that, we will consider the following operator, which we introduced in \cite{hemo14}, and which
establishes the connection between the synthesis operator of a
fusion frame system and the synthesis operator of its associated
frame.

Let $(V_i, v_i, \varphi^{(i)})_{i\in I}$ be a Bessel fusion system (see Definition \ref{def:fusionframeseystem}). Then, by Lemma \ref{blockdiagonalbounded}, $\bigoplus_{i\in I} D_{\varphi^{(i)}}$ is a bounded operator with $\Vert \bigoplus_{i\in
I} D_{\varphi^{(i)}} \Vert \leq \sqrt{B}$. Moreover, each component $D_{\varphi^{(i)}}$ is surjective onto $V_i$, and since it holds $\sup_{i\in I} \Vert D_{\varphi^{(i)}}^{\dagger} \Vert = \sup_{i\in I} \Vert C_{\varphi^{(i)}} S_{\varphi^{(i)}}^{-1} \Vert \leq \sqrt{B}/A$, the block-diagonal operator $\bigoplus_{i\in I}D_{\varphi^{(i)}}^{\dagger}$ is bounded as well. In particular, for every $(f_i)_{i\in I} \in \mathcal{K}_V^2$, we see that $(D_{\varphi^{(i)}}^{\dagger}f_i)_{i\in I} \in \paren{\sum_{i \in I} \oplus
\ell^2(J_{i})}_{\ell^{2}}$, hence $(f_i)_{i\in I} = (D_{\varphi^{(i)}} D_{\varphi^{(i)}}^{\dagger} f_i)_{i\in I} = \bigoplus_{i\in I}D_{\varphi^{(i)}} (D_{\varphi^{(i)}}^{\dagger} f_i)_{i\in I}$, which implies that $\mathcal{R}(\bigoplus_{i\in
I}D_{\varphi^{(i)}})=\mathcal{K}_{V}^{2}$. The adjoint of $\bigoplus_{i\in I}D_{\varphi^{(i)}}$ is given by $\bigoplus_{i\in I}C_{\varphi^{(i)}}$ and satisfies 
$$A \Vert (g_{i})_{i \in I} \Vert \leq \Big\Vert \bigoplus_{i\in I}C_{\varphi^{(i)}}(g_{i})_{i \in I} \Big\Vert \leq
B\Vert(g_{i})_{i \in I}\Vert \qquad (\forall (g_i)_{i\in I} \in \mathcal{K}_V^2).$$
Assume that for each $i\in I$, $\widetilde{\varphi^{(i)}}$ is a dual frame of
$\varphi^{(i)}$ with upper frame bound $\widetilde{B}_{i}$, such that
${\rm sup}_{i \in I}\widetilde{B}_{i} < \infty$. Then, as above, $\bigoplus_{i\in I}C_{\widetilde{\varphi^{(i)}}}$ is bounded and
\begin{equation} \label{eq:locrec1} \bigoplus_{i\in
I}D_{\varphi^{(i)}}C_{\widetilde{\varphi^{(i)}}} = \bigoplus_{i\in
I}\pi_{V_{i}}.\end{equation}
If we identify $\paren{\sum_{i \in I} \oplus
\ell^2(J_{i})}_{\ell^{2}}$ with the domain of $D_{v\varphi}$, where $v\varphi$ $-$ as usual $-$ denotes the global family $(v_i \varphi_{ij})_{i\in I, j\in J_i}$, then we see that 
\begin{equation}\label{opid1}
D_{v\varphi}=D_V\bigoplus_{i\in
I}D_{\varphi^{(i)}} \quad \text {and, by adjointing,} \quad  C_{v\varphi} = \big( \bigoplus_{i\in
I}C_{\varphi^{(i)}} \big) C_V,
\end{equation}
see also \cite{koeba23}. Moreover, by multiplying the left identity of (\ref{opid1}) with the operator  $\bigoplus_{i\in I} C_{\widetilde{\varphi^{(i)}}}$ from the right, we obtain
\begin{equation}\label{opid2}
D_V = D_{v\varphi}\bigoplus_{i\in
I}C_{\widetilde{\varphi^{(i)}}} \quad \text {and, by adjointing,} \quad C_V = \big( \bigoplus_{i\in
I} D_{\widetilde{\varphi^{(i)}}}\big) C_{v\varphi},   
\end{equation}
see also \cite{koeba23} for more details on operator identities for Bessel fusion systems.

The latter observations motivate the following definition. 

\begin{definition}\label{D dual fusion frame system}\cite{hemo14}
The fusion frame system $(W_i,w_i,\psi^{(i)})_{i \in I}$ is {\em a dual
fusion frame system} of the fusion frame system $(V_i,v_i,\varphi^{(i)})_{i \in I}$ if
$(W_i,w_i)_{i \in I}$ is a $\bigoplus_{i\in
I}D_{\psi^{(i)}}C_{\varphi^{(i)}}$- dual fusion frame of
$(V_i,v_i)_{i \in I}$.
\end{definition}

By  (\ref{eq:locrec1}), having a pair of dual fusion frame systems means that the global families  $(w_{i}\psi_{ij})_{i\in I,j\in J_i}$ and $(v_i\varphi_{ij})_{i\in I,j\in J_i}$ form a dual frame pair. Note that $(\mathcal{K}_{V}^{2})^{\perp}\subseteq\mathcal{N}(\bigoplus_{i\in
I}D_{\psi^{(i)}}C_{\varphi^{(i)}})$. In case  $\bigoplus_{i\in I}D_{\psi^{(i)}}C_{\varphi^{(i)}}$ in
Definition~\ref{D dual fusion frame system}  is $\mathcal{B}(\mathcal{K}_{V}^{2},\mathcal{K}_{W}^{2})$-component
preserving, then we call $(W_i,w_i,\psi^{(i)})_{i \in I}$ a
\emph{component preserving dual fusion frame system} of
$(V_i,v_i,\varphi^{(i)})_{i \in I}.$

\subsection{Relation between block-diagonal dual frames,
dual fusion frame systems and dual frames} 

In view of Definition \ref{D dual fusion frame system} we see that each dual fusion frame system
can be linked to a block-diagonal dual fusion frame. Conversely, under certain assumptions, we
can also associate to a block-diagonal dual fusion frame pair a dual fusion frame system pair. In order to see this we need Corollary~\ref{C Q=CGCF*} which follows from the lemma below.

\begin{lemma}\label{L A=CGCF*}\cite{HeinekenMorillas18}
If $A \in \mathcal{B}(\mathcal{H}, \mathcal{K})$ with closed range, then there exists a frame $\varphi$ for $\mathcal{H}$ and a frame $\psi$
for $\mathcal{K}$ with the same index set and
$A=D_{\psi}C_{\varphi}$. In particular, $\varphi$
and $\psi$ can be chosen in such a way that $\varphi$ has frame bounds $1$ and $2$, and $\psi$ has frame bounds $\min\set{1, \Vert A^{\dag} \Vert^{-2}}$ and $\max\set{1, \Vert A \Vert^{2}}$ respectively.
\end{lemma}

\begin{proof}
Let $\varphi$ be any frame for $\mathcal{H}$ and $\widetilde{\varphi}$ be any dual frame of $\varphi$. By \cite[Proposition 5.3.1]{ole1n}, $A\widetilde{\varphi}$ is a frame for $\mathcal{R}(A)$ with the same index set. We also have
$A = D_{A\widetilde{\varphi}} C_{\varphi}$. In particular, if $\varphi$ is a Parseval frame for $\mathcal{H}$ and $\widetilde{\varphi}=\varphi$ is the canonical dual frame of $\varphi$, then $\varphi$ has frame bound $1$ and $A\widetilde{\varphi}$ has frame bounds $\Vert A^{\dag} \Vert ^{-2}$ and $\Vert A \Vert^2$. In case $\mathcal{R}(A) =  \mathcal{K}$, the statement follows immediately. In case $\mathcal{R}(A)\neq \mathcal{K}$, we will extend the families $\varphi$ and $A\widetilde{\varphi}$ in a suitable way, so that the statement follows.

To this end, assume that $\mathcal{R}(A)\neq \mathcal{K}$. Let $\psi=(\psi_{i})_{i\in J}$ be any frame for $\mathcal{R}(A)^{\perp}$. If we consider the family $\Tilde{\psi} = (\Tilde{\psi})_{i\in J}$, where $\Tilde{\psi}_i = 0$ for all $i$, then it holds $D_{\psi}C_{\widetilde{\psi}}=0$. On the other hand, we can construct $\widetilde{\psi}$ with not all of its elements
being equal to zero so that $D_{\psi}C_{\widetilde{\psi}}=0$.
For this, we consider a frame
$\psi=(\psi_{j})_{j \in J}$ for $\mathcal{R}(A)^{\perp}$ that
is not a basis. Let $c=(c_{m})_{m \in \mathbb{M}}$ be an orthonormal
basis for $\mathcal{N}(D_{\psi}) \subset \ell^{2}(J)$ where
$\mathbb{M}=\mathbb{N}$ or $\mathbb{M}=\{1, \ldots, M\}$. Let
$e=(e_{l})_{l \in \mathbb{L}}$ be an orthonormal basis for
$\mathcal{H}$ where $\mathbb{L}=\mathbb{N}$ or $\mathbb{L}=\{1,
\ldots, L\}$. Let $\mathbb{I} \subseteq \mathbb{M} \cap
\mathbb{L}$ with $\vert\mathbb{I}\vert <\infty$ and define $\widetilde{\psi}_{j}=\sum_{l \in \mathbb{I}}\overline{c_{l}(j)}e_{l}$ for each $j \in J$. In other words, $\widetilde{\psi}_{j} = D_e C_c \delta_j$, where $(\delta_j)_{j\in I}$ is the standard basis of $\ell^2$. By the injectivity of the involved operators 
not all vectors $\widetilde{\psi}_{j}$ are equal to $0$. By the Cauchy-Schwarz inequality we see that 
\begin{align*}
\sum_{j \in J} \vert \langle f, \widetilde{\psi}_{j}\rangle \vert^{2} &= \sum_{j \in J} \Big\vert \Big\langle f,\sum_{l \in \mathbb{I}}\overline{c_{l}(j)} e_{l} \Big\rangle \Big\vert^{2} = \sum_{j \in J} \Big\vert \sum_{l \in \mathbb{I}}c_{l}(j)\langle f,e_{l}\rangle \Big\vert^{2}\\ & \leq \sum_{j \in J}\sum_{l \in \mathbb{I}}|c_{l}(j)\vert^{2}\sum_{l \in \mathbb{I}}\vert \langle f,e_{l}\rangle \vert^{2} \leq \sum_{l \in \mathbb{I}}\sum_{j \in J} \vert c_{l}(j) \vert^{2} \Vert f \Vert^{2} = \vert \mathbb{I} \vert  \Vert f \Vert^{2}.
\end{align*}
\noindent Therefore, $\widetilde{\psi} = (\widetilde{\psi}_{j})_{j \in J}$ is a Bessel sequence
with Bessel bound $\vert \mathbb{I}\vert$. Note that if $d \in \ell^{2}(J)$ and $f \in
\mathcal{H}$, then
\begin{align*}
\langle D_{\widetilde{\psi}}d, f \rangle
&= \Big\langle \sum_{j \in J}d(j)\sum_{l \in
\mathbb{I}}\overline{c_{l}(j)}e_{l}, f \Big\rangle \\ &= \sum_{j \in
J}d(j)\sum_{l \in \mathbb{I}}\overline{c_{l}(j)}\langle e_{l},
f \rangle = \Big\langle d , \sum_{l \in \mathbb{I}}\langle f,
e_{l}\rangle c_{l} \Big\rangle.
\end{align*}
Hence $C_{\widetilde{\psi}}f = \sum_{l \in \mathbb{I}}\langle f,
e_{l}\rangle c_{l}$ and $D_{\psi}C_{\widetilde{\psi}}f=\sum_{l \in
\mathbb{I}}\langle f, e_{l}\rangle D_{\psi}c_{l}=0$ since
$c_{l} \in \mathcal{N}(D_{\psi})$ for each $l \in
\mathbb{I}$. Finally, $(\varphi , \widetilde{\psi})$ is a frame for
$\mathcal{H}$, $(A\widetilde{\varphi}, \psi)$ one for $\mathcal{K}$,
$\vert (\varphi,\widetilde{\psi}) \vert = \vert (A\widetilde{\varphi},\psi) \vert$ and $A = D_{(A\widetilde{\varphi},\psi )} C_{(\varphi,\widetilde{\psi})}$.

In particular, if we choose $\varphi$ and
$\psi$ to be Parseval, $\widetilde{\varphi}=\varphi$ and $|\mathbb{I}|=1$, then
$(\varphi , \widetilde{\psi})$ is a frame with frame bounds $1$ and $2$, and
$(A\widetilde{\varphi},\psi)$ is a frame with frame bounds
$\min\set{1, \Vert A^{\dag} \Vert^{-2}}$ and  $\max\set{1,\Vert A \Vert^{2}}$, respectively.
\end{proof}

\begin{corollary}\label{C Q=CGCF*}\cite{HeinekenMorillas18} 
Let $(V_i)_{i\in I}$ and $(W_i)_{i\in I}$ be two families of closed subspaces in $\Hil$ and assume that $Q = \bigoplus_{i \in I} Q_{i} \in \mathcal{B}(\mathcal{K}_{\Hil}^2)$ with $\mathcal{N}(Q_{i})=V_{i}^{\perp}$, $\mathcal{R}(Q_{i})=W_{i}$ for each $i \in I$ and $\inf_{i \in I}\|Q_{i}^{\dag}\|^{-1} > 0$. Then, for each $i\in I$, there exists a frame $\varphi^{(i)}$ for $V_i$ with frame bounds $A_{i}, B_{i}$, satisfying $ 0 < {\rm inf}_{i \in I} A_{i} \leq {\rm sup}_{i \in I}B_{i} < \infty$, and a frame $\psi^{(i)}$ for $W_i$ with equal index set and frame bounds $\widetilde{A}_{i}, \widetilde{B}_{i}$, satisfying $ 0 < {\rm inf}_{i \in I}\widetilde{A}_{i} \leq {\rm sup}_{i \in I}\widetilde{B}_{i} < \infty$, such that $Q=\bigoplus_{i \in I}D_{\psi^{(i)}}C_{\varphi^{(i)}}$.
\end{corollary}

\begin{proof}
Assume that $Q=\bigoplus_{i \in I}Q_{i} \in \mathcal{B}(\mathcal{K}_{\Hil}^2)$ with
$\mathcal{N}(Q_{i})=V_{i}^{\perp}$, $\mathcal{R}(Q_{i})=W_{i}$ for each $i \in I$ and $\inf_{i \in I}\|Q_{i}^{\dag}\|^{-2} > 0$. Recall, that by Lemma \ref{blockdiagonalbounded}, we also have $\sup_{i\in I} \Vert Q_i \Vert < \infty$. By applying Lemma~\ref{L A=CGCF*},
for each $i \in I$ there exists a frame
$\varphi^{(i)}$ for $V_i$ with frame bounds $A_{i} = 1$ and $B_{i} =2$ and a frame $\psi^{(i)}$ for $W_i$ with frame bounds $\widetilde{A}_{i}= \min\big\lbrace 1, \|Q_{i}^{\dag}\|^{-2} \big\rbrace$ and 
$\widetilde{B}_{i}=\max\set{1, \Vert Q_{i} \Vert^{2}}$ such that
$Q_{i} \vert_{V_{i}}=(D_{\psi^{(i)}}C_{\varphi^{(i)}})\vert_{V_{i}}$. In particular, we have $Q=\bigoplus_{i \in I}D_{\psi^{(i)}}C_{\varphi^{(i)}}$, as well as $\inf_{i\in I} A_i = 1$, $\sup_{i\in I} B_i = 2$, $\inf_{i\in I} \widetilde{A}_i \geq \min\big\lbrace 1, \inf_{i\in I} \Vert Q_i^{\dagger} \Vert^{-2} \big\rbrace >0$ and $\sup_{i\in I} \widetilde{B}_i \leq \max \big\lbrace 1, \sup_{i\in I} \Vert Q_i \Vert \big\rbrace < \infty$.
\end{proof}

After these preparations, we are able to prove an important relation between block-diagonal dual fusion frames and dual fusion frame systems.

\begin{theorem}\label{T relacion entre dual y dual system}\cite{HeinekenMorillas18} 
Let $V=(V_i,v_i)_{i \in I}$ be a fusion frame for $\Hil$ and let
$W=(W_i,w_i)_{i \in I}$ be a $\bigoplus_{i \in I}Q_{i}$-dual fusion frame of
$V$, where $\mathcal{N}(Q_{i})=V_{i}^{\perp}$, $\mathcal{R}(Q_{i})=W_i$ for each $i\in I$ and $\inf_{i \in I}\|Q_{i}^{\dag}\|^{-1} > 0$. Then there exist local frames $\varphi^{(i)}$ for $V_i$ and $\psi^{(i)}$ for $W_i$ (for all $i\in I$) such that $(V_i,v_i,\varphi^{(i)})_{i \in I}$ is a fusion frame system in $\Hil$ and $(W_i,w_i,\psi^{(i)})_{i \in
I}$ is a dual fusion frame system of $(V_i,v_i,\varphi^{(i)})_{i \in I}$. In this case $Q=\bigoplus_{i\in I}D_{\psi^{(i)}}C_{\varphi^{(i)}}$.
\end{theorem}

\begin{proof} Combining Definition~\ref{D dual
fusion frame}, Corollary~\ref{C Q=CGCF*} and Definition~\ref{D dual
fusion frame system} yields the claim.
\end{proof}

The following theorem establishes the connection between dual fusion frame system pairs and and dual frame pairs.

\begin{theorem}\label{T dual fusion frame systems}\cite{HeinekenMorillas18}
Let $(V_i,v_i,\varphi^{(i)})_{i \in I}$ and $(W_i,w_i,\psi^{(i)})_{i \in I}$ be two Bessel fusion systems for $\Hil$, where $\varphi^{(i)} = (\varphi_{ij})_{j\in J_i}$ and $\psi^{(i)} = (\psi_{ik})_{k\in K_i}$ for $i\in I$, and let $v\varphi = (v_i\varphi_{ij})_{i\in I, j\in J_i}$ and $w\psi = (w_i\psi_{ik})_{i\in I, k\in K_i}$ be their respective associated global families. If $\vert K_i \vert=\vert J_i\vert$ for all $i \in I$, then the following are equivalent:
  \begin{enumerate}
    \item[(i)] $v\varphi$ and $w\psi$ is a dual frame pair for $\mathcal{H}$.
    \item[(ii)] $(V_i,v_i,\varphi^{(i)})_{i \in I}$ and $(W_i,w_i,\psi^{(i)})_{i \in I}$ is a dual fusion frame system pair for $\mathcal{H}$.
  \end{enumerate}
\end{theorem}

\begin{proof}
This follows from Theorem \ref{fusframesysTHM}, Lemma~\ref{L equivalencias}, \cite[Lemma 6.3.2]{ole1n}, and the operator identity $D_{w\psi}C_{v\varphi} = D_{W}(\bigoplus_{i\in
I}D_{\psi^{(i)}}C_{\varphi^{(i)}})C_{V}$ which is deduced immediately from (\ref{opid1}).
\end{proof}

\subsection{Dual families}\label{S dual families}

In this subsection we consider dual fusion frame families. In resemblance to dual frames, we will see that component preserving dual fusion frames are related to the bounded left
inverses of the analysis operator. We begin with the following preparatory result.

\begin{lemma}\label{leftright}
Let $V$ be a fusion frame for $\Hil$. Then
\begin{enumerate}
\item[(i)] \cite{heimoanza14} The set $\mathcal{B}_L(C_V)$ of all bounded left-inverses of $C_V$ is given by
$$\Big\lbrace S_V^{-1} D_V + L(\mathcal{I}_{\mathcal{K}_{\Hil} ^2} - C_V S_V^{-1} D_V ) \Big\vert \,  L \in \mathcal{B}(\mathcal{K}_{\Hil}^2 , \Hil) \Big\rbrace .$$
\item[(ii)] The set $\mathcal{B}_R (D_V)$ of all bounded right-inverses of $D_V$ is given by
$$\Big\lbrace C_V S_V^{-1} + (\mathcal{I}_{\mathcal{K}_{\Hil}^2} - C_V S_V^{-1} D_V )R \Big\vert \, R \in \mathcal{B}( \Hil , \mathcal{K}_{\Hil}^2) \Big\rbrace .$$
\end{enumerate}
\end{lemma}

\begin{proof}
(i) For any $L \in \mathcal{B}(\mathcal{K}_{\Hil}^2 , \Hil)$ we have
\begin{align}
\Big[S_V^{-1} D_V + L \big( \mathcal{I}_{\mathcal{K}_{\Hil}^2} - C_V
S_V^{-1} D_V \big) \Big] C_V = \mathcal{I}_{\mathcal{H}} + L C_V - L
C_V =  \mathcal{I}_{\mathcal{H}} . \notag
\end{align}
Conversely, assume that $A \in \mathcal{B}_L(C_V)$. Then
\begin{align}
A = S_V^{-1} D_V + A - S_V^{-1} D_V = S_V^{-1} D_V + A \big(
\mathcal{I}_{\mathcal{K}_{\Hil}^2} - C_V S_V^{-1} D_V \big) \notag
\end{align}
is contained in the proclaimed set.

(ii) is proved analogously: For any $R \in \mathcal{B}(\Hil,
\mathcal{K}_{\Hil}^2)$ we see that
\begin{align}
D_V \Big[ C_V S_V^{-1} + \big( \mathcal{I}_{\mathcal{K}_{\Hil}^2} -
C_V S_V^{-1} D_V \big)R \Big] = \mathcal{I}_{\mathcal{H}} + D_V R -
D_V R=  \mathcal{I}_{\mathcal{H}} . \notag
\end{align}
On the other hand, for every $A \in \mathcal{B}_R (D_V)$ we see that
\begin{align}
A = C_V S_V^{-1} + A - C_V S_V^{-1} = C_V S_V^{-1} + \big(
\mathcal{I}_{\mathcal{K}_{\Hil}^2} - C_V S_V^{-1} D_V \big) A \notag
\end{align}
is contained in the proclaimed set.
\end{proof}

\begin{remark}\label{newsynthesis}
Let us consider the following special case of the above
characterizations: If $V$ is a fusion Riesz basis, then by Theorem
\ref{fusionrieszbasischar}, $C_V:\Hil \longrightarrow
\mathcal{K}_V^2$ is a bounded bijection, and thus
$$\mathcal{I}_{\mathcal{K}_{\Hil}^2} - C_V S_V^{-1} D_V =
\mathcal{I}_{\mathcal{K}_{\Hil}^2} - \pi_{\mathcal{K}_V^2} =
\pi_{(\mathcal{K}_V^2)^{\perp}}.$$ Thus, if $V$ is a fusion Riesz
basis, then any left-inverse operator of $C_V$ is $S_V^{-1} D_V$
plus some remainder operator $L\pi_{(\mathcal{K}_V^2)^{\perp}}$. This remainder operator can be seen as the trade-off for our more general definition of $D_V$ with domain $\mathcal{K}_{\Hil}^2$. If we defined $D_V$ as an operator
$\mathcal{K}_V^2 \longrightarrow \Hil$, then we would have
$\mathcal{I}_{\mathcal{K}_V^2}$ instead of
$\mathcal{I}_{\mathcal{K}_{\Hil}^2}$ in the characterizations from
Lemma \ref{leftright} and the remainder operator would vanish in
this case, since $\mathcal{I}_{\mathcal{K}_V^2} - C_V S_V^{-1}
D_V = 0$.
\end{remark}

We continue with the following lemmata in order to give a characterization of component preserving dual fusion frame pairs.

\begin{lemma}\label{L W,w dual fusion frame entonces Vi=ApiWj}\cite{heimoanza14}
Let $V=(V_i,v_i)_{i \in I}$ be a fusion frame for $\mathcal{H}$. If
$(W_i,w_i)_{i \in I}$ is a $Q$-component preserving dual fusion
frame of $V$ then, for every $i\in I$, $W_{i}=LM_{i}\mathcal{K}_{V}^{2}$ for each $i\in
I$, where $L=D_{W}Q \in \mathcal{B}_L(C_V)$.
\end{lemma}

\begin{proof}
Let $Q\in \mathcal{B}(\mathcal{K}_{\mathcal{H}}^{2})$ be $\mathcal{B}(\mathcal{K}_{V}^{2},\mathcal{K}_{W}^{2})$-component preserving operator such that $D_{W}QC_{V}=\mathcal{I}_{\Hil}$ and let
$L=D_{W}Q$. Using that $Q$ is
$\mathcal{B}(\mathcal{K}_{V}^{2},\mathcal{K}_{W}^{2})$-component preserving, we see that

\centerline{$LM_{i}(\mathcal{K}_{V}^{2})=D_{W}QM_{i}(\mathcal{K}_{V}^{2})=D_{W}M_{i}(\mathcal{K}_{W}^{2})=W_{i}$}
\noindent for each $i \in I$.
\end{proof}

A reciprocal of the latter result is the following.

\begin{lemma}\label{L Vi=ApiWj entonces W,w dual fusion frame}\cite{heimoanza14}
Let $V=(V_i,v_i)_{i \in I}$ be a fusion frame for $\mathcal{H}$, $L
\in \mathcal{B}_L(C_V)$, $W_{i}=LM_{i}(\mathcal{K}_{V}^{2})$ and $w_{i} > 0$ for
each $i \in I$. If $W=(W_i,w_i)_{i \in I}$ is a Bessel fusion
sequence and

\centerline{$Q_{L,w}: \mathcal{K}_{\mathcal{H}}^{2} \rightarrow
\mathcal{K}_{\mathcal{H}}^{2} \text{ , }Q_{L,w}(f_j)_{j \in
I}=\paren{\frac{1}{w_{i}}LM_{i}(f_j)_{j \in I}}_{i \in I},$}

\noindent is a well defined bounded operator, then $W$ is a
$Q_{L,w}$-component preserving dual fusion frame of $V.$
\end{lemma}

\begin{proof} From the hypotheses, $Q_{L,w}$ is $\mathcal{B}(\mathcal{K}_{V}^{2},\mathcal{K}_{W}^{2})$-component
preserving and $L\vert_{\mathcal{K}_{V}^{2}} =(D_{W}Q_{L,w})\vert_{\mathcal{K}_{V}^{2}}$. Since $L\in\mathcal{B}_L(C_V)$,
$D_{W}Q_{L,w}C_{V}=\mathcal{I}_{\Hil}$. So $W$ is a $Q_{L,w}$-component
preserving dual fusion frame of $V$.
\end{proof}

\begin{remark}\label{R V Bessel Q acotada}
Let $L$, $V$ and $W$ be as in Lemma \ref{L Vi=ApiWj entonces W,w dual fusion frame}.

\noindent (a) Since $L\in \mathcal{B}(\mathcal{K}_{\Hil}^2, \Hil)$, we know from Subsection \ref{Operators between Hilbert direct sums} that the canonical matrix representation of $L$ is given by an operator valued row matrix $[\dots \, L_{i-1} \, L_i \, L_{i-1} \, \dots]$, where $L_i \in \mathcal{B}(\Hil)$ for every $i\in I$. Note that $Q_{L,w} = \bigoplus_{i\in I} \frac{1}{w_i} L_i$.

\noindent (b) We can give
the following sufficient conditions for $W$
being a Bessel fusion sequence and for $Q_{L,w}$ being a
well defined bounded operator:

(1) Let $\gamma(A)$ denote the the reduced minimum modulus of an operator $A\in \mathcal{B}(\Hil_1, \Hil_2)$ (see also the paragraph after Example \ref{nogosurjective}). 
Assume that $w_{i}^{-2}\gamma(LM_{i})^{2} \geq \delta >0$ for all $i \in I$. Since
$(M_{i}\mathcal{K}_{V}^2,1)_{i \in I}$ is an orthonormal fusion basis
for $\mathcal{K}_{V}^2$, by \cite[Theorem 3.6]{rust08}
$W$ is a Bessel fusion sequence for
$\mathcal{H}$ with Bessel fusion bound $\delta^{-1}\|L\|^2$.

(2) If $w$ is semi-normalized (i.e. $\exists \delta >0$: $w_i > \delta$ for each $i \in I$), then $Q_{L,w}$ is a well-defined bounded operator with
$\|Q_{L,w}\| \leq \delta^{-1} \|L\|$. Indeed, if $f=(f_i)_{i\in I} \in \mathcal{K}_{\mathcal{H}}^2$, then
$$\Vert Q_{L,w} f \Vert_{\Hil}^2 = \sum_{i \in I} \Vert w_i^{-1} L M_{i} f \Vert_{\Hil}^2 
\leq \frac{\|L\|^{2}}{\delta^{2}}\sum_{i \in
I}\|M_{i}f\|_{\mathcal{K}_{\Hil}^2}^{2} = \frac{\|L\|^{2}}{\delta^{2}}\|f\|_{\mathcal{K}_{\Hil}^2}^{2}.$$
In particular, the latter is true in case $\vert I \vert < \infty$.
\end{remark}

We are now ready to prove the following characterization of component preserving dual fusion frames. Note that the result below demonstrates that component preserving dual fusion frames can be obtained in a similar manner as in the vectorial case.

\begin{theorem}\label{T W,w dual fusion
frame sii Vi=ApiWj}\cite{hemo14}  Let
$V=(V_i,v_i)_{i \in I}$ be a fusion frame for $\mathcal{H}$. If $(w_{i})_{i \in I}$ is semi-normalized, then the $Q$-component preserving dual fusion
frames of $V$ are precisely the Bessel fusion sequences $W=(W_i,w_i)_{i \in I}$ with $W_{i}= L M_{i} (\mathcal{K}_{V}^{2})$ for all $i\in I$, where 
$L \in \mathcal{B}_L(C_V)$. 

Moreover, if $\mathcal{H}$ is finite-dimensional, any element of
$\mathcal{B}_L(C_V)$ satisfies $L\vert_{\mathcal{K}_{V}^{2}}=(D_{W}Q_{L,w})\vert_{\mathcal{K}_{V}^{2}}$ where $W$ is some
$Q$-component preserving dual fusion frame of $V$.
\end{theorem}

\begin{proof} If $W$ is a $Q$-component preserving dual fusion frame of $V$ then, by definition, $W$ is a Bessel fusion sequence for $\Hil$ and by Lemma \ref{L W,w dual fusion frame entonces Vi=ApiWj} we have  $W_i = L M_{i} (\mathcal{K}_{V}^{2})$ for all $i\in I$, where 
$L \in \mathcal{B}_L(C_V)$.

Conversely, assume that $W=(W_i, w_i)_{i\in I}$ is a Bessel fusion sequence for $\Hil$ with $W_{i}= L M_{i} (\mathcal{K}_{V}^{2})$ ($i\in I$) for some $L \in \mathcal{B}_L(C_V)$. Since we assume that $(w_{i})_{i \in I}$ is semi-normalized, Remark \ref{R V Bessel Q acotada} (b) and Lemma \ref{L Vi=ApiWj entonces W,w dual fusion frame} imply that the operator $Q = Q_{L,w}$ is a bounded $\mathcal{B}(\mathcal{K}_{V}^{2},\mathcal{K}_{W}^{2})$-component preserving operator and that $W$ is a dual fusion frame of $V$ with respect to $Q$.

The moreover-part follows from Lemma \ref{L Vi=ApiWj entonces W,w dual fusion frame} and the corresponding proof details.
\end{proof}

\begin{remark}
As a consequence of Lemmata \ref{L W,w dual fusion frame entonces Vi=ApiWj} and \ref{L Vi=ApiWj entonces W,w dual fusion frame}, if $\mathcal{H}$ is finite-dimensional $-$ in this case one typically only considers finite index sets $I$ $-$ we can associate to any $Q$-dual fusion frame $(W_i,w_i)_{i \in I}$ of $(V_i,v_i)_{i \in I}$ the $Q_{L,\widetilde{w}}$-component preserving dual fusion frame $(LM_{i}\mathcal{K}_{V}^{2},\widetilde{w}_{i})_{i \in I}$ with $L=D_{W}Q$ and $(\widetilde{w}_{i})_{i \in I}$ arbitrary weights. Furthermore, if $Q$ is block-diagonal, then $Q_{D_WQ,w}=\big( \bigoplus_{i\in I} \pi_{W_i} \big)Q$. In particular, if $Q$ is $\mathcal{B}(\mathcal{K}_{V}^{2},\mathcal{K}_{W}^{2})$-component preserving, $Q_{D_WQ,w}=Q.$
\end{remark}

The next proposition shows that we can construct component
preserving dual fusion frame systems from a given fusion frame via
local dual frames and a bounded left inverse of the fusion analysis operator.

\begin{proposition}\label{P L TW T fusion frame system dual}\cite{hemo14}
Let $V=(V_i,v_i)_{i \in I}$ be a fusion frame, $L\in
\mathcal{B}_L(C_V)$ and $w$ be a semi-normalized weight such that
$m \leq w_i \leq M$ for all $i \in I$. For each $i\in I$, let
$\varphi^{(i)}:=(\varphi_{ij})_{j\in J_i}$ and
$\widetilde{\varphi}^{(i)}:=(\widetilde{\varphi}_{ij})_{j\in J_i}$
be a dual frame pair for $V_i$, $B_i$ an upper frame bound for  $\varphi^{(i)}$ such that $\sup_{i \in I}B_i =:B <\infty$, $\widetilde{A}_i$ and
$\widetilde{B}_i$ frame bounds of $\widetilde{\varphi}^{(i)}$ such that $\inf_{i\in I}\widetilde{A}_i =:\widetilde{A} >0$ and $\sup_{i \in I}\widetilde{B}_i =: \widetilde{B} <\infty$, and set $\psi^{(i)}= (\psi_{ij})_{j\in J_i} := \big(\frac{1}{w_{i}} L (\delta_{il}\widetilde{\varphi}_{ij})_{l \in
I}\big)_{j\in J_i}$ and  $W_{i}=\overline{\text{span}} (\psi^{(i)})$. Then
\begin{enumerate}
 \item[(i)] For every $i\in I$, $\psi^{(i)}$ is a frame for $W_{i}$ with frame
bounds $\|L^{\dagger}\|^{-2}\frac{\widetilde{A}_i}{w_{i}^{2}}$ and
$\|L\|^{2}\frac{\widetilde{B}_i}{w_{i}^{2}}$.
 \item[(ii)] $(W_i,w_i,\psi^{(i)})_{i \in I}$ is a component preserving $Q_{L,w}$-dual fusion frame system of $(V_i,v_i,\varphi^{(i)})_{i \in I}.$
\end{enumerate}
\end{proposition}

\begin{proof}
(i) Clearly, for every $i\in I$, $\big((\delta_{il}\widetilde{\varphi}_{ij})_{l \in
I}\big)_{j\in J_i}$ is a frame for $M_{i}\mathcal{K}_{V}^2$ with frame bounds $\widetilde{A}_i$ and
$\widetilde{B}_i$. Thus $\big(\frac{1}{w_{i}} (\delta_{il}\widetilde{\varphi}_{ij})_{l \in
I}\big)_{j\in J_i}$ is a frame for $M_{i}\mathcal{K}_{V}^2$ with frame bounds $w_i^{-2}\widetilde{A}_i$ and
$w_i^{-2} \widetilde{B}_i$, for every $i\in I$. Since $L$ is a bounded operator with $\mathcal{R}(L) = \mathcal{H}$, the claim follows from \cite[Proposition 5.3.1]{ole1n}.

(ii) We first show that the associated global family $\big(L (\delta_{il}\widetilde{\varphi}_{ij})_{l \in
I}\big)_{i\in I, j\in J_i}$ is a Bessel sequence for $\Hil$. For $g \in \mathcal{H}$ we have 
\begin{flalign}
\sum_{i \in I} \sum_{j\in
J_{i}}\vert\langle g, L(\delta_{il}\widetilde{\varphi}_{ij})_{l \in
I} \rangle \vert^{2} & \leq  \sum_{i \in I} \sum_{j\in J_{i}} \vert \langle
L^{*}g,(\delta_{il}\widetilde{\varphi}_{ij})_{l \in
I}\rangle_{\mathcal{K}_{\Hil}^2} \vert^{2} \notag \\ 
&= \sum_{i \in I} \sum_{j\in J_{i}} \big\vert \big\langle [ L^{*}g ]_i, \widetilde{\varphi}_{ij} \big\rangle \big\vert^{2} \notag \\
&\leq \sum_{i \in I} \widetilde{B}_i\| [L^{*}g ]_i \|^{2} \notag \\ 
&\leq \widetilde{B} \|L^{*}g\|^{2} \leq \widetilde{B} \Vert L \Vert ^{2} \Vert g \Vert^2. \notag
\end{flalign}
As in the proof of Theorem~\ref{fusframesysTHM}, it follows that $(W_i, w_i)_{i\in I}$ is a Bessel fusion sequence with Bessel fusion bound $\widetilde{B} \|L\|^{2} M^2 \widetilde{A}^{-1} \Vert L^{\dagger} \Vert^2$. Thus, the family $(W_i,w_i, \psi^{(i)})_{i \in I}$ is a Bessel fusion system in $\mathcal{H}$. Now, recall from Remark \ref{R V Bessel Q acotada} (b), that $Q_{L,w}$ is a well defined bounded operator with $\Vert Q_{L,w} \Vert \leq \Vert L \Vert m^{-1}$. Now, by identifying $L$ with its canonical matrix representation $[\dots \, L_{i-1} \, L_i \, L_{i-1} \, \dots]$, observe that for $(h_i)_{i\in I} \in \mathcal{K}_{\mathcal{H}}^{2}$ we have 
\begin{flalign}
Q_{L,w}(h_i)_{i\in I} &= \Big(\frac{1}{w_{i}}L M_{i}(h_{l})_{l\in
I} \Big)_{i\in I} \notag \\
&= \Big( \frac{1}{w_{i}}LM_{i}\Big(\sum_{j\in
J_l} \langle h_l,\varphi_{lj} \rangle \widetilde{\varphi_{lj}}
\Big)_{l \in I} \Big)_{i\in
I} \notag \\ 
&= \Big( \sum_{j\in
J_i} \langle h_i,\varphi_{ij} \rangle \frac{1}{w_{i}} L_i \widetilde{\varphi_{ij}} \Big)_{i\in I} \notag \\
&= \Big(\sum_{j\in J_i} \langle h_i,\varphi_{ij} \rangle 
\frac{1}{w_{i}}L(\delta_{il}\widetilde{\varphi}_{ij})_{l \in I}\Big)_{i\in I} \notag \\
&= \Big(\sum_{j\in J_i} \langle h_i,\varphi_{ij} \rangle 
\psi_{ij} \Big)_{i\in I} = \bigoplus_{i\in I} D_{\psi^{(i)}}C_{\varphi^{(i)}}(h_i)_{i\in I}.\notag
\end{flalign}
Hence (ii) follows from (i) and Lemma~\ref{L Vi=ApiWj entonces W,w
dual fusion frame}.
\end{proof}

The following proposition presents a way to construct component
preserving dual fusion frame systems from Bessel fusion systems by using a bounded left-inverse of the global analysis operator.

\begin{proposition}\label{P L TF T fusion frame system dual}\cite{hemo14} 
Let $(V_i, v_i, \varphi^{(i)})_{i \in I}$ be a Bessel fusion system with associated global frame $v\varphi$, assume that $w=(w_{i})_{i \in I}$ is a semi-normalized collection of weights. Let $L$ be a bounded left-inverse of $C_{v\varphi}$ and let $((e_{ij})_{j\in J_{i}})_{i\in I}$ be the canonical orthonormal basis for $(\sum_{i \in I}\oplus\ell^2(J_{i}))_{\ell^{2}}$. For each $i \in I$, set $\psi^{(i)}:= \big(\frac{1}{w_{i}}Le_{ij}\big)_{j\in J_{i}}$ and $W_{i}:= \overline{\text{span}}(\psi^{(i)})$. Then
\begin{enumerate}
 \item[(i)] $\psi^{(i)}$ is a frame for $W_{i}$ with frame
bounds $\frac{\|L^{\dagger}\|^{-2}}{w_{i}^{2}}$ and
$\frac{\|L\|^{2}}{w_{i}^{2}}$.
 \item[(ii)] $(W_i,w_i,\psi^{(i)})_{i \in I}$ is a dual fusion frame system of
$(V_i,v_i,\varphi^{(i)})_{i \in I}.$
\end{enumerate}
\end{proposition}

\begin{proof}
(i) As in the proof of Proposition \ref{P L TW T fusion frame system dual} (i), the claim follows from \cite[Proposition 5.3.1]{ole1n}.

(ii) Since $w$ is semi-normalized, the local frames $\psi^{(i)}$ have common frame bounds. The associated global family $w\psi$ is a Bessel sequence for $\Hil$, since
$$\sum_{i \in I}\sum_{j\in J_{i}}|\langle
g,Le_{ij}\rangle|^{2}=\sum_{i \in I}\sum_{j\in J_{i}}|\langle
L^{*}g,e_{ij}\rangle|^{2}=\|L^{*}g\|^{2}\leq\|L\|^{2}\|g\|^{2}$$
for all $g\in \Hil$. Thus, as in the proof of Theorem \ref{fusframesysTHM}, $(W_i,w_i)_{i\in I}$ is a Bessel fusion sequence for $\Hil$ and thus  $(W_i,w_i,\psi^{(i)})_{i \in I}$ a
Bessel fusion system for $\mathcal{H}$. 
Now, note that for each $i\in I$, $(\delta^{(i)}_j)_{j\in J_i}$, where $\delta^{(i)}_j = (\delta_{jk})_{k\in J_i}$, is the canonical orthonormal basis for $\ell^2(J_i)$. By Lemma \ref{ONBHDS}, each $e_{ij}$ is given by $e_{ij} = (\delta_{ik} \delta^{(k)}_j)_{k\in I}$. Therefore, for any $g\in \Hil$ we have that 
$$D_{w\psi} C_{v\varphi} g = L \sum_{i\in I} \sum_{j\in J_i} \langle g, v_i \varphi_{ij} \rangle e_{ij} = L (\langle g, v_k \varphi_{kl} \rangle )_{i\in I, j\in J_i} = LC_{v\varphi} g = g.$$
In other words, $w\psi$ is a dual frame of $v\varphi$. Finally, (ii) follows from an application of Theorem~\ref{T dual fusion frame systems}.
\end{proof}

\subsection{The canonical dual fusion frame}\label{Ej dual canonico}

Let $V=(V_i,v_i)_{i \in I}$ be a fusion frame for $\mathcal{H}$. Let
$L=S_{V}^{-1}D_{V} \in \mathcal{B}_L(C_{V})$ and $w=(w_{i})_{i \in I}$ be a family of weights such that $(S_{V}^{-1}V_i,w_i)_{i \in I}$ is a Bessel
fusion sequence for $\mathcal{H}$. Assume that $Q_{L,w} :
\mathcal{K}_{\Hil}^{2} \rightarrow \mathcal{K}_{\Hil}^2$
given by $Q_{L,w}(f_i)_{i \in
I}=(\frac{v_{i}}{w_{i}}S_{V}^{-1} \pi_{V_i}f_{i})_{i \in I}$ is a well defined
bounded operator. By Lemma \ref{blockdiagonalbounded}, $Q_{L,w} = \bigoplus_{i\in I} \frac{v_i}{w_i}S_V^{-1}\pi_{V_i}$ is a well-defined and bounded operator if and only if $\sup_{i\in I} \frac{v_i}{w_i} < \infty$. In particular, by Lemma \ref{weights}, $v_i\leq w_i$ for all $i\in I$ or $w$ being semi-normalized is sufficient for $Q_{L,w}$ to be bounded. In that case, by Lemma~\ref{L Vi=ApiWj entonces W,w dual fusion frame}, $(S_{V}^{-1}V_i,w_i)_{i \in I}$ is a component preserving $Q_{L,w}$-dual fusion frame of $(V_i,v_i)_{i \in I}$. We will refer to this dual as the {\em
canonical dual with weights $w$}. In particular, the choice $v=w$ yields that $(S_{V}^{-1}V_i,v_i)_{i
\in I}$ is a $Q_{S_{V}^{-1}D_{V},v}$-component preserving dual
of $(V_i,v_i)_{i \in I}$ (see also \cite[Example 3.7]{heimoanza14}). Furthermore, if
in Definition~\ref{D dual fusion frame system} $(W_i,w_i)_{i \in I}$
is a canonical dual fusion frame of $(V_i,v_i)_{i \in I}$  we say
that $(W_i,w_i,\psi^{(i)})_{i \in I}$ is a \emph{canonical dual
fusion frame system} of $(V_i,v_i,\varphi^{(i)})_{i \in I}.$ 

Now, note that by substituting $T = S_V^{-1}$ and $V = V_i$ in Lemma \ref{fusionframeoplem} and using that $S_V^{-1} V_i$ is a closed subspace of $\Hil$ by Lemma \ref{techlemma2}, we obtain that 
$$\pi_{V_i} S_V^{-1} \pi_{S_V^{-1}V_i} = \pi_{V_i} S_V^{-1} \qquad (\forall i\in I).$$
This implies that 
$$Q_{S_{V}^{-1}D_{V},w}^{*}{C}_{S_{V}^{-1}V,w} f = C_{V}S_{V}^{-1}  f \qquad (\forall f\in \Hil).$$
We call $C_V S_V^{-1} f = (v_i \pi_{V_i} S_V^{-1} f)_{i\in I}$ the {\em canonical dual fusion frame coefficients} of $f \in \mathcal{H}$ associated to $V$ and note that they also appear in fusion frame reconstruction (\ref{fusionframereconstruction}). 

The
following lemma implies that the canonical dual fusion frame coefficients are those coefficients which have minimal norm among all other coefficients.

\begin{lemma}\cite{heimoanza14} 
Let $V=(V_i,v_i)_{i \in I}$ be a fusion frame for $\mathcal{H}$ and
$f \in \mathcal{H}$. For all $(f_{i})_{i \in I} \in
\mathcal{K}_{\Hil}^{2}$ satisfying $D_{V}(f_{i})_{i \in I}=f$, we have
$$\Vert (f_{i})_{i \in
I} \Vert^{2} = \Vert C_{V}S_{V}^{-1}f \Vert^{2} + \Vert (f_i)_{i \in
I} - C_{V}S_{V}^{-1}f \Vert^{2}.$$
\end{lemma}

\begin{proof}
Suppose that $(f_{i})_{i \in I} \in \mathcal{K}_{\Hil}^{2}$ satisfies
$D_{V}(f_{i})_{i \in I}=f$. Then we have $(f_{i})_{i \in I}-C_{V}S_{V}^{-1}f \in
\mathcal{N}(D_{V})=\mathcal{R}(C_{V})^{\perp}$. Since $C_{V}S_{V}^{-1}f \in \mathcal{R}(C_{V})$, the
result follows.
\end{proof}

We conclude this section with a result on duals of a fusion Riesz basis.

\begin{proposition}\label{DualRFB}\cite{HeinekenMorillas18} 
Let $V=(V_i,v_i)_{i \in I}$ be a fusion Riesz basis for
$\mathcal{H}$. Then the following assertions hold:
\begin{enumerate}
\item[(i)] The component preserving duals of $V$ are of the form $(S_{V}^{-1}V_i,w_i)_{i \in I}$.
\item[(ii)] If $W=(W_i,w_i)_{i \in I}$ is a Riesz fusion basis for $\mathcal{H}$, then the restriction $(C_{W}D_{V})\vert_{\mathcal{K}_{V}^{2}}: \mathcal{K}_{V}^{2} \longrightarrow \mathcal{K}_{W}^{2}$ is bounded and invertible.
\item[(iii)] \label{MC} If $W=(W_i,w_i)_{i \in I}$ is a block-diagonal dual fusion frame of $V$, then $S_{V}^{-1}V_{i} \subseteq W_{i}$ for every $i \in I$.
\item[(iv)] \label{CFRB} If $W=(W_i,w_i)_{i \in I}$ is a fusion Riesz basis for $\mathcal{H}$ and a block-diagonal dual fusion frame of $V$, then $W_i=S_{V}^{-1}V_{i}$ for every $i \in I$.
\end{enumerate}
\end{proposition}

\begin{proof}
(i) 
By Lemma~\ref{leftright} and Remark~\ref{newsynthesis}, the set of all bounded left-inverses of $C_V$ is given by
$$\mathcal{B}_{L}(C_{V})=\set{S_{V}^{-1}D_{V}+L\pi_{(\sum_{i\in
I} \oplus V_i^{\perp})_{\ell^2}}: L \in
\mathcal{B}(\mathcal{K}_{\mathcal{H}}^{2},\mathcal{H})} .$$
Taking into account that $L\pi_{(\sum_{i\in I} \oplus
V_i^{\perp})_{\ell^2}}M_{i}\mathcal{K}_V^2=0$, we conclude via Lemma \ref{L W,w dual fusion frame entonces Vi=ApiWj} that any component preserving dual fusion frame of $V$ has to be of the form $(S_{V}^{-1}V_i,w_i)_{i \in I}$. 

(ii) This is an immediate consequence of Theorem \ref{fusionrieszbasischar}.

(iii) Let $W=(W_i,w_i)_{i \in I}$ be a block-diagonal dual fusion frame of the fusion Riesz basis $V$ with respect to $\bigoplus_{i\in I} Q_i$. Then $D_{W} \big( \bigoplus_{i\in I} Q_i \big)  C_{V} = \mathcal{I}_{\Hil}$ and $C_V (S_{V}^{-1}D_{V})\vert_{\mathcal{K}_{V}^{2}} = \mathcal{I}_{\mathcal{K}_{V}^{2}}$ (the latter is true since $C_{V}:\mathcal{H}\longrightarrow \mathcal{K}_{V}^{2}$ is bijective by Theorem \ref{fusionrieszbasischar} and
$S_{V}^{-1}D_{V}C_{V}=\mathcal{I}_{\Hil}$). For arbitrary $i\in I$ and $f_i \in V_i$, we have $(\delta_{ij}v_i^{-1} f_i)_{j\in I} \in \mathcal{K}_V^2$ and $D_V (\delta_{ij}v_i^{-1} f_i)_{j\in I} = f_i$. This implies 
\begin{align*}
S_{V}^{-1}f_{i}&=D_{W} \Big( \bigoplus_{i\in I} Q_i \Big) C_{V}S_{V}^{-1}D_V (\delta_{ij}v_i^{-1} f_i)_{j\in I} \\
&= D_{W} \Big( \bigoplus_{i\in I} Q_i \Big) (\delta_{ij} v_i^{-1} f_i)_{j\in I} = D_{W} (\delta_{ij} Q_i v_i^{-1} f_i)_{j\in I} \in W_{i} .
\end{align*}
\noindent Since $i$ and $f_i$ were chosen arbitrarily, the claim follows.

(iv) By (iii), $S_{V}^{-1}V_{i}\subseteq W_i$ for each
$i \in I$. Suppose that there exists some $i_0\in I$ so that
$S_{V}^{-1}V_{i_0}\subsetneqq W_{ i_0}$. By Lemma \ref{techlemma}, $S_{V}^{-1}V_{i_0}$ is a closed subspace of $W_{i_0}$ and thus $U_{i_{0}}:=(S_{V}^{-1}V_{i_0})^{\perp}\cap W_{ i_0}$ is a proper subspace of $W_{i_0}$, i.e. $W_{i_{0}} =S_{V}^{-1}V_{i_0}\oplus^{\perp} U_{i_0}$. Now, take $0 \neq u_{i_{0}} \in U_{ i_0}$. By Corollary \ref{canonicaldualfusion}, $(S_V^{-1}V_i, v_i)_{i\in I}$ is a fusion Riesz basis, hence, by Theorem \ref{fusionrieszbasischar}, it is a Riesz decomposition, so that $u_{i_{0}}=\sum_{i \in I}g_{i}$ for unique $g_{i} \in S_{V}^{-1}V_{i}$ $(i \in I)$. In particular $g_i \in W_i$ for each $i\in I$. On the other hand, $u_{i_0} = \dots + 0 + 0 + u_{i_0} + 0+ 0+ \dots $ can be viewed as the unique linear combination with respect to the Riesz decomposition $(W_i)_{i\in I}$. This implies that $u_{i_0} = g_{i_{0}}\in S_{V}^{-1}V_{i_0} \cap U_{i_0}=\{0\}$. This is absurd. Thus the conclusion follows.
\end{proof}

In Proposition~\ref{DualRFB} (i), if the operator $Q_{S_{V}^{-1}D_{V},w}$ with respect to $w$ is well-defined and bounded, the component preserving duals of $V$ coincide with the canonical ones with weights $w$, i.e. the operator $Q$ for these duals is $Q_{S_{V}^{-1}D_{V},w}$.

\subsection{Summary of properties of dual fusion frames }

Our aim here is to highlight the most important properties of dual
fusion frames presented before.

\begin{enumerate}
\item[(a)] Given a fusion frame $V = (V_i, v_i)_{i\in I}$, any bounded left-inverse $L$ of the analysis operator $C_{V}$ is a reconstruction operator which $-$ for every $f\in \Hil$ $-$ allows perfect reconstruction from the fusion frame measurements $C_V f = (v_i \pi_{V_i} f)_{i\in I}$, i.e.
\begin{equation}\label{dualfusion1}
    L C_V = \mathcal{I}_{\Hil}.
\end{equation}
By Lemma \ref{L Vi=ApiWj entonces W,w dual fusion frame}, the
operator $L$ can be related to a component preserving dual fusion
frame $W = (W_i, w_i)_{i\in I}$ and there is a link between $L$ and
the synthesis operator of $W$.

\item[(b)] We have that $\psi = (\psi_i)_{i\in I}$ is a frame if and only if $V
= (\text{span}\lbrace \psi_i \rbrace, \Vert \psi_i \Vert )_{i\in I}$
is a fusion frame. By Theorem \ref{T dual fusion frame systems} any
dual frame of $\psi$ corresponds to a dual fusion frame of $V$.

\item[(c)] The duality relation for frames is symmetric (i.e.
$\psi$ is a dual frame of $\varphi$ if and only if $\varphi$ is a
dual frame of $\psi$). By Lemma~\ref{L equivalencias}, the fusion
frame duality relation is also symmetric, regarding the adjoint operator $Q^{*}$.

\item[(d)] For every frame there always exists the
canonical dual frame. In Section~\ref{Ej dual canonico}, we saw that
the canonical dual fusion frame of a fusion frame always exists,
too.

\item[(e)] From the results of Section~\ref{S dual families}, we see that one advantage of $Q$-dual fusion frames is that
they can easily be obtained. They can be obtained from the left
inverses of the analysis operator of the fusion frames, or from dual
frames. This fact leads to a plethora of $Q$-dual fusion frames of
different types and with different properties and reconstruction
formulas.

\end{enumerate}

\subsection{Examples of dual fusion frames}\label{Examples of dual fusion frames}

In this subsection we provide some examples of dual fusion frames. 

We start with giving an example of a block-diagonal dual fusion frame of a fusion Riesz basis, which is different from the canonical dual and not a fusion Riesz basis itself.    

\begin{example}{Example \cite{hemo14}}\label{Ej Rieszdualdistcanon}
Let $\mathcal{H}=\mathbb{C}^{4}$ and let $V_{1}=\{(x_{1}, x_{2}, 0 , 0)^T : x_{1}, x_{2} \in \mathbb{C}\}$ and
$V_{2}=\{(0, x_{2}, x_{3} , -x_{2})^T : x_{2}, x_{3} \in
\mathbb{C}\}$. Then $V=(V_i, 1)_{i=1}^{2}$ is a ($2$-equi-dimensional) fusion Riesz basis for $\mathbb{C}^{4}$ and so
its unique component preserving duals are the canonical ones. 

Although $V=(V_{i},1)_{i=1}^{2}$ is a fusion Riesz basis, it is
possible to construct a block diagonal dual fusion frame which is not a fusion Riesz basis. To this end, let
\begin{align*}
\varphi^{(1)}&=\Big( (1,0,0,0)^T, (0,1,0,0)^T, (1,0,0,0)^T \Big),\\
\varphi^{(2)}&=\Big( (0,1,0,-1)^T, (0,0,1,0)^T, (0,0,1,0)^T \Big),\\
\psi^{(1)}&=\Big( (1/2, 1/2, -1/2, 0)^T, (0,1,0,1)^T, (1/2, -1/2, 1/2, 0)^T \Big),\\
\psi^{(2)}&=\Big( (0,0,0,-1)^T, (1/2, -1/2, 1/2, 0)^T, (-1/2, 1/2, 1/2, 0)^T \Big).
\end{align*}
Then we have $V_i=\text{span}\lbrace \varphi^{(i)} \rbrace$ for $i=1, 2$ and $(V_{i},1,\varphi^{(i)})_{i=1}^{2}$ is a fusion
frame system for $\mathbb{C}^{4}$. Moreover, if we set $W_i=\text{span}\lbrace \psi^{(i)} \rbrace$ for $i=1,2$, then $(W_{i},1,\psi^{(i)})_{i=1}^{2}$ is a fusion frame system as well. A direct computation shows, that their respective global frames $\varphi$ and $\psi$ form a dual frame pair and that $\psi$ is not the canonical dual frame of $\varphi$. By Theorem \ref{T dual fusion frame systems}, $(W_{i},1,\psi^{(i)})_{i=1}^{2}$ is a dual
fusion frame system of $(V_{i},v_{i},\varphi^{(i)})_{i=1}^{2}$. Note that $\bigoplus_{i=1}^2 D_{\psi^{(i)}}C_{\varphi^{(i)}}$ is block-diagonal but not $\mathcal{B}(\mathcal{K}_{V}^{2},\mathcal{K}_{W}^{2})$-component preserving: Indeed, a direct computation shows that $\mathcal{R}(D_{\psi^{(1)}}C_{\varphi^{(1)}}) = \{(x_{1}, x_{2}, 0 , x_2)^T : x_{1}, x_{2} \in \mathbb{C}\}$ is $2$-dimensional, while $W_1$ is $3$-dimensional, so $\mathcal{R}(D_{\psi^{(1)}}C_{\varphi^{(1)}})\subsetneqq W_1$.  

Since $\dim(W_i)=3>2 = \dim(V_i)$ for $i=1,2$, we see (compare with Proposition \ref{finite1} (ii)) that $W=(W_{i},1)_{i=1}^{2}$ is a fusion frame but not a fusion Riesz basis. In particular, this means that $W$ is not a canonical dual of $V$, because otherwise Theorem \ref{3itemeqri} would be hurt.
\end{example}

The next example proves the existence of self-dual fusion frames which are not Parseval. 

\begin{example}{Example \cite{hemo14}}
Let $\mathcal{H}=\mathbb{R}^{3}$ and consider the subspaces 
\begin{flalign}
V_{1} &={\rm span}\big((0,1,0)^T, (0,0,1)^T\big), \notag \\
V_{2} &={\rm span}\big( (1,0,0)^T, (0,0,1)^T\big). \notag
\end{flalign}
Then $V=(V_i, 1)_{i=1}^{2}$ is a ($2$-equi-dimensional) fusion frame for $\mathbb{R}^{3}$
with
$$S_{V}=\begin{bmatrix}
     1 & 0 & 0  \\
     0 & 1 & 0 \\
     0 & 0 & 2  
\end{bmatrix} \qquad \text{and} \qquad S_{V}^{-1}=\frac{1}{2}\begin{bmatrix}
     2 & 0 & 0  \\
     0 & 2 & 0 \\
     0 & 0 & 1  
\end{bmatrix} .$$
In particular, $S_V^{-1} V_1 = V_1$ and $S_V^{-1} V_2 = V_2$, so the canonical dual of $V$ with same weights equals $V$. In other words, $V$ is self-dual even though $V$ is not a Parseval fusion frame.
\end{example}

Finally, the following example demonstrates that the canonical dual of the canonical dual of a fusion frame $V$ is not necessarily $V$. It also shows that the fusion frame operator of the canonical dual can differ from the inverse of the fusion frame operator.

\begin{example}{Example}\label{canonicaldualcanonicaldual}
Let $\Hil = \mathbb{R}^{4}$ and consider the subspaces 
\begin{flalign}
V_{1} &={\rm span}\big((1,0,0,0)^T, (0,1,0,0)^T, (0,0,-1,1)^T\big), \notag \\
V_{2} &={\rm span}\big( (0,1,0,0)^T, (0,0,1,0)^T\big), \notag \\
V_{3} &= {\rm span}\big( ((0,0,1,0)^T, (0,0,0,1)^T \big). \notag
\end{flalign}
Then $V=(V_{i}, 1)_{i=1}^{3}$ is a fusion frame for $\mathbb{R}^{4}$. We have
$$S_{V}=\frac{1}{2} \begin{bmatrix}
     2 & 0 & 0 & 0 \\
                      0 & 4 & 0 & 0 \\
                      0 & 0 & 5 & -1 \\
                      0 & 0 & -1 & 3
\end{bmatrix} \qquad \text{and} \qquad S_{V}^{-1}=\frac{1}{14} \begin{bmatrix}
    14 & 0 & 0 & 0 \\
        0 & 7 & 0 & 0 \\
       0 & 0 & 6 & 2 \\
       0 & 0 & 2 & 10
\end{bmatrix} .$$
For the canonical dual $\widetilde{V}=(\widetilde{V}_{i}, 1)_{i=1}^{3}$ with same weights we obtain
\begin{flalign}
\widetilde{V}_{1}=S_{V}^{-1}V_{1} &={\rm span}\big((1,0,0,0)^T, (0,1,0,0)^T, (0,0,1,-2)^T\big), \notag \\
\widetilde{V}_2 = S_{V}^{-1}V_2 &={\rm span}\big( (0,1,0,0)^T, (0,0,3,1)^T\big), \notag \\
\widetilde{V}_{3}=S_{V}^{-1}V_{3} &= {\rm span}\big( ((0,0,3,1)^T, (0,0,1,5)^T \big). \notag
\end{flalign}
and further
$$S_{S_{V}^{-1}V}=\frac{1}{10} \begin{bmatrix}
  10 & 0 & 0 & 0 \\
                      0 & 20 & 0 & 0 \\
                      0 & 0 & 21 & -1 \\
                      0 & 0 & -1 & 19   
\end{bmatrix}, \qquad S_{S_V^{-1} V}^{-1} = \frac{1}{398} \begin{bmatrix}
    398 & 0 & 0 & 0 \\
                      0 & 199 & 0 & 0 \\
                      0 & 0 & 190 & 10 \\
                      0 & 0 & 10 & 210
\end{bmatrix} .$$
For the canonical dual $\widetilde{\widetilde{V}}=(\widetilde{\widetilde{V}}_{i},v_{i})_{i=1}^{3}$ of $\widetilde{V}$ with same weights we obtain
\begin{flalign}
\widetilde{\widetilde{V}}_{1}=S_{S_{V}^{-1}V}^{-1}\widetilde{V}_{1} &={\rm span}\big((1,0,0,0)^T, (0,1,0,0)^T, (0,0,-17,41)^T\big) \neq V_1, \notag \\
\widetilde{\widetilde{V}}_{2}=S_{S_{V}^{-1}V}^{-1}\widetilde{V}_{2} &={\rm span}\big( (0,1,0,0)^T, (0,0,29,12)^T\big) \neq V_2, \notag \\
\widetilde{\widetilde{V}}_{3}=S_{S_{V}^{-1}V}^{-1}\widetilde{V}_{3} &= V_3. \notag
\end{flalign}
Hence the canonical dual of the canonical dual of $V$ differs from $V$. We also observe that $S_{S^{-1}V}\neq S_{V}^{-1}$, i.e. the fusion frame operator of the canonical dual differs from the inverse of the fusion frame operator.
\end{example}

\subsection{Particular choices of $Q$}\label{S other approaches}

We conclude this section on dual fusion frames with highlighting two particular choices of $Q$ in Definition \ref{D dual fusion frame}.

\subsubsection{Case $Q=\bigoplus_{i \in I} \mathcal{I}_{\mathcal{H}}$}

As was pointed out in the beginning of Section \ref{Duality in Fusion Frame Theory}, when the coefficient space is $\mathcal{K}_{\mathcal{H}}^{2}$ and hence there is no domain difficulty when composing $D_W$ with $C_V$, the first natural idea for the definition of a dual fusion frame would be to generalize the
duality concept directly from (\ref{DCisIframes}). That is, we consider the block-diagonal operator $Q= \mathcal{I}_{\mathcal{K}_{\mathcal{H}}^{2}} = \bigoplus_{i\in I} \mathcal{I}_{\Hil}$ in Definition~\ref{D
dual fusion frame}. Hence (\ref{E D dual fusion frame}) becomes
\begin{equation}\label{dualfusion2}
    D_W C_V = \mathcal{I}_{\Hil}.
\end{equation}
However, the relation (\ref{dualfusion2}) is quite restrictive and does not lead to a good duality concept. 

In particular, the following example demonstrates that the fusion frames associated to a dual frame pair do not necessarily satisfy (\ref{dualfusion2}). In other words, $\mathcal{I}_{\mathcal{K}_{\Hil}^2}$-dual fusion frames do not include dual frames as a special case.

\begin{example}{Example}\label{newsynthesis4}
Assume that $\psi = (\psi_i)_{i\in I}$ and $\varphi =
(\varphi_i)_{i\in I}$ is a pair of dual frames for $\Hil$. Consider their respective associated fusion frames $V = (\text{span}\lbrace
\varphi_i \rbrace, \Vert \varphi_i \Vert )_{i\in I}$ and $W = (
\text{span}\lbrace \psi_i \rbrace, \Vert \psi_i \Vert )_{i\in
I}$. 
For arbitrary $f\in \Hil$,
we see that
\begin{flalign}
D_W C_V f &= \sum_{i\in I} \Vert \varphi_i \Vert \Vert \psi_i \Vert \pi_{\text{span}\lbrace \psi_i \rbrace} \Big\langle f, \frac{\varphi_i}{\Vert \varphi_i \Vert} \Big\rangle \frac{\varphi_i}{\Vert \varphi_i \Vert} \notag \\
&= \sum_{i\in I} \Vert \psi_i \Vert \langle f, \varphi_i \rangle   \Big\langle \frac{\varphi_i}{\Vert \varphi_i \Vert}, \frac{\psi_i}{\Vert \psi_i \Vert} \Big\rangle \frac{\psi_i}{\Vert \psi_i \Vert} \notag \\
&= \sum_{i\in I} \langle f, \varphi_i \rangle  \Big\langle
\frac{\varphi_i}{\Vert \varphi_i \Vert}, \frac{\psi_i}{\Vert \psi_i
\Vert} \Big\rangle \psi_i . \notag
\end{flalign}
Now, if we had
\begin{equation}\label{exx}
    \Big\langle \frac{\varphi_i}{\Vert \psi_i \Vert}, \frac{\psi_i}{\Vert \psi_i \Vert} \Big\rangle = 1 \qquad (\forall i\in I),
\end{equation}
then $D_W C_V = \mathcal{I}_{\Hil}$. For instance, if $\psi =
\varphi$ was a Parseval frame, then (\ref{exxx}) would be satisfied.
However, it is easy to find a pair of dual frames, such that $D_W
C_V \neq \mathcal{I}_{\Hil}$: For simplicity, we consider the case,
where $\varphi$ is a Riesz basis. Then $\varphi$ has a unique dual given
by $\psi = (S_{\varphi}^{-1} \varphi_i)_{i\in I}$ and, by bijectivity
of the corresponding synthesis and analysis operators, condition
(\ref{exx}) becomes necessary and sufficient for
(\ref{dualfusion2}) to be true. By the biorthogonality of a Riesz
basis and its dual, condition (\ref{dualfusion2}) is true if and
only if
\begin{equation}\label{exxx}
   \Vert \varphi_i \Vert  \Vert \psi_i \Vert = 1 \qquad (\forall i\in I)
\end{equation}
in this case. Now, choose $\Hil = \mathbb{R}^2$ and consider the
Riesz basis given by $\varphi_1 = (1,2)^T$ and  $\varphi_2 = (1,-1)^T$. Then
we obtain
$$S_{\varphi} = \begin{bmatrix}
    2 & 1 \\
    1 & 5 \\
    \end{bmatrix} \quad \text{and} \quad S_{\varphi}^{-1} = \frac{1}{9} \cdot \begin{bmatrix}
    5 & -1 \\
    -1 & 2 \\
    \end{bmatrix}.$$
In particular, we see that $\Vert \varphi_1 \Vert = \sqrt{5}$, while
$\Vert S_{\varphi}^{-1} \varphi_1 \Vert = \Vert (1/3, 1/3)^T \Vert =
\sqrt{2}/3$.
\end{example}
In contrast to the above example, finding pairs of fusion frames $V$ and
$W$ such that $D_W C_V = \mathcal{I}_{\Hil}$ still seems very
interesting for applications. Hence, the remainder of this
subsection is devoted to this task.

\

We start with the following characterization in the special case,
where $V$ is a fusion Riesz basis.

\begin{proposition}\label{newdualprop}
Let $V$ be a fusion Riesz basis and $W$ be a fusion frame. Then the
following are equivalent.
\begin{enumerate}
    \item[(i)] $D_W C_V = \mathcal{I}_{\Hil}$.
    \item[(ii)] $ w_i \pi_{W_i} \pi_{V_i} = v_i S_V^{-1} \pi_{V_i}$ for all $i\in I$.
\end{enumerate}
\end{proposition}

\begin{proof}
If (ii) is satisfied, then $$D_W C_V = \sum_{i\in I} w_i v_i
\pi_{W_i} \pi_{V_i} = \sum_{i\in I} v_i^2 S_V^{-1}  \pi_{V_i} =
S_V^{-1} S_V = \mathcal{I}_{\Hil}.$$ Conversely, assume that (i) is
satisfied. For each subspace $V_i$, choose an orthonormal basis $e_i
= (e_{ij})_{j\in J_i}$. Then, since $V$ is a fusion Riesz basis, the
associated global frame $we = (v_i e_{ij})_{i\in I, j\in J_i}$ is a
Riesz basis by Theorem \ref{minimalchaR}. Moreover, since all the
local frames are orthonormal bases, we have $S_V = S_{we}$ by Lemma
\ref{globalisfusion}. Now, we observe that for all $f\in
\mathcal{H}$ we have
\begin{flalign}
    f &= \sum_{i\in I} w_i v_i \pi_{W_i} \pi_{V_i} f \notag \\
    &=  \sum_{i\in I} w_i v_i \pi_{W_i} \sum_{j\in J_i} \langle f, e_{ij} \rangle e_{ij} =  \sum_{i\in I} \sum_{j\in J_i} \langle f, v_i e_{ij} \rangle w_i \pi_{W_i} e_{ij} \notag.
\end{flalign}
This means that $(w_i \pi_{W_i} e_{ij})_{i\in I, j\in J_i}$ is a
dual frame of $we$. However, since $we$ is a Riesz basis, it has a
unique dual frame, namely its canonical dual frame. Thus
$$w_i \pi_{W_i} e_{ij} = S_{we}^{-1} v_i e_{ij} = S_V^{-1} v_i e_{ij} \quad (\forall i\in I, j\in J_i).$$
In particular,  for any $g \in V_i$ it holds
$$w_i \pi_{W_i} g = \sum_{j\in J_i} \langle g, e_{ij} \rangle w_i \pi_{W_i} e_{ij} = \sum_{j\in J_i} \langle g, e_{ij} \rangle v_i S_V^{-1} e_{ij} = v_i S_V^{-1} g.$$
Since this is true for every $i\in I$, (i) follows.
\end{proof}

For an arbitrary fusion frame $V$, the theorem below gives an operator theoretic characterization of all fusion frames $W$ such that $D_W C_V = \mathcal{I}_{\Hil}$. Its proof relies on the canonical matrix representation of bounded operators between Hilbert direct sums as discussed in Section \ref{Operators between Hilbert direct sums}.

\begin{theorem}\label{normaldualthm}
Let $V = (V_i, v_i)_{i\in I}$ be a fusion frame for $\Hil$. Then the
following are equivalent:
\begin{enumerate}
    \item[(i)] $W = (W_i, w_i)_{i\in I}$ is a Bessel fusion sequence for $\Hil$ satisfying $D_W C_V = \mathcal{I}_{\Hil}$.
    \item[(ii)] There exists a family $(L_i)_{i\in I}$ in $\mathcal{B}(\Hil)$, such that $[\dots L_{i-1} \, L_i \, L_{i+1} \dots] \in \mathcal{B}(\mathcal{K}_{\Hil}^2 , \Hil)$ and
\begin{equation}\label{entries}
    w_i \pi_{W_i} = v_i S_V^{-1} \pi_{V_i} + L_i - \Big( \sum_{k\in I} L_k v_k \pi_{V_k} \Big) v_i S_V^{-1} \pi_{V_i} \quad (\forall i\in I).
\end{equation} 
\end{enumerate}
\end{theorem}

\begin{proof}
(i) is satisfied if and only if $D_W$ is a bounded left-inverse of
$C_V$. By Lemma \ref{leftright} (a), this is equivalent to $D_W =
S_V^{-1} D_V + L(\mathcal{I}_{\mathcal{K}_{\Hil} ^2} - C_V S_V^{-1}
D_V )$, where $L \in \mathcal{B}(\mathcal{K}_{\Hil}^2 , \Hil)$. By
the uniqueness of matrix representations (Propositions
\ref{LemmaBounded2} and \ref{LemmaBounded}), there exists precisely
one family $(L_i)_{i\in I}$ in $\mathcal{B}(\Hil)$, such that
$\mathbb{M}(L) = [\dots L_{i-1} \, L_i \, L_{i+1} \dots]$.
Furthermore, the matrix representation of the bounded operator
$\mathcal{I}_{\mathcal{K}_{\Hil}^2} - C_V S_V^{-1} D_V$ is given by
$$\begin{bmatrix}
    \ddots & \vdots & \vdots & \, \\
    \dots & \mathcal{I}_{\Hil} - v_i^2 \pi_{V_i} S_V^{-1} \pi_{V_i} & \, \, \,   - v_i v_{i+1} \pi_{V_i} S_V^{-1} \pi_{V_{i+1}} & \dots \\
    \dots & - v_{i+1} v_i \pi_{V_{i+1}} S_V^{-1} \pi_{V_i} & \, \, \, \mathcal{I}_{\Hil} - v_{i+1}^2 \pi_{V_{i+1}} S_V^{-1} \pi_{V_{i+1}} & \dots \\
    \, & \vdots & \vdots & \ddots \\
    \end{bmatrix} .$$
This implies that the matrix representation of $S_V^{-1}D_V + L
\big( \mathcal{I}_{\mathcal{K}_{\Hil}^2} - C_V S_V^{-1} D_V \big)$
is given by the row matrix $[\dots A_{i-1} \, A_i \, A_{i+1}
\dots]$, where
$$A_i = v_i S_V^{-1} \pi_{V_i} + L_i - \Big( \sum_{k\in I} L_k v_k \pi_{V_k} \Big) v_i S_V^{-1} \pi_{V_i}.$$
Thus, (i) is equivalent to (\ref{entries}).
\end{proof}

\begin{corollary}\label{normaldualcor0}
Let $V = (V_i, v_i)_{i\in I}$ be a fusion frame for $\Hil$. If there
exists a family $(L_i)_{i\in I}$ in $\mathcal{B}(\Hil)$, such
that $(\Vert L_i \Vert)_{i\in I} \in \ell^2(I)$ and if $W = (W_i,
w_i)_{i\in I}$ is a weighted family of closed subspaces of $\Hil$
satisfying
$$w_i \pi_{W_i} = v_i S_V^{-1} \pi_{V_i} + L_i - \Big( \sum_{k\in I} L_k v_k \pi_{V_k} \Big) v_i S_V^{-1} \pi_{V_i} \quad (\forall i\in I),$$
then $W$ is a Bessel fusion sequence for $\Hil$ and it holds $D_W C_V = \mathcal{I}_{\Hil}$.
\end{corollary}

\begin{proof}
By Theorem \ref{normaldualthm}, it suffices to show that our
assumptions imply that $[\dots L_{i-1} \, L_i \, L_{i+1} \dots] \in
\mathcal{B}(\mathcal{K}_{\Hil}^2 , \Hil)$. To this end, observe that
for every $(f_i)_{i\in I} \in \mathcal{K}_{\Hil}^2$, we have
\begin{flalign}
   \big\Vert [\dots L_{i-1} \, L_i \, L_{i+1} \dots] (f_i)_{i\in I} \big\Vert &=  \Big\Vert \sum_{I\in I} L_i f_i \Big\Vert \notag \\
    & \leq \sum_{I\in I} \Vert L_i \Vert \Vert f_i \Vert \notag \\
    &\leq \Big( \sum_{i\in I} \Vert L_i \Vert^2 \Big)^{1/2} \Big( \sum_{i\in I} \Vert f_i \Vert^2 \Big)^{1/2} \notag
\end{flalign}
by the Cauchy-Schwarz inequality.
\end{proof}

\begin{corollary}\label{normaldualcor}
Let $V = (V_i, v_i)_{i\in I}$ be a fusion frame for $\Hil$. If $W =
(W_i, w_i)_{i\in I}$ is a Bessel fusion sequence for $\Hil$ satisfying $D_W
C_V = \mathcal{I}_{\Hil}$, then there exists a family $(L_i)_{i\in
I}$ in $\mathcal{B}(\Hil)$, such that $L = [\dots L_{i-1} \, L_i \,
L_{i+1} \dots] \in \mathcal{B}(\mathcal{K}_{\Hil}^2 , \Hil)$ and
$$W_i \subseteq S_V^{-1} V_i + \mathcal{R}( L_i) + \Big( \sum_{k\in I} L_k v_k \pi_{V_k} \Big) S_V^{-1} V_i \quad (\forall i\in I).$$
\end{corollary}

\begin{proof}
By Theorem \ref{normaldualthm}, there
exists a family $(L_i)_{i\in I}$ of operators  in $\mathcal{B}(\Hil)$, such
that $[\dots L_{i-1} \, L_i \, L_{i+1} \dots] \in
\mathcal{B}(\mathcal{K}_{\Hil}^2 , \Hil)$ and (\ref{entries}) is satisfied. In particular, the ranges of the respective operators in (\ref{entries}) have to
coincide. Thus we obtain
\begin{flalign}
    W_i = \mathcal{R}(w_i \pi_{W_i}) &=  \mathcal{R}\Big(  v_i S_V^{-1} \pi_{V_i} + L_i - \Big( \sum_{k\in I} L_k v_k \pi_{V_k} \Big) v_i S_V^{-1} \pi_{V_i} \Big) \notag \\
    & \subseteq \mathcal{R}(v_i S_V^{-1} \pi_{V_i}) + \mathcal{R}(L_i) + \mathcal{R}\Big(\Big( \sum_{k\in I} L_k v_k \pi_{V_k} \Big) v_i S_V^{-1} \pi_{V_i} \Big) \notag \\
    &= S_V^{-1} V_i + \mathcal{R}(L_i) + \Big( \sum_{k\in I} L_k v_k \pi_{V_k} \Big) S_V^{-1} V_i \notag
\end{flalign}
for all $i\in I$.
\end{proof}

We conclude this subsection with the following generalization of
Proposition \ref{newdualprop}.

\begin{theorem}\label{newdualthm}
Let $V$ and $W$ be fusion frames. Then the following are equivalent.
\begin{enumerate}
    \item[(i)] $D_W C_V = \mathcal{I}_{\Hil}$.
    \item[(ii)] For all choices of orthonormal bases $(e_{ij})_{j\in J_i}$ for $V_i$, the global family $(w_i \pi_{W_i} e_{ij})_{i\in I, j\in J_i}$ is a dual frame of the global frame $ve = (v_i e_{ij})_{i\in I, j\in J_i}$.
    \item[(iii)] For all choices of orthonormal bases $(e_{ij})_{j\in J_i}$ for $V_i$ it holds
    $$w_i \pi_{W_i} e_{ij} = S_V^{-1} v_i e_{ij} + h_{ij} - \sum_{k\in I}\sum_{j\in J_k} \langle S_V^{-1} v_i e_{ij}, v_k e_{kl} \rangle h_{kl} \quad (\forall i\in I, j\in J_i),$$
    where $(h_{ij})_{i\in I, j\in J_i}$ is a Bessel sequence for $\Hil$.
\end{enumerate}
\end{theorem}

\begin{proof}
For each subspace $V_i$, choose an orthonormal basis $(e_{ij})_{j\in J_i}$. Then for the associated global frame $ve =
(v_i e_{ij})_{i\in I, j\in J_i}$ we have $S_V = S_{ve}$ by Lemma
\ref{globalisfusion}. Now, we observe that for all $f\in
\mathcal{H}$ we have
\begin{flalign}
D_W C_V f &= \sum_{i\in I} v_i w_i \pi_{W_i} \pi_{V_i} f \notag \\
&= \sum_{i\in I} v_i w_i \pi_{W_i} \sum_{j\in J_i} \langle f, e_{ij} \rangle e_{ij} \notag \\
&=  \sum_{i\in I} \sum_{j\in J_i} \langle f, v_i e_{ij} \rangle w_i
\pi_{W_i} e_{ij}. \notag
\end{flalign}
This implies that $D_W C_V = \mathcal{I}_{\Hil}$ if and only if
$(w_i \pi_{W_i} e_{ij})_{i\in I, j\in J_i}$ is a dual frame of $ve$.
Thus the equivalence (i) $\Leftrightarrow$ (ii) follows. The
equivalence (ii) $\Leftrightarrow$ (iii) follows immediately from
$S_V^{-1} = S_{ve}^{-1}$ and the characterization of dual frames in
Hilbert spaces (Theorem \ref{dualframechar}).
\end{proof}

\subsubsection{Case $Q=\bigoplus
_{i\in I} S_V^{-1}$} {\label{sec:Gavruta}}

The second special case we are presenting in this section, is the
historically first one for the definition of \emph{alternate dual fusion frames},
done by P. G\v{a}vru\c{t}a in \cite{Gavruta7}.
In his paper, G\v{a}vru\c{t}a
worked with the synthesis operator $D_V$ and analysis operators
$C_W$ as operators $\mathcal{K}_V^2 \longrightarrow \Hil$ and $\Hil
\longrightarrow \mathcal{K}_W^2$ respectively. Since they can not be
composed as in (\ref{dualfusion2}), G\v{a}vru\c{t}a approached
the definition of a dual fusion frame from a different direction. In analogy to frame theory, he generalized
the canonical dual fusion
frame coming from the fusion frame reconstruction (see Section~\ref{Ej dual canonico}). In that process, using Lemma \ref{fusionframeoplem}, he shows
that $S^{-1} _V \pi _{V_i} = \pi _{S^{-1} _V V_i} S^{-1} _V \pi
_{V_i}$ and proves the reconstruction formula
\begin{equation}\label{fusionframerecunstruction2}
f =  \sum_{i\in I} v_i ^2 \pi _{S^{-1} _V V_i} S^{-1} _V \pi _{V_i}
f \quad (\forall f\in \mathcal{H}).
\end{equation}
Moreover, G\v{a}vru\c{t}a generalizes
(\ref{fusionframerecunstruction2}) and calls a Bessel fusion
sequence $W=(W_i, w_i)_{i\in I}$ an \emph{alternate dual fusion frame} of
$V$, if
\begin{equation}\label{alternatedual}
    f =  \sum_{i\in I} v_i w_i \pi _{W_i} S^{-1} _V \pi _{V_i} f \quad (\forall f\in \mathcal{H}).
\end{equation}
As in \cite{Shaarbal}, using our notation for block-diagonal
operators, we may rewrite (\ref{alternatedual}) to (\ref{E D dual
fusion frame}) via the block-diagonal operator $Q=\bigoplus _{i\in I} S_V^{-1}$.

\

While alternate duals are perfectly fitting to include fusion frame
reconstruction in the duality setting, this special case lacks symmetry (see below) and therefore does not contain dual frames as a special case. Again, this points towards the need of a more general $Q$, which makes the dual frame definition flexible enough for all tasks expected from a dual fusion frame.

\begin{example}{Example \cite{rohtua}}\label{S}
Consider the subspaces of $\mathbb{R}^3$
given by
\begin{equation*}
V_{1}= \text{span}\big( (1, 0, 0)^T\big),\ \ V_{2}=\text{span}\big( (0, 1, 0)^T\big),\ \
V_{3}=\text{span}\big( (0, 0, 1)^T\big),
\end{equation*}
\begin{equation*}
V_{4}=\text{span}\big((0, 1, 0)^T\big) ,\ \ V_{5}=\text{span}\big( (1, 0, 0)^T\big),\ \
V_{6}=\text{span}\big( (0, 0, 1)^T \big),
\end{equation*}
\begin{equation*}
W_{1}=\text{span}\big( (1, 0, 0)^T \big),\ \ W_{2}=\text{span}\big((0, 1, 0)^T\big),\ \
W_{3}=\text{span}\big( (0, 0, 1)^T\big),
\end{equation*}
\begin{equation*}
W_{4}=\text{span}\big( (0, 0, 1)^T\big),\ \ W_{5}=\text{span}\big((0, 1, 0)^T\big),\ \
W_{6}=\text{span}\big((1, 0, 0)^T\big),
\end{equation*}
Then both $V=(V_{i}, 1)_{i=1}^6$ and $W=(W_{i}, 2)_{i=1}^6$ are fusion frames for $\mathbb{R}^{3}$ with respective fusion frame operators $S_{V}=2\cdot{\mathcal{I}}_{\mathbb{R}^{3}}$ and
$S_{W}=8\cdot{\mathcal{I}}_{\mathbb{R}^{3}}$. 
A direct computation shows that 
$$D_W \Big( \bigoplus_{i=1}^6 S_V^{-1} \Big) C_V = \mathcal{I}_{\mathbb{R}^3} \qquad \text{and} \qquad D_V \Big( \bigoplus_{i=1}^6 S_W^{-1} \Big) C_W = \frac{1}{4}\cdot \mathcal{I}_{\mathbb{R}^3}.$$
Hence, $W$ is an alternate dual of $V$, but $V$ is not an alternate dual of $W$.
\end{example}

\section{Finite fusion frames}\label{Finite Fusion Frames and Implementation}

In the previous sections we mainly focused on the general theory for fusion frames, i.e. fusion frames in (possibly) infinite-dimensional Hilbert spaces $\Hil$ indexed by an arbitrary countable set $I$. Although all previous results can be easily transferred to the finite-dimensional setting, we will recap a few aspects of the already elaborated theory in the final-dimensional setting, before presenting some results which particularly rely on $\dim (\Hil) < \infty$. 
In Section \ref{Implementations} we provide the reader with the implementation of fusion frames. 

\

Until now, we have nearly always assumed in the proofs that $I$ is countably infinite and ensured convergence whenever necessary, even though $I$ can be finite in case $\dim (\Hil) = \infty$. For instance, any closed subspace $V$ of $\Hil$ and its orthogonal complement $V^{\perp}$ (with constant weights $1$) form a (very boring example of a) fusion frame for $\Hil$ with $\vert I \vert = 2$. Conversely, if $\dim (\Hil) < \infty$, then there exist fusion frames with infinite index sets such as $(\Hil, 2^{-n})_{n=1}^{\infty}$. 

From here on we will only consider the typical case, where $I$ is finite and $\Hil$ finite-dimensional.  

\subsection{Fusion frames in finite-dimensional Hilbert spaces}

We start this section with a short summary on the most basic notions in finite-dimensional fusion frame theory and add some first results towards numerics.

\

Let $\Hil = \Hil^L$ be a $L$-dimensional Hilbert space over $\mathbb{F} = \mathbb{C}$ or $\mathbb{F} = \mathbb{R}$, i.e. we may identify $\Hil^L = \mathbb{F}^L$. 

Let $V=(V_i,v_i)_{i=1}^K$ be a family of subspaces $V_i$ of $\Hil^L$ with weights $v_i >0$.
Then $V$ is always a Bessel fusion sequence by the Cauchy-Schwarz inequality. Moreover, analogously to frames \cite{ole1} we may consider the minimum 
$$A_V := \min \Big\lbrace \sum_{i=1}^K v_i^2 \Vert \pi_{V_i} f \Vert^2, f\in \Hil^L, \Vert f \Vert =1 \Big\rbrace ,$$
which exists, since $f \mapsto \sum_{i=1}^K v_i^2 \Vert \pi_{V_i} f \Vert^2$ is continuous and the unit-sphere in a finite-dimensional Hilbert space is compact. This implies, that 
\begin{flalign}\label{lower}
A_V >0 \qquad &\Longleftrightarrow \qquad \text{span}(V_i)_{i=1}^K = \text{span}\cup_{i=1}^K V_i = \Hil^L \\
&\Longleftrightarrow \qquad V \, \text{is a fusion frame for} \, \, \Hil^L, \notag
\end{flalign}
in which case $A_V$ is a lower fusion frame bound. In particular, in the finite-dimensional setting, the fusion frame property of $V$ is independent of the weights $v_i>0$. Moreover, in view of Theorem \ref{minimalchaR}, a fusion Riesz basis for $\Hil^L$ is simply a fusion frame $V=(V_i,v_i)_{i=1}^{K}$, where $\Hil^L = \oplus_{i=1}^K V_i$ is the direct sum of the subspaces $V_i$. 

In summary, we obtain the following:

\begin{proposition}\label{finite1}
Let $V=(V_i,v_i)_{i=1}^K$ be a family of subspaces $V_i$ of $\Hil^L = \mathbb{F}^L$ and weights $v_i >0$. 
\begin{enumerate}
    \item[(i)] $V$ is a fusion frame for $\Hil^L$ if and only if $\text{span}\cup_{i=1}^K V_i = \Hil^L$.
    \item[(ii)] $V$ is a fusion Riesz basis for $\Hil^L$ if and only if $\text{span}\cup_{i=1}^K V_i = \Hil^L$ and $\sum_{i=1}^K \dim(V_i) = L$.
\end{enumerate}
\end{proposition}

\begin{remark}

\noindent (a) In view of Proposition \ref{finite1}, one might ask, why the fusion frame theory concept in the finite dimensional setting should be considered interesting, as the most important definitions are either trivial or boil down to basic linear algebra. The answer to this question is (at least) twofold: 
At the one hand, the fusion frame inequalities work for the finite-dimensional and the infinite-dimensional setting. Some results cannot be transfer - in particular, (\ref{lower}) is not true if $\Hil$ is infinite-dimensional (in fact, the unit sphere in $\Hil$ is not compact in this case, and $\overline{\text{span}}(V_i)_{i\in I} = \Hil$ is only necessary, but not sufficient for the lower frame inequality to be satisfied). But the homogenous definition gives a way to specialize all the abstract results for fusion frames to the finite dimensional setting. 

On the other hand, the fusion frame bounds $A_V$ and $B_V$ contain numerical information, see Lemmata \ref{num1} and \ref{num2}.   

\noindent (b) Equivalently, we could call a family $(P_i)_{i=1}^K$ of $L\times L-$matrices over $\mathbb{F}$ a fusion frame for $\Hil^L$, if $P_i^2 = P_i$ and $P_i^* = P_i$ for $i=1, \dots , K$ and if the matrix $\sum_{i=1}^K v_i^2 P_i$ is positive definite. Indeed, we have $\pi_{V_i} = P_i$, $V_i= \mathcal{R}(P_i)$,  and $\sum_{i=1}^K v_i^2 P_i = S_V$ corresponds to the fusion frame operator $S_V$.
\end{remark}

For finite tight fusion frames we observe the following relation between the dimension $L$ of $\Hil$ and the respective dimensions of the subspaces $V_i$.

\begin{proposition}\cite{cakuli08}
Let $V=(V_i,v_i)_{i=1}^K$ be an $A_V$-tight fusion frame for $\Hil^L$. Then
$$\sum_{i=1}^K v_i^2 \dim (V_i) = L\cdot A_V.$$
\end{proposition}

\begin{proof}
For every $i$, let $(e_{ij})_{j=1}^{\dim (V_i)}$ be an orthonormal basis for $V_i$. Then the global frame $(v_i e_{ij})_{i,j=1}^{K,  \dim (V_i)}$ is $A_V$-tight frame by Corollary \ref{parsfusion}. Hence, employing \cite[Sec. 2.3]{CasKov03} gives us $$L\cdot A_V = \sum_{i=1}^K \sum_{j=1}^{\dim (V_i)} \Vert v_i e_{ij} \Vert^2 = \sum_{i=1}^K v_i^2 \dim (V_i),$$
as desired.
\end{proof}

The coefficient space associated to $V$ is $\sum_{i=1}^K \oplus \Hil^L=\left\{(f_i)_{i=1}^K:
f_i\in \Hil^L \right\}$, equipped with $\left<(f_i)_{i=1}^K,(g_i)_{i=1}^K\right>=\sum_{i=1}^K\left<f_i,g_i\right>$, which we may identify with $\big(\mathbb{F}^{K L}, \Vert \cdot \Vert_2\big) \cong \big(\mathbb{F}^{L \times K}, \Vert \cdot \Vert_{\text{Fro}}\big)$. For consistency reason we set $\mathcal{K}_{\Hil}^2 := \sum_{i=1}^K \oplus \Hil^L$. 

Hence, the synthesis operator $D_V: \mathcal{K}_{\Hil}^2 \longrightarrow \Hil^L$ associated to $V$ can be identified with an $L\times KL$ -matrix over $\mathbb{F}$ and the analysis operator $C_V: \Hil^L \longrightarrow \mathcal{K}_{\Hil}^2$ can be identified with an $KL\times L$ -matrix over $\mathbb{F}$. In particular $S_V = D_V C_V = \sum_{i=1}^K v_i^2 \pi_{V_i}$ can be identified with an $L\times L$ -matrix over  $\mathbb{F}$.

\

As discussed in Section \ref{Bessel fusion sequences and fusion frames}, given a fusion frame $V=(V_i, v_i)_{i=1}^K$, we can achieve perfect reconstruction by utilizing the invertibility of $S_V$, see (\ref{fusionframereconstruction}). In numerical procedures, an important measure for inverting a positive definite matrix is given by the \emph{condition number}, which is defined as the ratio between the largest and the smallest eigenvalue. Analogously to frames, the condition number of the fusion frame operator is given by the ratio of the optimal (fusion) frame bounds \cite{ole1}, see below.

\begin{lemma}\label{num1}
Let $V=(V_i,v_i)_{i=1}^K$ be a fusion frame for $\Hil^L$. Let $A_V^{opt}$ denote the largest possible lower fusion frame bound and $B_V^{opt}$ the smallest possible upper fusion frame bound for $V$. Then the condition number of $S_V$ is given by $B_V^{opt}/A_V^{opt}$. 
\end{lemma}
 
\begin{proof}
For every $i$, let $(e_{ij})_{j=1}^{\dim (V_i)}$ be an orthonormal basis for $V_i$. Then, by Lemma \ref{globalisfusion}, $ve = (v_i e_{ij})_{i,j=1}^{K, \dim (V_i)}$ is frame for $\Hil^L$ and $S_{ve} = S_V$. In particular their respective optimal frame bounds and condition numbers coincide, thus the claim follows \cite{ole1}.
\end{proof} 

If one prefers to avoid the inversion of $S_V$, one can approximate $f\in \Hil^L$ by performing the \emph{frame algorithm} \cite{ole1} instead, which can be shown completely analogously as in the frame case, since its proof relies solely on the positive definiteness of the (fusion) frame operator. For completeness reason we still present the proof following \cite{ole1}. 
For a concrete implementation see Section \ref{Implementations}.

\begin{lemma}\label{num2}\cite{cakuli08}
Let $V=(V_i,v_i)_{i=1}^K$ be a fusion frame for $\Hil^L$ with fusion frame bounds $A_V \leq B_V$. Given $f\in \Hil^L$, define the sequence $\lbrace g_n \rbrace_{n=0}^{\infty}$ by
$$g_0 = 0, \quad g_n = g_{n-1} + \frac{2}{A_V + B_V} S_V(f-g_{n-1}) \qquad (n\geq 1).$$
Then 
$$\Vert f - g_n \Vert \leq \left( \frac{B_V-A_V}{B_V+A_V} \right)^n \Vert f \Vert .$$
In particular, $\Vert f - g_n \Vert \longrightarrow 0$ as $n\rightarrow \infty$.
\end{lemma}

\begin{proof}
Choose $f\in \Hil^L$. By the lower fusion frame inequality, \begin{flalign}
    \Big\langle \Big( \mathcal{I}_{\Hil} - \frac{2}{A_V + B_V} S_V \Big) f, f   \Big\rangle &= \Vert f \Vert^2 - \frac{2}{A_V + B_V} \sum_{i=1}^K v_i^2 \Vert \pi_{V_i} f \Vert^2 \notag \\
    &\leq \Vert f \Vert^2 - \frac{2A_V}{A_V + B_V} \Vert f \Vert^2 \notag \\
    &=  \frac{B_V - A_V}{B_V + A_V} \Vert f \Vert^2. \notag
\end{flalign}
Since $\mathcal{I}_{\Hil} - 2(A_V + B_V)^{-1}S_V$ is self-adjoint, this implies 
$$\Big\Vert \mathcal{I}_{\Hil} - \frac{2}{A_V + B_V} S_V \Big\Vert \leq \frac{B_V - A_V}{B_V + A_V} .$$
Now, by definition of $\lbrace g_n \rbrace_{n=0}^{\infty}$ we see that
\begin{flalign}
      f - g_n &= f - g_{n-1} - \frac{2}{A_V + B_V} S_V (f - g_{n-1}) \notag \\
      &= \Big( \mathcal{I}_{\Hil} - \frac{2}{A_V + B_V} S_V \Big) (f - g_{n-1}) .\notag
\end{flalign}
By iteration and the assumption $g_0 = 0$ we obtain 
$$f-g_n = \Big( \mathcal{I}_{\Hil} - \frac{2}{A_V + B_V} S_V \Big)^n f.$$
This implies 
$$\Vert f - g_n \Vert \leq \left( \frac{B_V-A_V}{B_V+A_V} \right)^n \Vert f \Vert ,$$
as desired.
\end{proof}

%

\section{Reproducible research implementations} \label{Implementations}

Finite fusion frames - here, $N$ subspaces of $\Hil^L$ -  are implemented as part of the frames framework in the Large Time Frequency Analysis Toolbox (LTFAT)~\cite{soendxxl10}. The frames framework was introduced in \cite{ltfatnote030} and follows an object oriented programming paradigm, with the \verb+frame+ class serving as the parent class from which specific frame objects are derived. Fusion frames are represented as container frames that can comprise arbitrary combinations of frame objects. 

The LTFAT frames framework allows an object oriented approach to transforms \cite{ltfatnote030} - i.e. frame analysis matrices. 
Frame objects can be constructed either in a concrete way,  by providing a synthesis matrix to the \verb+frame+ class using the option \verb_gen_ or 
as an abstract class - initiating the frame objects with specific desired properties, such as a specific frame structure or adaptable time or frequency scales. 
Currently supported abstract, structured frames in LTFAT comprise the discrete Gabor transform (DGT) \cite{ltfatnote011} and the non-stationary Gabor transform (NSDGT) \cite{nsdgt10}, a wide range of filterbank implementations with linear and logarithmic frequency scales \cite{AUDLET16}, and wavelet transforms \cite{ltfatnote038}. For a complete list of the currently supported frame types, we refer to LTFAT's documentation~\cite{ltfatdoc}.

The properties of a frame object, such as e.g. the fusion frame weights \verb+w+, correspond to the object's attributes, while the frame-associated operations, such as analysis and synthesis, align with the frame object's methods. Rather than creating the coefficient matrices directly, a frame object is instantiated in an operator-like fashion, i.e. as a structure comprising mostly anonymous functions. The frame object becomes non-abstract and applicable, when it is fully initialized, and fully defined in a mathematical sense, once its dimensions have been fixed to a specific $\mathbb{C}^L$.  
In Table \ref{table:fultfat} we list the frame properties that are either specific, or particularly relevant, to fusion frames, along with their name within the frame object.

For the creation of a fusion frame, the synthesis matrices of the respective local frames are used for defining them, i.e.
\begin{equation} \label{eq:finsynt1} 
D^{(i)}:=D_{\phi^{(i)}} = \begin{bmatrix}
    \vline & \vline & \, & \vline \\
\phi_{i1} & \phi_{i2} &  \, \dots \, & \phi_{iK_i}  \\ 
\vline & \vline & \, & \vline 
\end{bmatrix}.
\end{equation} 
where $K_{i}$ is the number of frame vectors of $\phi^{(i)}$. Local frames $\phi^{(i)}=(\phi_{ik})_{k\in K_i}$ can be created from a given synthesis matrix via the command \\
 \verb+locfram+${}^{(i)}$ \verb+ = frame('gen',+$D^{(i)}$\verb+);+ \\
 and the fusion frame is then created using \\
  \verb+ F = frame ('fusion', weights, locfram+ ${}^{(1)}, \dots, $ \verb+ locfram+${}^{(N)}$\verb+);+.

The fusion frame attributes 
weights (\verb+w+), the local frame structs (\verb+frames+), their number (\verb+Nframes+), and their analysis and synthesis methods (\verb+frana+ and \verb+frsyn+, respectively) are set during the instantiation of the frame object via the function \verb+frame+. 


\begin{table}[!t]
\caption{Some relevant fields in a fully initialized LTFAT fusion frame object, an indication if they fusion frame specific and a description of their corresponding frame property. 
}
\label{table:fultfat}       
%
%
\begin{tabular}{p{2.5cm}p{2cm}p{6.7cm}}
\hline\noalign{\smallskip}
field name & fusion specific & description \\
\noalign{\smallskip}\svhline\noalign{\smallskip}
\texttt{Nframes} & +   & number of local frames $N$\\
\texttt{w}  & + & fusion frame weights $\left( w_1, \dots, w_N \right)$  \\
\texttt{frames} & +  & the local frame objects $\left( D_1, \dots, D_N \right)$   \\
\texttt{frana} & & method for constructing the analysis operator   \\
\texttt{frsyn} & & method for constructing the synthesis operator   \\
\texttt{L} & & the dimension of the full space, the length of the frame vectors \\
\texttt{M} & + & number of elements/vectors in each local frame: $\left( K_1, \dots, K_N \right)$ \\
\texttt{A} & +  & lower fusion frame bound   \\
\texttt{B} & +  & upper fusion frame bound   \\
\texttt{istight} & +  & if the frame is tight, i.e. $A=B$   \\
\texttt{isparseval} & + & if the frame is Parseval, i.e. $A=B = 1$   \\
\texttt{isuniform} & +  & if the frame is uniform, i.e. $w_i = 1 \forall i$   \\
\texttt{localdual} & +  & the local dual frame object   \\
\texttt{frameoperator} & +  & the frame operator as a matrix  \\
\noalign{\smallskip}\hline\noalign{\smallskip}
\end{tabular}
\end{table}

\

A fusion frame is a class of subspaces  $W_i$ in $\Hil^L$. 
To reduce the dimensions of the presentation spaces and be numerically more efficient one could consider spaces $W_i$ of different dimensions. But then it has to be clarified how to fuse the local frames together, e.g. by assigning an operator from the lower dimensional subspaces $W_i$ to $\Hil^L$ \cite{caku13}[Sec. 13.2.3].

To avoid this additional operator, we consider the subspaces of $\Hil^L$ directly. 
The definition of the synthesis operator $D_W$ by 
$f = \sum_{i=1}^N w_i \pi_{W_i} f_i$, 
allows  us to avoid any checking if the $f_i$ are in the considered subspaces $W_i$, and the representation of the synthesis operator  as a matrix in $\mathbb{C}^{L \times L \cdot N}$ is applicable to the \emph{full} coefficient space $\mathcal{K}_{\Hil}^2 = \mathbb{C}^{L \cdot N}$. 
%
Another advantage of this approach is that a modification of the coefficients within an analysis-synthesis stream - like a multiplier \cite{xxlmult1,Arias2008581} - is possible directly, without having to 
ensure to stay in the relevant subspaces.
Additionally, for the numerical solution of operator equations - represented as block matrix equations using fusion frames - this is much more convenient~\cite{xxlcharshaare18} - for the same reason. 
\

For the analysis operator $C_W: \mathbb{C}^{L} \longrightarrow \mathbb{C}^{L\times N}$ one has to implement an orthogonal projection on the subspace $W_i$. In our implementation, the subspaces $W_i$ are always given by the local frames $\phi^{(i)} = (\phi_{ik})_{k \in K_i}$, i.e. $W_i = \text{span}(\phi_{ik})_{k\in K_i}$ (in other words, we always consider a fusion frame system $\left( W_i, w_i, \phi^{(i)} \right)_{i=1}^N$). 
Those local frames are frame sequences in the full space $\Hil^L$. Following \cite[Prop. 5.2.3] {ole1n} the orthogonal projections is computed via local frame reconstruction using the respective local dual frame vectors $\widetilde{\phi}_{ik}= S_{\phi^{(i)}}^{-1} \cdot \phi_{ik}$ for $i = 1, \dots , N$ and $k \in K_i$, i.e.  
\begin{equation} \label{eq:proj}\pi_{W_i} f = \sum_{k \in K_i} \left< f, \phi_{ik}\right> \widetilde{\phi}_{ik}.
\end{equation}

The analysis operator $C_W f = \vec p = {\left(p_i\right)_{i=1}^N}$ 
is given component-wise by
$$p_i = w_i \pi_{W_i}f$$
with $f$ being the input signal, $w_i$ the fusion frame weights, and $W_i$ the subspaces spanned by the respective local frames. In LTFAT, this is implemented via the 
\verb+fuana+ function. 
If one wants to extract only one projection, i.e. the projection to a fixed subspace $W_{i_0}$, this can be done by using the  option \verb+fuana(..., 'range',+$i_0$\verb+)+.
This is used in the  synthesis operator, implemented via the function \verb+fusyn+ .  One can also specify any vector of subset of $1, \dots, N$.

\

The code used for generating the Figures of this chapter can be downloaded from \url{http://ltfat.org/notes/ltfatnote058.zip}. It will also be included in the upcoming LTFAT version 2.6.

\subsection{Dual fusion frame calculation}\label{sec:dualfucalc1}

In order to compute a canonical dual fusion frame $(S_W^{-1}W_i, w_i)_{i=1}^N$ of a given fusion frame $(W_i, w_i)_{i=1}^N$, we need to know the inverse $S_W^{-1}$ of the fusion frame operator $S_W$. The operator $S_W^{-1}$ can be calculated via direct inversion of the fusion frame operator $S_W$ using the \verb+fudual+ function.  For a general fusion frame, no assumptions on the structure of the coefficient space can be made. Consequentially, the application of efficient factorization algorithms, such as e.g. for Gabor frames~\cite{strohmer1998}, is not possible. 
In general, direct inversion is numerically unfeasible for large $L$ and/or $N$. 
Another option is to use an iterative approach. We have implemented the fusion frame algorithm, see Theorem \ref{num2}. This is a standard Richardson iteration with a matrix of size $L \times L$ that stems from a fusion frame. A block-wise iteration on the coefficient space is also possible \cite{xxlcharshaare18}, but not necessary here.

See Figure~\ref{fig_richardson} for an example the convergence rate for a uniform random fusion frame. The frame was created with dimension $L=105$ and comprises eight local frame sequences - i.e. eight subspaces, $N = 8$. 
The number of elements in the local frames was randomized within the interval $[0.25\cdot L, 0.75 \cdot L]$, along with their coefficients, that were randomized within the range $(0, \dots , 0.5)$. 

%
%
%

\begin{figure}
\begin{center}
\includegraphics[width=0.5\textwidth]
{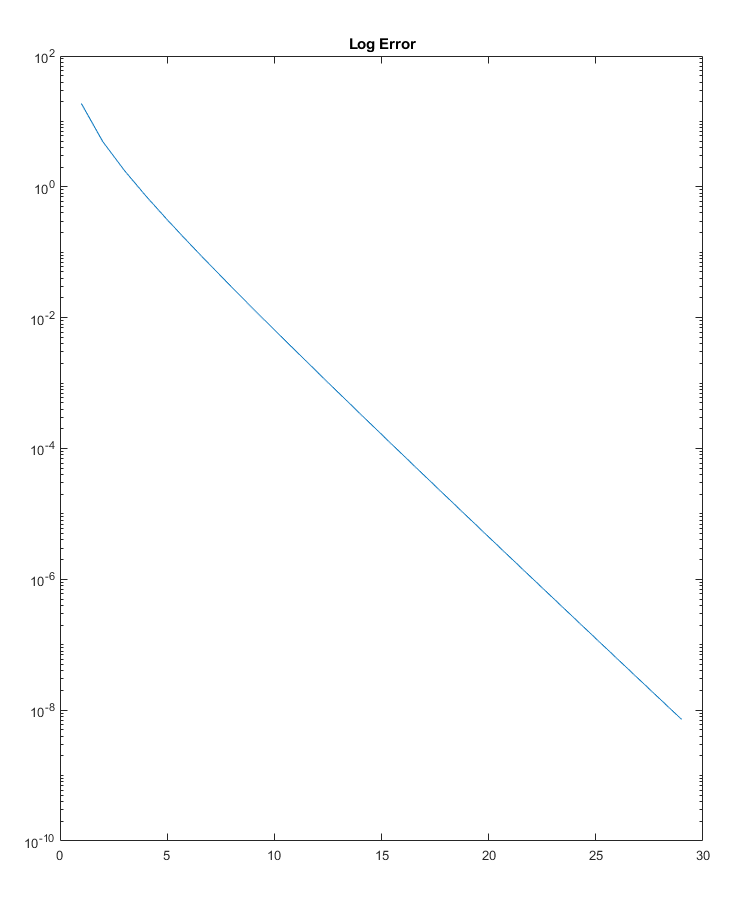}
   
\end{center}
\caption{Logarithmic convergence rate of the Richardson algorithm for the fusion frame operator. 
}\label{fig_richardson}
\end{figure}

As an example that can easily be visualized, consider the fusion frame system defined by the local frame vectors $\phi_{11} = (1,0,0)^T$, $\phi_{12} = (0,1,1/2)^T$, $\phi_{21} = (1,1,0)^T$ and $\phi_{22} = (0,0,1)^T$ and the weights $w_1 = w_2 = 1$. Depicted in Figure~\ref{fig_projection} are the projections of the vector $f = (-0.5, -1.5, 0.5)^T$ onto the two corresponding subspaces, i.e. the fusion frame coefficients. A 3D-plot of the canonical dual frame, along with local frames in $\RR^3$, is shown in Figure~\ref{fig_fusionrec}. 

\begin{figure}
\includegraphics[width=\textwidth]
{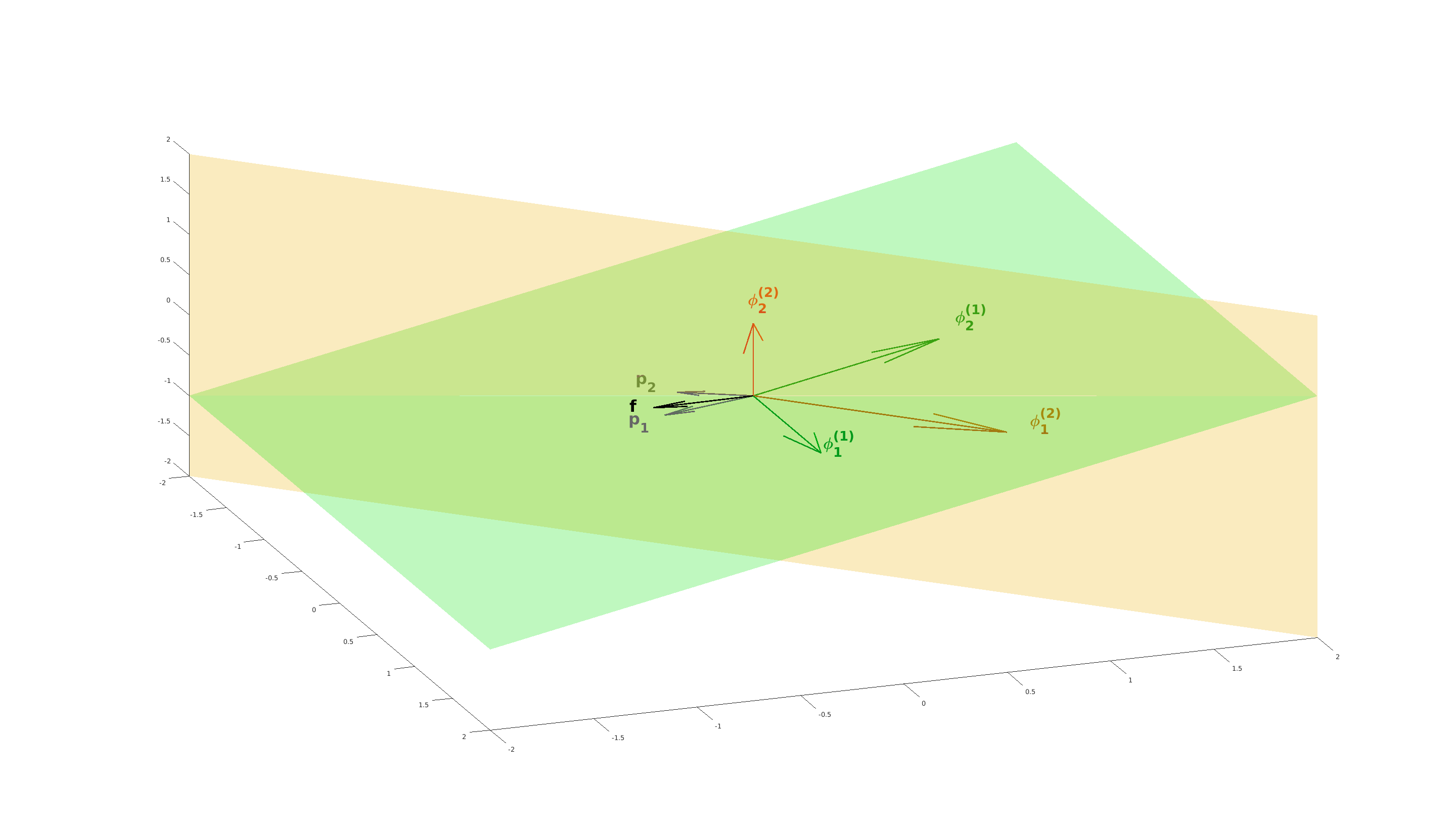}
\caption{Projection of the vector $f = (-0.5, -1.5, 0.5)^T$ on a uniform fusion frame comprising the subspaces $W_1 = \text{span}( \phi_{11}, \phi_{12})$ (in green) and $W_2 = \text{span}(\phi_{21},  \phi_{22})$ (in orange), where $\phi_{11} = (1,0,0)^T$, $\phi_{12} = (0,1,1/2)^T$, $\phi_{21} = (1,1,0)^T$ and $\phi_{22} = (0,0,1)^T$. The local frames are indicated as arrows in their respective colors, orange and green. The vector $f$ is indicated as a black arrow, its respective projections $p_1$ and $p_2$ onto the two corresponding subspaces in grey.
}\label{fig_projection}
\end{figure}

\begin{figure}
\includegraphics[width=\textwidth]{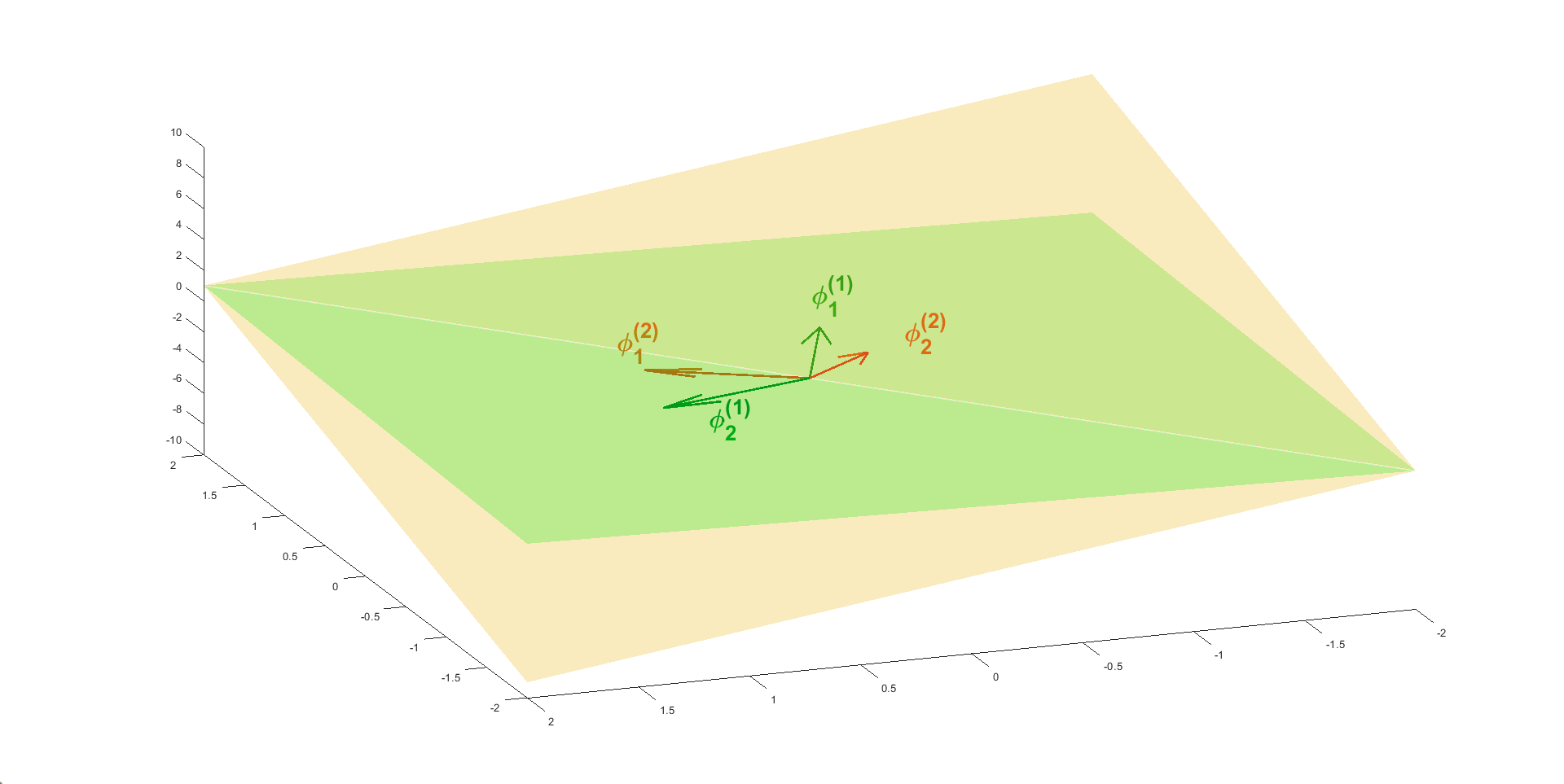}
\caption{The canonical dual fusion frame of the above fusion frame system  - the display shifted in azimuth by 90 degrees.
}\label{fig_fusionrec}
\end{figure}

\subsection{Towards the Visualization of Fusion Frame Coefficients}\label{fusiograms}
It is unclear what a canonical way of displaying the coefficients of a fusion frame would be. Let us present some options in this section, to instigate future discussions. This topic will - certainly - be the subject of future research.

One - straightforward - option would be to display the vectors $p_i = \pi_{W_i} f$ directly.
This is depicted in Figure~\ref{fig_coefsdirect_ana} for the random fusion frame  for $L=105$ introduced above in Section \ref{sec:dualfucalc1}.
The fusion frame analysis coefficients of a sinusoidal input signal $f(t) = \sin(2\pi 4 t)$ with $t=0, \frac{1}{L},\frac{1}{L}, \dots, \frac{1}{L}$ are plotted. This could be called a \emph{direct fusion frame coefficient plot}.

Even for those random frames, the frequency of the sinusoid can be guessed in the picture. 
At the right side - in  Figure~\ref{fig_coefsnorm_ana} - the norms $\Vert p_i \Vert = \Vert \pi_{W_i} f \Vert$ of those coefficients are plotted. This could be seen as an analogue to a spectrogram, so could be called a \emph{fusiogram}. 

\begin{figure}
     \centering
     \begin{subfigure}{\textwidth}
         \centering
         \includegraphics[trim={0.5cm 0.5cm 0.5cm 0.5cm},clip,width=\textwidth]
         {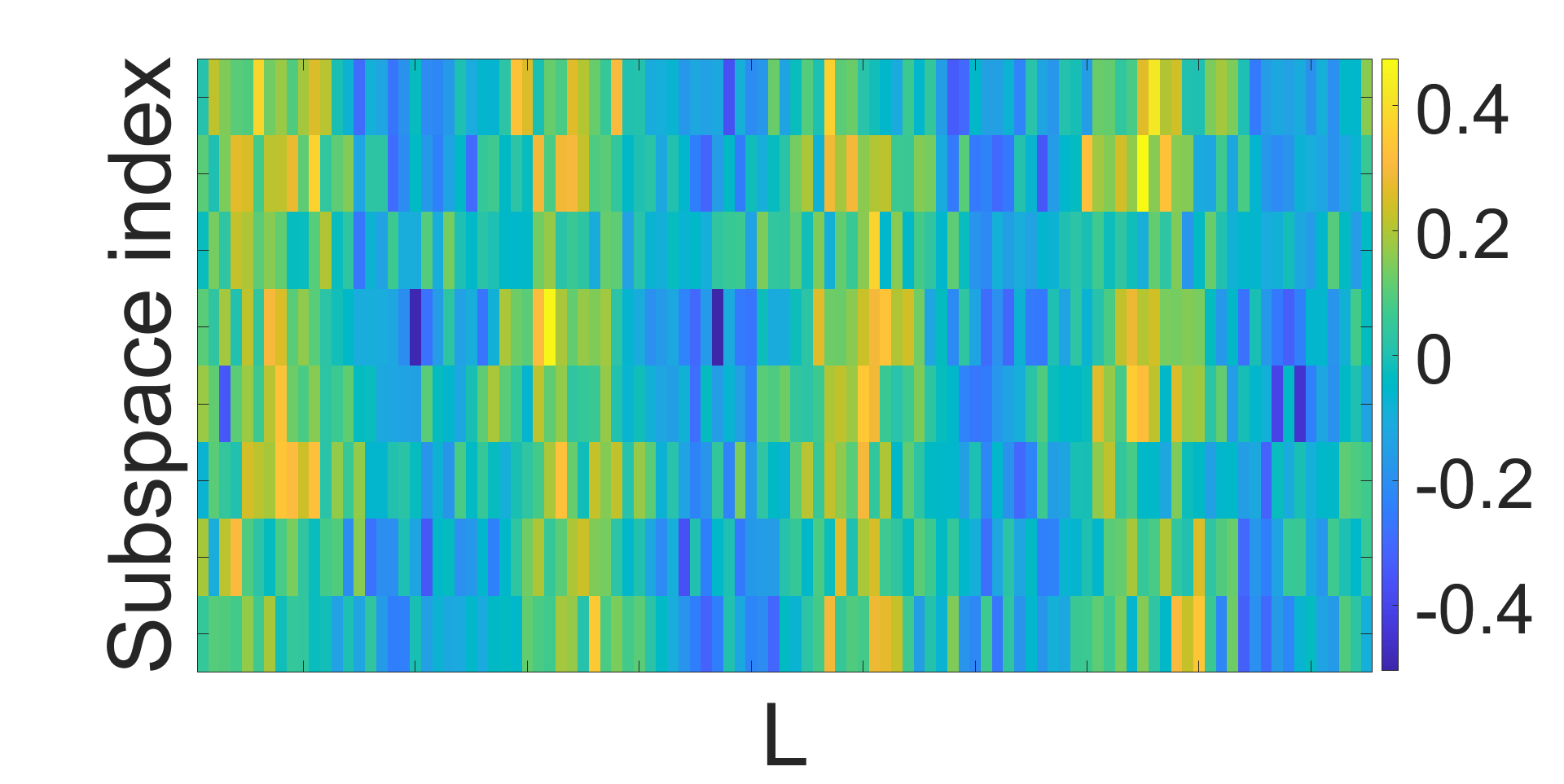}
         \caption{Fusion frame analysis operator $C_W$ of a uniform random frame with $L=105$ and $N=8$ applied on a a four oscillation sinusoid (i.e. $f(t) = \sin(2\pi 4 t)$ with $t=0, \frac{1}{L},\frac{1}{L}, \dots, \frac{1}{L}$).}
         \label{fig_coefsdirect_ana}
     \end{subfigure}
     \hfill
     \begin{subfigure}{\textwidth}
         \centering
         \includegraphics[trim={0.5cm 0.5cm 0.5cm 0.5cm},clip,width=\textwidth]      
            {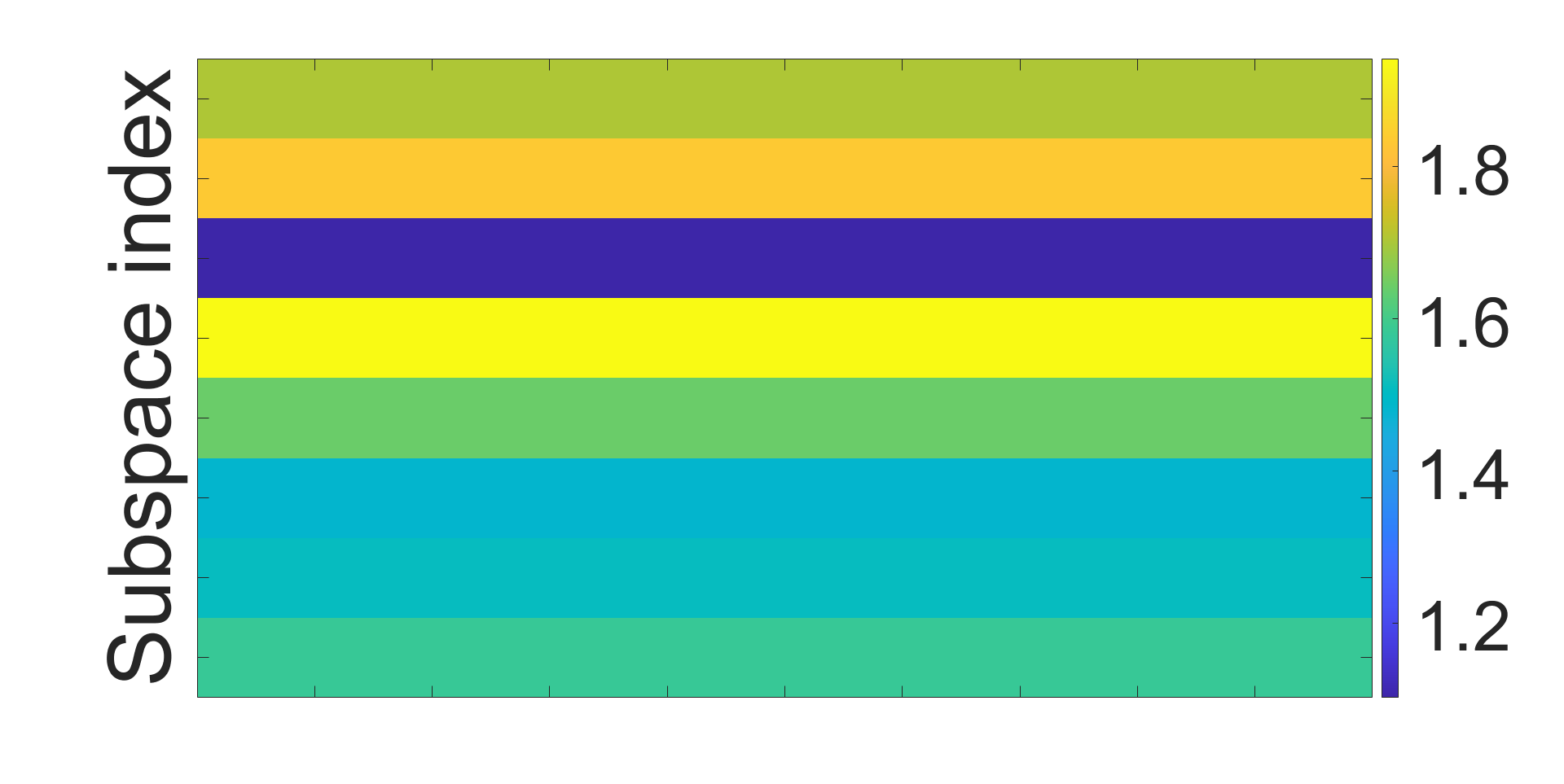}
         \caption{Norm of the fusion frame analysis operator of a uniform random frame with $L=105$ and $N=8$, applied on the same signal.}
         \label{fig_coefsnorm_ana}
     \end{subfigure}
     \end{figure}

Another option is to display the coefficients of the global frame of the fusion frame system, in a way to also indicate the fusion frames, i.e. the respective local frames. This could be called a \emph{local fusiogram}. 
For unstructured frames, such a representation would just be a concatenation of several displays like in Fig. \ref{fig_coefsnorm_ana}. 
This does not seem too interesting or useful. If there is a specific structure shared by all the local frames, and therefore subspaces, taking such an approach seems much more promising.

A prime example for this is the construction of general time-frequency representations as a collection of tiles, arranged on a two dimensional plane, that overlap in both, time and frequency - along the lines of the jigsaw \cite{jaill1} or quilting \cite{do10} idea. With the fusion frame conditions in the background - contrary to other approaches - the reconstruction of such time-frequency representations is easily possible.\\

\subsection{Fusion Frames with a Time-Frequency Relation}

Our approach how we could describe a fusion frame system with 'some' relation to time-frequency is the following:
Using the 'full time-frequency plane', i.e. the  $L \times L$ plane that would result from an STFT ($a = b = 1$) as a starting point, we consider some (time-frequency-related) frame sequences $\psi^{(i)} = \left( \psi_{ik} \right)_{k \in K_i}$. 
We further assume that for all elements $\psi_{ik}$ of the respective local frames $\psi^{(i)}$, we can find corresponding time-frequency points $(t^{i}_k, f^{i}_k) \in [0,..,L-1] \times [0,..,L-1]$. 
We now define 
regions $Q_i$ in the (full) time frequency plane (with some overlap). Finally, by that we find new local frames by
\begin{equation}
    \label{eq:fusiontf1}
\varphi^{(i)} = (\psi_{ik})_{k\in K_i, (t^{i}_k, f^{i}_k)\in Q_i}.
\end{equation}
Now $W_i := \text{span}\lbrace\varphi^{(i)}\rbrace$ and we consider the \emph{'mixed time-frequency transform'} system $(W_i, 1, \varphi^{(i)})_{i=1}^N$. Here we can apply all results shown above for fusion frames. In particular, the reconstruction from the frame coefficients is solved using Lemma \ref{globalisfusionre} - different to the approaches in \cite{jaill1} and \cite{do10}. \\

To illustrate this idea, we exemplify the practical implementation of a specific "local fusiogram" with a clear time-frequency correspondence below.

\subsubsection{Local Fusiogram Proof-of-Concept} 


\



In this example, we consider frame sequences from samples / subsets from two well-known time-frequency representations: the discrete Gabor transform (DGT), and the fast wavelet transform (FWT)~\cite{mallat}. Both transforms have a time-frequency correspondence, and the assignment of the time-frequency points associated to a specific frame element is possible:

\

A Gabor system in $\CC^L$ consists of the elements
$$ g_{m,n} [k] = M_{mb} T_{na} g [k] = e^{\frac{2 \pi i m b}{L}} g[k - na],$$
where we use the standard notation \cite{feistro1}, i.e. $M_{\omega}$ denotes modulation by $\omega$, $T_{x}$ denotes translation by $x$, and $g$ is a given window. 
For the representation as in \eqref{eq:finsynt1} we have that $k = \frac{L}{a} n + m$. And therefore the time-frequency association is 
$$(t^{DGT}_k, f^{DGT}_k) =  \left( \left\lfloor \frac{ka}{L} \right\rfloor , ka\mod L \right)$$
assuming that the window is centered around $0$ in both time and frequency. The wavelet system - following the implementation of Mallat's algorithm \cite{mall1} in \verb+fwt+ is given by 
\begin{align}
     h_{j,n} [k] &= D_{j} T_{n} h [k] = h[2^j \cdot \left(k - n\right)] \qquad  (1 \leq j < J), \notag \\
     h_{J,n} [k] &= D_{j} T_{n} \tilde h [k] = \tilde h[{2^J} \cdot \left( k - n \right)] \notag,
\end{align}
where we use the usual notation from \cite{carmtorr1}, i.e. $D_{j}$ denotes dyadic dilation, $h$ is a given mother wavelet and $\tilde h$ denotes the scaling function. 
For the representation as in \eqref{eq:finsynt1} for a given $k$ let 
$o: = \max \big\lbrace  {\left\lceil \frac{L}{i} \right\rceil }, J \big\rbrace$. Then a corresponding time-frequency association is  
$$(t^{FWT}_k, f^{FWT}_k) =  \left( 2^{\left\lceil \frac{L}{i} \right\rceil } \left( k - \frac{L}{2^{\left\lceil \frac{L}{i} \right\rceil }}-1 \right)  + \frac{1 + 2^{\left\lceil \frac{L}{i} \right\rceil }}{2} , \frac{3 L}{2^{\left\lceil \frac{L}{i} \right\rceil }} \right) \, \text{for} \, {\left\lceil \frac{L}{i} \right\rceil } \le J$$
and
$$(t^{FWT}_k, f^{FWT}_k) =  \left( 2^J \left( k - \frac{L}{2^{J}}-1 \right) + \frac{1 + 2^{J}}{2} , \frac{L}{2^{J}} \right) \, \text{for} \, {\left\lceil \frac{L}{i} \right\rceil } > J,$$
assuming that the window is centered in its support in both time and frequency.

\

As specified above \eqref{eq:fusiontf1} the local frames $\varphi^{(i)}$ are defined by choosing the regions $Q_i$ as $N$ overlapping rectangles on the time-frequency plane and 
selecting 
the corresponding DGT and FWT 
elements. 
 In our example, it comprises a sequence of six DGT and FWT frame segments (three segments of each), alternatingly ordered, and arranged in an overlapping, jigsaw-like fashion. 
Rather than using the whole $[L \times L]$ time-frequency plane, we display only positive frequencies, and thus restrict our considerations to a time-frequency plane of dimension $[\frac{L}{2} \times L]$. The resulting fusiogram is depicted in Figure~\ref{fig:fusio}.The borders of the regions $Q_i$ are indicated as white lines in the Figures \ref{fig:fusio} and \ref{fig:fusionorm}.

For the concrete example we have chosen a DGT with a Gaussian window 
using a hop size $a_{DGT}=16$ for $M_{DGT}=48$ channels, yielding a redundancy $r =\frac{M}{a}=3$. For the FWT, an iteration depth $J=6$, leading to seven frequency channels $M_{WL}$, and a Daubechies wavelet of order eight was used. Both implementations~\cite{ltfatnote011, ltfatnote038} are readily available in LTFAT. 
With this setting we get a fusion frame with frame bounds $A_W=1.7321$ and $B_W=5.6051$. The resulting global frame has the bounds $A=1$ and $B=5.0512$.

\

We display an input signal $f$ of length $L=384$, comprising a sequence of random numbers (with values ranging from $(0, \dots, 1)$) windowed by a Gaussian window, a sinusoid with a frequency of one fourth the maximal frequency, i.e. $\left(\frac{L}{2}\right) / 4 = 48$, and an impulse, a sequence of six times the number $2$, at samples $320$ to $325$. The input signal is, along with the Gaussian window used, depicted in Figure~\ref{fig:input}.

For all time-frequency positions $(t_k^i,f_k^i)$ in the full matrix where there is a corresponding frame element, we set the value to $\left< f , \psi_{i,k} \right> $. If several coefficients correspond to the same position, we use their mean value.
 The rest of the coefficients are filled using the Matlab interpolation \verb+fillmissing+ .
 We plot the absolute values on a logarithmic scale in Fig. \ref{fig:fusio}. 
 In the bottom tiles, the sinusoidal signal component is visible as a red bar. The impulse lies roughly between time samples $250$ and $350$. The higher energy regions in the middle stem from the windowed noise.

In Figure~\ref{fig:fusionorm}, we plot the norms of the fusion frame coefficients (as in Fig. \ref{fig_coefsnorm_ana}) but arranged in the same time-frequency pattern as before. Here, again, the energy distribution of the input signal is reflected, with little energy in the first tile, some energy in the center tiles, again due to the windowed noise, and the sinusoid contributing energy in the bottom tiles. The two rightmost tiles contain higher energy relative to the two leftmost tiles due to the impulse.



\begin{figure}
\includegraphics[trim={4.0cm 0cm 3.5cm 2cm}, clip, width=\textwidth]{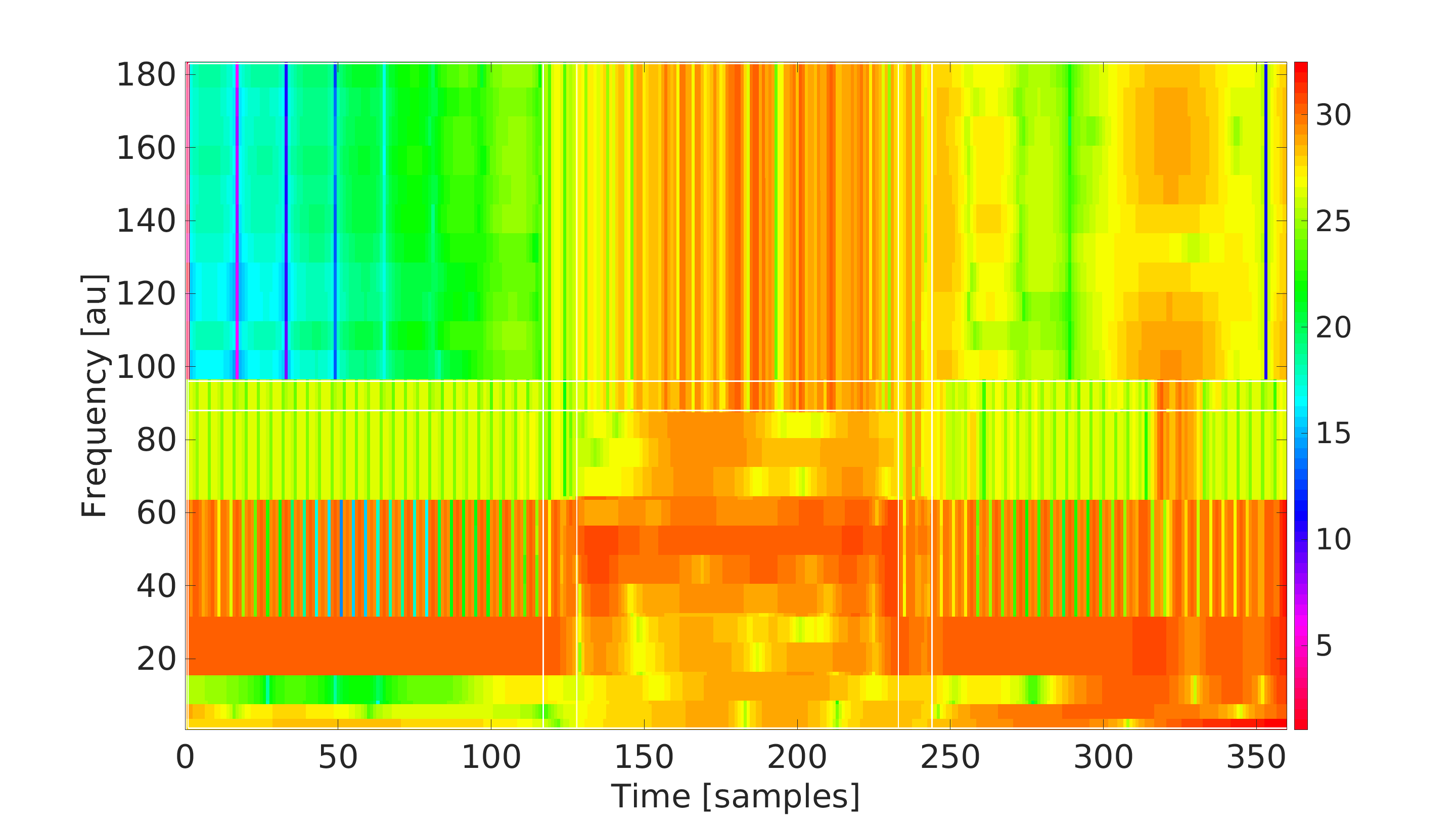}
\caption{The fusiogram, generated from a wavelet and a Gabor frame (from top to bottom and left to right: Gabor, wavelet, wavelet, Gabor, Gabor, wavelet). The representation was retrieved as the absolute values of the logarithm of the interpolated frame coefficients. (For display purposes, to map negative coefficients to a positive range, prior to taking the logarithm, an offset corresponding to the minimum value of the logarithm of the fusion frame coefficients was added.) The overlap, indicated in the Figure by white lines, was chosen to mutually cover 10 \% of the respective frame lengths of neighbouring frames. 
}
\label{fig:fusio}
\end{figure}

\begin{figure}
\includegraphics[trim={4cm 0cm 3.5cm 2cm}, clip, width=\textwidth]{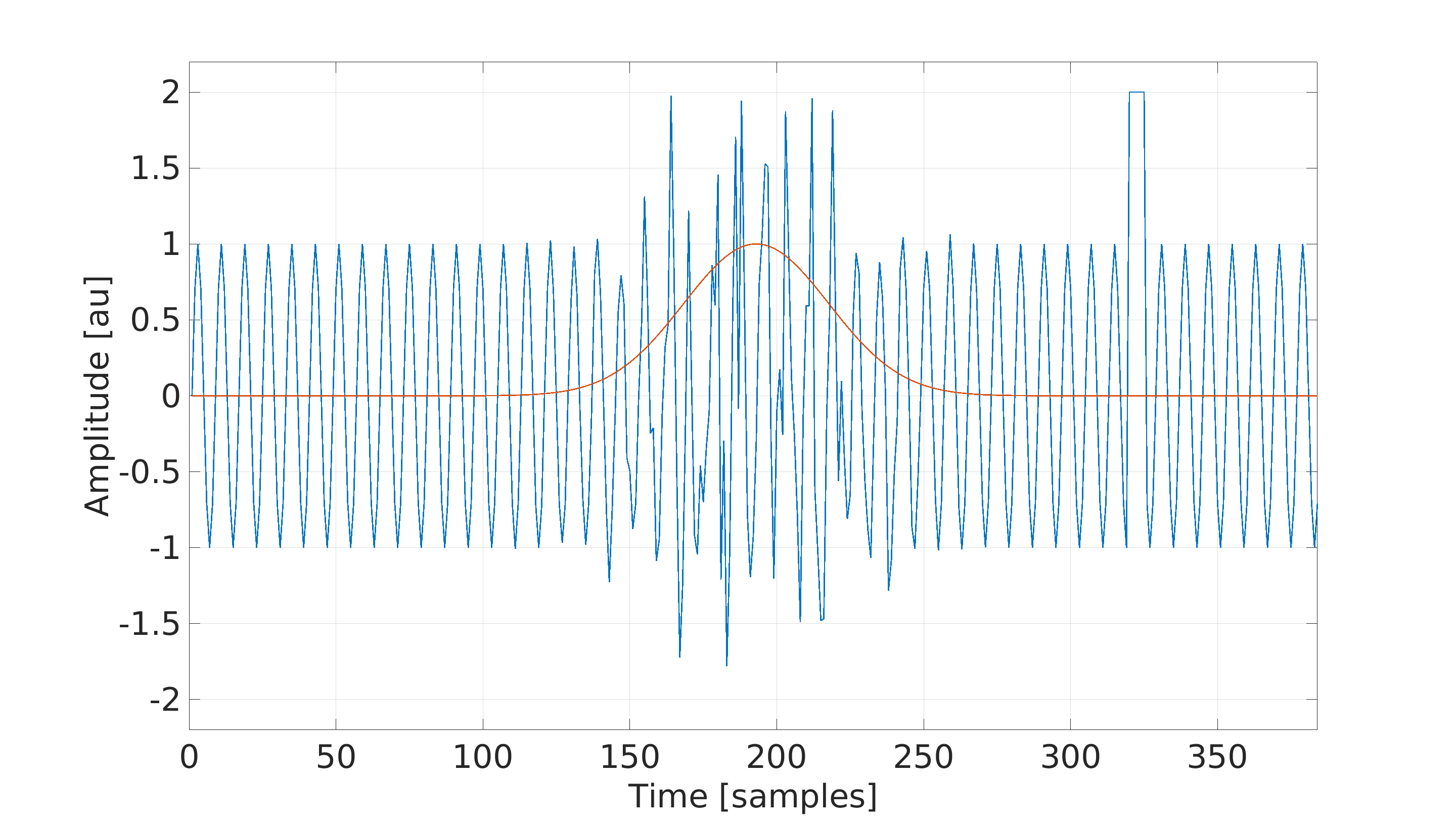}
\caption{The input signal used is the superposition of a windowed sinusoid with $f=48 Hz$, a Gaussian windowed random signal with values ranging from $(0,\dots, 1)$, and an impulse ranging from samples $320$ to $325$. Depicted in red is the Gaussian window used for windowing the noise signal.}\label{fig:input}
\end{figure}


\begin{figure}
\includegraphics[trim={4cm 0cm 3.5cm 2cm}, clip, width=\textwidth]{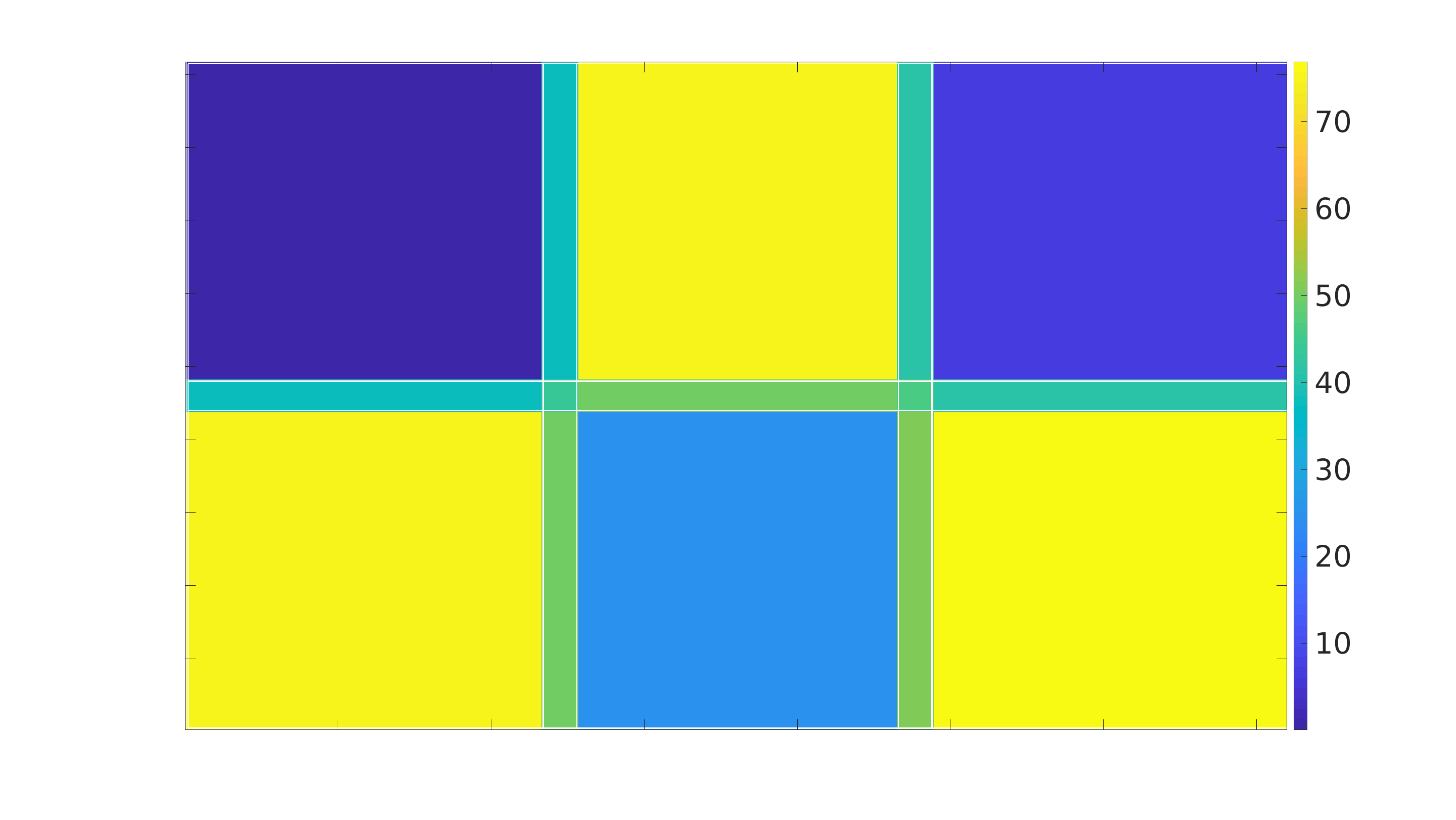}
\caption{ The norms of the fusion frame analysis coefficients are displayed, as in Figure~\ref{fig_coefsnorm_ana} but using the same arrangement as in Figure \ref{fig:fusio}, where the value is averaged in the overlapping domain}\label{fig:fusionorm}
\end{figure}

\section{Further study}\label{Further Study}

This chapter is designed as an introduction to the fundamentals of fusion frame theory  with a focus on the motivation given in the introduction. 

As such, it does not - and could not - contain many sub-areas of the theory,
such as: (1) Phase retrieval and norm retrieval in fusion frames \cite{CAR, BCE,BBCE,CAR,CT,CW}, (2) Fusion frames and the Restricted Isometry Property \cite{BCC}, 
(3) Fusion frames and
unbiased basic sequences \cite{BCPT}, (4) Sparse signal recovery with fusion frames \cite{BKR}, (5)  Fusion frame constructions
\cite{CCH,CFMWZ}, (6) Fusion frame potentials \cite{CF, HLlM2018, MasseyRuizStojanoff10}, (7) Spectral tetris fusion frame constructions \cite{CFHWZ,CP,CW2}, (8) Fusion frames and distributed processing 
\cite{cakuli08,koeba23}, (9) Sensor networks and fusion frames \cite{CKLR}, (10) Random fusion frames \cite{CCEL}, (11) Subspace erasures
for fusion frames \cite{CK2,hemo14,Mo17}, 
(12) Matrix representation of operators using fusion frames \cite{xxlcharshaare18,Shaarbal},
(13) Non-orthogonal fusion frames \cite{CCL}, (14) Fusion frames and combinatorics \cite{BLR}, (15) K-fusion frames \cite{Shaaresad}, (16) Weaving fusion frames \cite{Deepshikha2017OnWF}, (17) Fusion frame multipliers \cite{mitra}, (18) Weaving K-fusion frames \cite{sala20},
(19) Equichordal tight fusion frames \cite{M}, (20) p-fusion frames \cite{baeh13} and (21) Perturbations of fusion frames \cite{RC}. Finally, (22) \cite{Camp} is the final word on (algorithm-free) fusion frame constructions. The author constructs almost all fusion frames, shows that the construction is unique, and gives an exact formula for each fusion frame. For (23) Various applications of fusion frames we refer to the introduction.

\section*{Acknowledgments}

The authors thank N. Holighaus and A. A. Arefijamaal for related discussions and an informal reviewing of this long chapter.

The work of B.P. and Sh.M. was supported by the OeAW Innovation grant FUn ({\em Frames and Unbounded Operators}; IF\_2019\_24\_Fun). The work of K.L., B.P. and Sh.M. was supported by he project P 34624 {\em "Localized, Fusion and Tensors of Frames"} (LoFT) of the Austrian Science Fund (FWF). The work of H.S. and M.P. was supported by Grants PIP 112-201501-00589-CO
(CONICET) and PROIPRO 03-1420 (UNSL).

\bibliographystyle{abbrv}
\bibliography{Main}
\end{document}